\documentclass[a4paper,11pt]{article}
\usepackage{LatexDefinitions}
\usepackage{color, booktabs}

\usepackage{esint}

\floatstyle{ruled}
\newfloat{algorithmfloat}{t}{lop}
\floatname{algorithmfloat}{Algorithm}

\newcommand{\proj}{\tn{P}}
\newcommand{\cont}{\mc{C}}

\newcommand{\partA}{\mc{J}_A}
\newcommand{\partB}{\mc{J}_B}

\newcommand{\partJ}{\mc{J}}
\newcommand{\partBasic}{I}

\newcommand{\partAn}{\mc{J}_A^{n}}
\newcommand{\partBn}{\mc{J}_B^{n}}
\newcommand{\partBasicn}{I^n}
\newcommand{\partk}{\iter{\partJ}{k}}
\newcommand{\partnk}{\iter{\partJ}{n,k}}
\newcommand{\partnkk}{\iter{\partJ}{n,k+1}}
\newcommand{\piInit}{\pi_{\tn{init}}}
\newcommand{\piInitN}{\pi_{\tn{init}}^n}
\newcommand{\neigh}{\mc{N}}

\newcommand{\floor}[1]{\lfloor #1 \rfloor}
\newcommand{\bmpi}{\bm{\pi}}
\newcommand{\bmomega}{\bm{\omega}}
\newcommand{\bmnu}{\bm{\lambda}}
\newcommand{\vecnu}{\lambda}
\newcommand{\projTX}{\proj_{(\R_+ \times X)}}

\newcommand{\projTime}{\proj_{\R_+}}
\newcommand{\Lebesgue}{\mathcal{L}}

\DeclareMathOperator{\sign}{sign}

\newcommand{\WTV}{\tn{WTV}}
\newcommand{\WTVb}{\tn{WTVB}}
\newcommand{\WoY}{W_Y}
\newcommand{\xJn}{\ol{x}_{t,x}^n}
\newcommand{\boundaryIn}{\partial I^n}
\newcommand{\interiorIn}{\mathring{I}^n}
\newcommand{\pivot}{\mathfrak{p}}

\newcommand{\bitset}{B}
\newcommand{\nuset}{\mc{V}}
\newcommand{\nset}{\mc{Z}}
\newcommand{\nsetliminf}{\nset'}
\newcommand{\nsetb}{\hat{\nset}}
\newcommand{\inner}[2]{\langle {#1}, {#2}\rangle}

\newcommand{\MuDens}{\RadNikD{\mu}{\Lebesgue}}
\newcommand{\muDens}{\RadNik{\mu}{\Lebesgue}}
\newcommand{\muLow}{M_l}
\newcommand{\muHi}{M_u}
\newcommand{\WY}{\mc{W}}
\newcommand{\WoR}{W_{\R}}
\newcommand{\len}{L}
\newcommand{\measleq}{\preceq}

\newcommand{\BoundarySet}{\mathfrak{B}}
\newcommand{\IntWeight}{\mathcal{I}}

\tikzstyle{result} = [rectangle, minimum width=3cm, minimum height=1cm, text centered, draw=black, align = center]
\tikzstyle{mainresult} = [rectangle, rounded corners, minimum width=3cm, minimum height=1cm,text centered, draw=black, align = center]
\tikzstyle{arrow} = [thick,->,>=stealth]

\title{Asymptotic analysis of domain decomposition for optimal transport}
\author{Mauro Bonafini, Ismael Medina, Bernhard Schmitzer}
\date{\today}

\begin{document}
\maketitle
\begin{abstract}
Large optimal transport problems can be approached via domain decomposition, i.e.~by iteratively solving small partial problems independently and in parallel.
Convergence to the global minimizers under suitable assumptions has been shown in the unregularized and entropy regularized setting and its computational efficiency has been demonstrated experimentally. An accurate theoretical understanding of its convergence speed in geometric settings is still lacking. In this article we work towards such an understanding by deriving, via $\Gamma$-convergence, an asymptotic description of the algorithm in the limit of infinitely fine partition cells.
The limit trajectory of couplings is described by a continuity equation on the product space where the momentum is purely horizontal and driven by the gradient of the cost function. Convergence hinges on a regularity assumption that we investigate in detail.
Global optimality of the limit trajectories remains an interesting open problem, even when global optimality is established at finite scales.
Our result provides insights about the efficiency of the domain decomposition algorithm at finite resolutions and in combination with coarse-to-fine schemes.
\end{abstract}

\section{Introduction}
\subsection{Overview}
\paragraph{(Computational) optimal transport.} Optimal transport (OT) is an ubiquitous optimization problem with applications in various branches of mathematics, including stochastics, PDE analysis and geometry.
Let $\mu$ and $\nu$ be probability measures over spaces $X$ and $Y$ and let $\Pi(\mu,\nu)$ be the set of transport plans, i.e.~probability measures on $X \times Y$ with $\mu$ and $\nu$ as first and second marginal.
Further, let $c : X \times Y \to \R$ be a \emph{cost function}.
The Kantorovich formulation of optimal transport is then given by
\begin{align}
	\label{eq:IntroOT}
	\inf \left\{ \int_{X \times Y} c(x,y)\,\diff \pi(x,y) \middle| \pi \in \Pi(\mu,\nu) \right\}.
\end{align}
We refer to the monographs \cite{Villani-OptimalTransport-09} and \cite{SantambrogioOT} for a thorough introduction and historical context.
Due to its geometric intuition and robustness it is becoming particularly popular in data analysis and machine learning.
Therefore, the development of efficient numerical methods is of immense importance, and considerable progress was made in recent years, such as solvers for the Monge--Amp\`ere equation \cite{ObermanMongeAmpere2014}, semi-discrete methods \cite{LevySemiDiscrete2015,KiMeThi2019}, entropic regularization \cite{Cuturi2013}, and multi-scale methods \cite{MultiscaleTransport2011,SchmitzerSchnoerr-SSVM2013}.
An introduction to computational optimal transport, an overview on available efficient algorithms, and applications can be found in \cite{PeyreCuturiCompOT}.

\begin{figure}[hbt]
	\centering
	\includegraphics[width=0.9\linewidth]{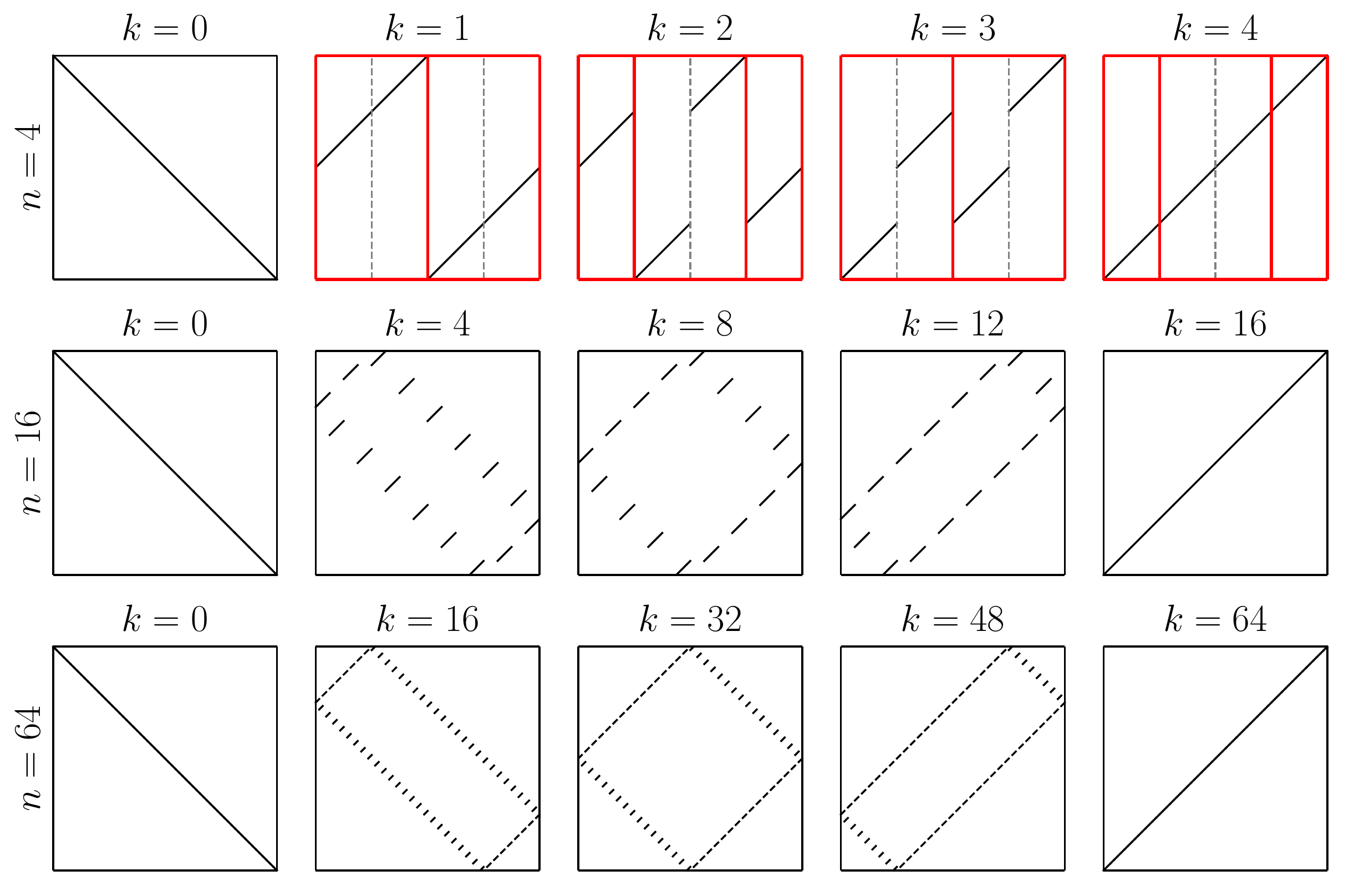}
	\caption{Iterations of the domain decomposition algorithm for $X=Y=[0,1]$, $\mu = \nu = \Lebesgue\restr [0,1]$, $c(x,y)=(x-y)^2$ and a ``flipped'' initialization (the diagonal plan is optimal, we start with the flipped `anti-diagonal'), for several resolution levels.
	At each level $X$ is divided into $n$ equal cells, which are then grouped into two staggered partitions $\partA$ and $\partB$. Partition cells of $\partA$ and $\partB$ are shown in red in the first row. As the number of cells $n$ is increased, the trajectories of the algorithm seem to converge to an asymptotic limit where each iteration corresponds to a time-step of size $1/n$.
	}
	\label{fig:introflipped}
\end{figure}
\paragraph{Domain decomposition.}
Benamou introduced a domain decomposition algorithm for Was\-ser\-stein-2 optimal transport on $\R^d$ \cite{BenamouPolarDomainDecomposition1994}, based on Brenier's polar factorization \cite{MonotoneRerrangement-91}. The case of entropic transport was studied in \cite{BoSch2020}.
The algorithm works as follows: $X$ is divided into two `staggered' partitions $\{X_J | J \in \partA\}$ and $\{X_{\hat{J}} | \hat{J} \in \partB\}$. In the first iteration, an initial coupling $\pi^0$ is optimized separately on the cells $X_J \times Y$ for $J \in \partA$, yielding $\pi^1$. Then $\pi^1$ is optimized separately on the cells $X_{\hat{J}} \times Y$ for $\hat{J} \in \partB$, yielding $\pi^2$. Subsequently, one continues alternating optimizing on the two partitions.
This is illustrated in the first row of Figure \ref{fig:introflipped}.
In each iteration the problems on the individual cells can be solved in parallel, thus making the algorithm amenable for large-scale parallelization.

In \cite{BenamouPolarDomainDecomposition1994} it was shown that the algorithm converges to the global minimizer of \eqref{eq:IntroOT} for $X, Y$ being bounded subsets of $\R^d$, $\mu$ being Lebesgue-absolutely continuous and $c(x,y)=\|x-y\|^2$, if each partition contains two cells that satisfy a `convex overlap principle' which roughly requires that a function $f : X \to \R$ which is convex on each of the cells of $\partA$ and $\partB$ must be convex on $X$. The extension to more complex partitions was discussed and clearly works on certain `non-cyclic' partitions, but no proof was given for beyond this case. Also, no rate of convergence was given.

In \cite{BoSch2020} convergence to the minimizer for the entropic setting was shown under rather mild conditions: $c$ needs to be bounded and the two partitions need to be `connected', indicating roughly that it is possible to traverse $X$ by jumping between overlapping cells of $\partA$ and $\partB$. Convergence was shown to be linear in the Kullback--Leibler (KL) divergence.
In addition, an efficient numerical implementation with various features such as parallelization, coarse-to-fine optimization, adaptive sparse truncation and gradual reduction of the regularization parameter was introduced and its favourable performance was demonstrated on numerical examples.
The convergence mechanism used in the proof is based on the entropic smoothing and the obtained convergence rate is exponentially slow as regularization goes to zero. It was shown to be approximately accurate on carefully designed worst-case problems. On problems with more geometric structure, such as the matching of image intensities with the quadratic cost, the algorithm empirically converged much faster. In combination with the coarse-to-fine scheme even a logarithmic number of iterations (in the image pixel number) was sufficient. The main mechanism for driving convergence seemed to be the geometric structure of the cells and the cost function. This was not reflected in the convergence analysis. The relation of the one-dimensional case to the odd-even-transposition sort was discussed, but the argument does not extend to higher dimensions.

\paragraph{Asymptotic dynamic of the Sinkhorn algorithm.}
The celebrated Sinkhorn algorithm has advanced to an ubiquitous numerical method for optimal transport by means of entropic regularization \cite{Cuturi2013,PeyreCuturiCompOT}. Linear convergence of the algorithm in Hilbert's projective metric is established in \cite{FranklinLorenz-Scaling-1989}. As in \cite{BoSch2020} the convergence analysis of \cite{FranklinLorenz-Scaling-1989} is solely based on the entropic smoothing and the convergence rate tends to 1 exponentially as regularization decreases (in fact, the former article was inspired by the latter).

In was observed numerically (e.g.~\cite{SchmitzerScaling2019}) that on problems with sufficient geometric structure the Sinkhorn algorithm tends to converge much faster, in particular with appropriate auxiliary techniques such as coarse-to-fine optimization and gradual reduction of the regularization parameter.

In \cite{Berman-Sinkhorn-2017} the asymptotic dynamic of the Sinkhorn algorithm for the squared distance cost on the torus is studied in the joint limit of decreasing regularization and refined discretization. The dynamic is fully characterized by the evolution of the dual variables (corresponding to the scaling factors in the Sinkhorn algorithm) which were shown to converge towards the solution of a parabolic PDE of Monge--Amp\`ere type. This PDE had already been studied by \cite{KitagawaFlow2012,KiStWa2012} and thus allowed estimates on the required number of iterations of the Sinkhorn algorithm for convergence within a given accuracy. This bound is much more accurate on geometric problems, providing a theoretical explanation for the efficiency of numerical methods.

\subsection{Contribution and outline}
\paragraph{Motivation.}
Empirically domain decomposition was demonstrated to be a robust and efficient numerical method for optimal transport, amenable for large-scale parallelization.
For the entropic setting a linear convergence rate has been derived based on the entropic smoothing.
On sufficiently `geometric' problems it appears to converge much faster but this mechanism is not yet understood theoretically.

In this article we work towards such an understanding.
At the level of a finite partition resolution this seems daunting. We therefore aim at giving an asymptotic description of the algorithm as the number of partition cells tends to $\infty$. The conjecture for the existence of such a limit behaviour is motivated by Figure \ref{fig:introflipped} (and additional illustrations throughout the article).
Intuitively, we therefore seek to provide for the domain decomposition algorithm an equivalent of what is provided by \cite{Berman-Sinkhorn-2017} for the Sinkhorn algorithm.

\paragraph{Preview of the main result.}
For simplicity we consider the case $X=[0,1]^d$, $Y \subset \R^d$ compact and $c \in \cont^1(X \times Y)$ with partition cells of $X$ being staggered regular $d$-dimensional cubes (see Figure \ref{fig:introflipped} for an illustration in one dimension). At discretization scale $n$, during iteration $k$, the domain decomposition algorithm applied to \eqref{eq:IntroOT} requires the solution of the cell problem
\begin{align}
	\label{eq:IntroCellProblem}
	\inf\left\{ \int_{X_J^n \times Y} c\,\diff \pi + \veps^n \cdot \KL(\pi|\mu \otimes \nu)
	\,\middle|\, \pi \in \Pi(\mu^n_J,\nu^{n,k}_J) \right\}.
\end{align}
Here, $X^n_J$ is a cell of the relevant partition $\partAn$ or $\partBn$ (depending on $k$), $\veps^n$ is the entropic regularization parameter at scale $n$ (we consider the cases $\veps^n=0$ and $\veps^n>0$, and in the latter case a dependency on $n$ will turn out to be essential), $\mu^n_J$ is the restriction of $\mu$ to $X^n_J$ and $\nu^{n,k}_J$ is the $Y$-marginal of the previous iterate $\pi^{n,k-1}$, restricted to $X^n_J \times Y$.

For each $n$ this generates a sequence of iterates $(\pi^{n,k})_k$, which we interpret as time-continuous piecewise constant trajectories $\R_+ \ni t \mapsto \bmpi^n_t \assign \pi^{n,\lfloor n \cdot k\rfloor}$. That is, at scale $n$, one iteration corresponds to a time-step $1/n$.

Our main result will be that, under suitable conditions, the sequence of trajectories $(t \mapsto \bmpi^n_t)_n$ converges (up to subsequences) to a limit trajectory $t \mapsto \bmpi_t$. The convergence is uniform on compact time intervals with respect to a metric $\WY$ on $\Pi(\mu,\nu)$ which is stronger than weak* convergence and which almost implies pointwise weak* convergence of the disintegrations of $\bmpi_t^n$ along $X$.
In addition, there will be a momentum field $\R_+ \ni t \mapsto \bmomega_t \in \meas(X \times Y)^d$ such that $\bmpi_t$ and $\bmomega_t$ solve a `horizontal' continuity equation on $X \times Y$,
\begin{align}
	\label{eq:IntroCE}
	\partial_t \bmpi_t + \ddiv_X \bmomega_t = 0
\end{align}
for $t \geq 0$ with an initial-time boundary condition, in a distributional sense. Here $\ddiv_X$ is the divergence of vector fields on $X \times Y$ that only have a `horizontal' component along $X$. We find that $\bmomega_t \ll \bmpi_t$ and the velocity $v_t \assign \RadNik{\bmomega_t}{\bmpi_t}$ has entries bounded by $1$, i.e.~mass moves at most with unit speed along each spatial axis, corresponding to the fact that particles can at most move by one cell per iteration.

The momentum field $\bmomega_t$ in turn is generated from a family of measures $(\bmnu_{t,x})_{x \in X}$ which are minimizers of
\begin{align}
	\label{eq:IntroCellProblemLimit}
	\inf \left\{
	\int_{Z \times Y}
	\inner{\nabla_X c(x, y)}{z} \, \diff \vecnu(z,y)
	+
	\eta \cdot \KL(\vecnu | \sigma\otimes \bmpi_{t,x})
	\middle|
	\vecnu\in\Pi(\sigma, \bmpi_{t,x})
	\right\}.
\end{align}
which can be shown to be the $\Gamma$-limit of problem \eqref{eq:IntroCellProblem}, which can be anticipated by a careful comparison of the two problems:
$Z=[-1,1]^d$ represents the asymptotic infinitesimal partition cell $X^n_J$ (blown up by a factor $n$), we find that the transport cost is linearly expanded in $X$-direction by the gradient, $\eta \assign \lim_{n \to \infty} \veps^n \cdot n$ is the asymptotic entropic contribution (which we assume to be finite for now, but the case $\eta=\infty$ is also discussed), $\sigma$ is the asymptotic infinitesimal restriction of $\mu$ to the partition cells (which may be the Lebesgue measure on $Z$, but we may also obtain different measures if $\mu$ is discretized) and $\bmpi_{t,x}$ is the disintegration of $\bmpi_t$ with respect to the $X$ marginal at $x$, which corresponds to the asymptotic infinitesimal $Y$-marginal of $\bmpi_t$, when restricted to the `point-like' cell at $x$.
It is this pointwise $\Gamma$-convergence that requires a particular notion of convergence of the trajectories $(\bmpi^n_t)_n$.

In one dimension, the disintegration $\bmomega_{t,x}$ of $\bmomega_t$ is obtained from $\bmnu_{t,x}$ via
\begin{align}
	\label{eq:IntroOmega}
	\bmomega_{t,x}(A)
	\assign
	\bmnu_{t,x}(\{z>0\} \times A)
	-
	\bmnu_{t,x}(\{z<0\} \times A)
\end{align}
for measurable $A \subset Y$. That is, particles sitting in the left half of the cell ($z<0$) move left with velocity $-1$, particles in the right half move right with velocity $+1$.

In a nutshell, the limit of the trajectories generated by the domain decomposition algorithm is described by a flow field which is generated by a limit version of the algorithm.

The necessary convergence of $\bmpi^n_t$ to $\bmpi_t$ in the metric $\WY$ hinges on a regularity assumption on the discrete iterates, which intuitively implies that the disintegrations of $\bmpi^n_t$ against $X$ are of bounded variation in a suitable sense, related to total variation of metric space valued functions \cite{AmbrosioMetricBVFunctions1990}. We cannot establish validity of the assumption in the general case. A proof for a simple one-dimensional setting is given (Section \ref{sec:OscillationsBound}). We conjecture that it holds in the majority of cases, but we also provide a potential numerical counter-example for a rather pathological setting (Section \ref{sec:WTVExample}).

Compared to a single Sinkhorn algorithm as in \cite{Berman-Sinkhorn-2017}, the state of the domain decomposition algorithm cannot be described by a scalar potential $X \to \R$, but requires the full (generally non-deterministic) coupling $\pi^{n,k}$. Consequently, the limit system \eqref{eq:IntroCE} - \eqref{eq:IntroOmega} is not a `relatively simple' PDE for a scalar function but formally a non-local PDE for a measure. This system has not been studied previously. Consequently, after having established the convergence to this system, we cannot use existing results to conclude our convergence analysis. Instead we are left with a variety of open questions, mostly concerning the behaviour of the limit system.

However, from our result we can already deduce that as we increase $n$, the number of iterations required to approximate the asymptotic stationary state of the algorithm (which may not necessarily be a global minimizer) increases linearly in $n$, which is much faster than the exponential bound in \cite{BoSch2020}.

\paragraph{Generalized proof of Benamou's convergence result.}
At finite discretization scales $n$, in this article we consider a decomposition of the domain into two staggered grids of cubes. While this is natural from a numerical point of view, see \cite{BoSch2020}, Benamou's original convergence proof does not cover this setting, even for $\mu \ll \Lebesgue$ and $c(x,y)=\|x-y\|^2$, because the arguments for the existence of a continuous, and subsequently convex, global Kantorovich potential do not apply.
Therefore, in Appendix \ref{sec:Benamou} we give a generalization of Benamou's convergence proof at finite discretization scales that covers our setting.

\paragraph{Outline.}
Notation and necessary background on optimal transport and domain decomposition are recalled in Section \ref{sec:Background}.
The detailed setting for the algorithm is introduced in Section \ref{sec:DomDecNotation}, the discrete trajectories are defined in Section \ref{sec:DiscreteTrajectories} (which includes a smoothing step that we have omitted in the above preview).
Once all preliminaries have been introduced, a more detailed preview of the subsequent sections is gathered in Section \ref{sec:Preview}.
Convergence of the trajectories $\bmpi^n_t$ to the limit $\bmpi_t$ is studied in Section \ref{sec:Oscillations}. Convergence of the cell problems, the continuity equation and the complete statement of the main result are given in Section \ref{sec:Convergence}.
Some numerical examples that illustrate extreme cases of the method are given in Section \ref{sec:Numerics}.
The paper ends with a conclusive discussion and open questions in Section \ref{sec:Conclusion}.
Several proofs are delegated to the Appendices.

\section{Background}
\label{sec:Background}
\subsection{Notation and setting}
\label{sec:Notation}
\begin{itemize}
	\item Let $X = [0,1]^d$, $Y$ be a compact subset of $\R^d$. We assume compactness to avoid overly technical arguments while covering the numerically relevant setting. We conjecture that the results of this paper can be generalized to compact $X$ with Lipschitz boundaries.

	\item For a metric space $Z$ denote by $\meas(Z)$ the $\sigma$-finite measures over $Z$. If $Z$ is compact, then measures in $\meas(Z)$ are finite. Further, denote by $\measp(Z)$ the subset of non-negative $\sigma$-finite measures and by $\prob(Z)$ the subset of probability measures.
	\item The Lebesgue measure of any dimension is denoted by $\Lebesgue$. The dimension will be clear from context.
	\item For $\rho\in\meas_+(Z)$ and a measurable $S\subset Z$ we denote by $\rho \restr S$ the restriction of $\rho$ to $S$.

	\item The maps $\proj_X : \measp(X \times Y) \to \measp(X)$ and $\proj_Y : \measp(X \times Y) \to \measp(Y)$ denote the projections of measures on $X \times Y$ to their marginals, i.e.
	\begin{align*}
	(\proj_X \pi)(S_X) \assign \pi(S_X \times Y) \qquad \tn{and} \qquad (\proj_Y \pi)(S_Y) \assign \pi(X \times S_Y)
	\end{align*}
	for $\pi \in \measp(X \times Y)$, $S_X \subset X$, $S_Y \subset Y$ measurable.
	We will use the projection notation analogously for other product spaces.

	\item For a compact metric space $Z$ and $\mu \in \meas(Z)$, $\nu \in \measp(Z)$ the \emph{Kullback--Leibler divergence} (or relative entropy) of $\mu$ with respect to $\nu$ is given by
	\begin{align*}
	\KL(\mu|\nu) \assign \begin{cases}
	\int_Z \varphi\left(\RadNik{\mu}{\nu}\right)\,\diff \nu & \tn{if } \mu \ll \nu,\,\mu \geq 0, \\
	+ \infty & \tn{else,}
	\end{cases}
	\;\; \tn{with} \;\;
	\varphi(s) \assign \begin{cases}
	s\,\log(s)-s+1 & \tn{if } s>0, \\
	1 & \tn{if } s=0, \\
	+ \infty & \tn{else.}
	\end{cases}
	\end{align*}
	\item The total variation of a function $u \in L^1(\R^d)$ is given by
	\begin{equation}
		\TV(u)
		\assign
		\sup \left\{
		\int_{\R^d} u \operatorname{div} \phi\, \mathrm{d} x\
		\middle|\
		\phi \in \cont_{c}^{1}(\R^d),\ |\phi| \leq 1
		\right\}.
	\end{equation}
	If the total variation of $u$ is finite, we say that $u$ is of bounded variation.
\end{itemize}

\subsection{Optimal transport}
\label{sec:BackgroundOT}
For ${\mu} \in \measp(X)$, ${\nu} \in \measp(Y)$ with the same mass, denote by
\begin{align}
	\label{eq:TransportPlans}
	\Pi({\mu},{\nu}) & \assign
	\left\{ \pi \in \measp(X \times Y) \,\middle|\, \proj_X \pi={\mu},\proj_Y \pi = {\nu}\right\}
\end{align}
the set of \emph{transport plans} between ${\mu}$ and ${\nu}$. Note that $\Pi(\mu,\nu)$ is non-empty if and only if $\mu(X)=\nu(Y)$.

Let $c \in \cont(X \times Y)$ and $\veps \in \R_+$. Pick $\hat{\mu} \in \measp(X)$ and $\hat{\nu} \in \measp(Y)$ such that $\mu \ll \hat{\mu}$ and $\nu \ll \hat{\nu}$.
The (entropic) optimal transport problem between $\mu$ and $\nu$ with respect to the cost function $c$, with regularization strength $\veps$ and with respect to the reference measure $\hat{\mu} \otimes \hat{\nu}$ is given by
\begin{align}
	\label{eq:OT}
	\inf \left\{ \int_{X \times Y} c(x,y)\,\diff \pi(x,y) + \veps\,\KL(\pi|\hat{\mu} \otimes \hat{\nu})
	\middle| \pi \in \Pi(\mu,\nu) \right\}.
\end{align}

For $\veps=0$ this is the (unregularized) Kantorovich optimal transport problem. The existence of minimizers follows from standard compactness and lower-semicontinuity arguments. Of course, more general cost functions (e.g.~lower-semicontinuous) can be considered. We refer, for instance, to \cite{Villani-OptimalTransport-09,SantambrogioOT} for in-depth introductions of unregularized optimal transport. Common motivations for choosing $\veps>0$ are the availablility of efficient numerical methods and increased robustness in machine learning applications, see \cite{PeyreCuturiCompOT} for a broader discussion of entropic regularization.
In this article, the above setting is entirely sufficient.

For a compact metric space $(Z,d)$ we denote by $W_Z$ the Wasserstein-1 metric on $\prob(Z)$ (or more generally, subsets of $\measp(Z)$ with a prescribed mass). By the Kantorovich--Rubinstein duality \cite[Remark 6.5]{Villani-OptimalTransport-09} one has for $\mu, \nu \in \measp(Z)$, $\mu(Z)=\nu(Z)$,
\begin{align}
	\label{eq:KantRubin}
	W_Z(\mu,\nu) & = \sup_{\phi \in \Lip_1(Z)} \int_Z \phi\,\diff (\mu-\nu)
\end{align}
where $\Lip_1(Z) \subset \cont(Z)$ denotes the Lipschitz continuous functions over $Z$ with Lipschitz constant at most $1$.

\subsection{Domain decomposition for optimal transport}
\label{sec:BackgroundDomDec}
Domain decomposition for solving the optimal transport problem \eqref{eq:OT} was proposed in \cite{BenamouPolarDomainDecomposition1994} and studied in \cite{BoSch2020} for the case of entropic transport. We briefly recall the main definitions.

\begin{definition}[Basic and composite partitions \protect{\cite[Definition 3.1]{BoSch2020}}]
	\label{def:Partitions}
	A partition of $X$ into measurable sets $\{X_i\}_{i \in \partBasic}$, for some finite index set $\partBasic$, is called a basic partition of $(X,\mu)$ if the measures $\mu_i \assign \mu \restr X_i$ for $i \in \partBasic$ satisfy $\mu_i(X_i)>0$.
	By construction one has $\sum_{i \in \partBasic} \mu_i = \mu$.
	Often we will refer to a basic partition merely by the index set $I$.

	For a basic partition $\{X_i\}_{i \in \partBasic}$ of $(X,\mu)$ a composite partition $\partJ$ is a partition of $\partBasic$.
	For $J \in \partJ$ we will use the following notation:
	\begin{align*}
	X_J & \assign \bigcup_{i \in J} X_i, &
	\mu_J & \assign \sum_{i \in J} \mu_i=\mu \restr X_J.
	\end{align*}
	Of course, the family $\{X_J\}_{J \in \partJ}$ is a measurable partition of $X$ and the families $\{X_i \times Y\}_{i \in \partBasic}$ and $\{X_J \times Y\}_{J \in \partJ}$ are measurable partitions of $X \times Y$.
\end{definition}

Throughout the article (at each discretization scale $n \in 2\N$) we will use one basic partition $\partBasic$ and two corresponding composite partitions $\partA$ and $\partB$. The precise choices of partitions will be given in Section \ref{sec:DomDec}.

The domain decomposition algorithm now works as follows: Starting with a feasible plan $\piInit \in \Pi(\mu,\nu)$, one optimizes the coupling within each cell $X_J \times Y$ for $J \in \partA$ separately, while keeping the marginals on that cell fixed.
This can be done independently and in parallel for each cell and the coupling attains a potentially better score while remaining in $\Pi(\mu,\nu)$ \cite[Proposition 3.3]{BoSch2020}.
Then repeat this step on partition $\partB$, then again on $\partA$ and so on, continuing to alternate between the two partitions. A formal statement of the algorithm is given in Algorithm \ref{alg:DomDec}.

\begin{algorithmfloat}[bt]
	\noindent
	\textbf{Input}: initial coupling $\piInit \in \Pi(\mu,\nu)$

	\noindent
	\textbf{Output}: a sequence $(\iter{\pi}{k})_{k}$ of feasible couplings in $\Pi(\mu,\nu)$
	\smallskip

	\begin{algorithmic}[1]
		\State $\iter{\pi}{0} \leftarrow \piInit$
		\State $k \leftarrow 0$
		\Loop
		\State $k \leftarrow k+1$
		\State \algorithmicif\ ($k$ is odd)\ \algorithmicthen\ $\partk \leftarrow \partA$\ \algorithmicelse\ $\partk \leftarrow \partB$
		\Comment{select the partition}
		\ForAll{$J \in \partk$}
		\Comment{iterate over each composite cell}
		\State $\iter{\nu_J}{k} \leftarrow \proj_Y(\iter{\pi}{k-1} \restr(X_{J} \times Y))$
		\Comment{compute $Y$-marginal on cell}
		\label{line:YMarginal}
		\State $\iter{\pi_{J}}{k} \leftarrow \arg\min \left\{
		\int_{X_J \times Y} c\,\diff \pi + \veps\,\KL(\pi|\mu_J \otimes \nu) \,\middle|\, \pi \in \Pi\left(\mu_J,\iter{\nu_J}{k}\right)\right\}$ 
		\label{line:CellProblem}
		\EndFor
		\State $\iter{\pi}{k} \leftarrow \sum_{J \in \partk} \iter{\pi}{k}_{J}$
		\EndLoop
	\end{algorithmic}
	\caption{Domain decomposition for optimal transport \protect{\cite[Algorithm 1]{BoSch2020}}}
	\label{alg:DomDec}
\end{algorithmfloat}

The case $\veps=0$ and $c(x,y)=\|x-y\|^2$ was studied in \cite{BenamouPolarDomainDecomposition1994} and it was shown that the sequence $(\iter{\pi}{k})_k$ converges to the unique minimizer of \eqref{eq:OT} when $\mu \ll \Lebesgue$ and the partitions $\partA$ and $\partB$ satisfy a particular convex overlapping condition. See Appendix \ref{sec:Benamou} for a generalization to a larger set of practically relevant decompositions.

The case $\veps>0$ and $c$ bounded was studied in \cite{BoSch2020} and convergence to the global minimizer of \eqref{eq:OT} was shown when the partitions $\partA$ and $\partB$ satisfy a particular (weaker) connectedness condition.

\section{Asymptotic analysis of domain decomposition}
\label{sec:DomDec}
\subsection{Problem setup and more notation}
\label{sec:DomDecNotation}

In this article, we are then concerned with applying the domain decomposition problem to (discretizations) of the (possibly entropy regularized) optimal transport problem \eqref{eq:OT},
\begin{equation}
	\inf \left\{ \int_{X \times Y} c(x,y)\,\diff \pi(x,y) + \veps\,\KL(\pi|\mu \otimes \nu)
	\middle| \pi \in \Pi(\mu,\nu) \right\},
\end{equation}
with increasingly finer cells and to study its asymptotic behaviour as the cell size tends to zero.

In the following we outline the adopted setting and corresponding notation that we require for the subsequent analysis.
\begin{enumerate}
	\item $\mu \in \prob(X)$, $\nu \in \prob(Y)$ with $\mu \ll \Lebesgue$. For some results, we will also require a further regularity on $\mu$.
\end{enumerate}
	\begin{assumption}\label{assumption:RegularityMu}
		$\muDens$ is bounded from below and above by two constants $\muLow$, $\muHi$ with $0<\muLow \leq \muDens(x) \leq \muHi < \infty$ for all $x \in X$, and $\muDens$ is of bounded variation, i.e.~$\TV(\muDens) < \infty$.
	\end{assumption}
\begin{enumerate}[resume]
	\item For a \emph{discretization level} $n \in 2\N$, that we assume \emph{even} for simplicity, the index set $\partBasicn$ for the \emph{basic partition} is given by a uniform Cartesian grid with $n$ points along each axis,
	\begin{align*}
		\partBasicn & \assign \Big\{(i_1,\ldots,i_d) \,\Big| \, i_1,\ldots,i_d \in \{0,\ldots,n-1\} \Big\},
		\intertext{and we set the corresponding basic cells as}
		X^n_i & \assign i/n + [0,1/n]^d \qquad \tn{for} \qquad i \in \partBasicn.
	\end{align*}
\end{enumerate}
\begin{remark}
	\label{rem:BasicCellOverlap}
	Note that strictly speaking the set $(X^n_i)_{i \in \partBasicn}$ of closed hypercubes does not form a partition of $X$ as they are not all pairwise disjoint, since adjacent sets contain their common boundary. However, due to the assumption $\mu \ll \Lebesgue$, these overlaps do not carry any mass (and neither do the overlaps between any $X_i \times Y$ with respect to any $\pi \in \Pi(\mu,\nu)$) and hence we could simply assign the boundary regions to any one of the adjacent sets without changes to the algorithm.
	Equivalently, we can just keep the $(X^n_i)_i$ as closed cubes, which is a bit simpler.
\end{remark}

\begin{enumerate}[resume]
	\item The \emph{mass} and the \emph{center} of basic cell $i \in \partBasicn$ are given by
	\begin{align*}
	     m_i^n & \assign \mu(X_i^n), &
	       x_i^n & \assign n^d \int_{X_i^n} x\, \diff x = (i+(\tfrac12,\ldots,\tfrac12))/n.
	\end{align*}

	\item At level $n$ we approximate the original marginal $\mu$ by $\mu^n\in \meas_1(X)$. For example, $\mu^n$ could be a discretization of $\mu$. We assume that $\mu^n(X_i^n) = m_i^n$ for each basic cell $i\in I^n$,
	which in particular implies that $(\mu^n)_n$ converges weak* to $\mu$. We also assume that $\mu^n$ assigns no mass to any basic cell boundary, so Remark \ref{rem:BasicCellOverlap} remains applicable. Further regularity conditions on the sequence $(\mu^n)_n$ will be required in Definition \ref{def:regular_discretization}. In accordance with Definition \ref{def:Partitions} we set
	\begin{align*}
		\mu^n_i & \assign \mu^n \restr X^n_i.
	\end{align*}
	\label{item:MuN}
	
	\item Analogously, let $(\nu^n)_{n}$ be a sequence in $\prob(Y)$, converging weak* to $\nu$, and $(\piInit^n)_{n}$ a sequence in $\prob(X \times Y)$ with $\piInit^n \in \Pi(\mu^n,\nu^n)$, converging weak* to some $\piInit \in \Pi(\mu,\nu)$.
	Again, $\nu^n$ can slightly differ from $\nu$ to allow for potential discretization or approximation steps.
	There are various ways how a corresponding sequence $\piInit^n$ could be generated, for instance via an adaptation of the \emph{block approximation} \cite{Carlier-EntropyJKO-2015} from some $\piInit \in \Pi(\mu,\nu)$.
	\item The cells of the composite partition $\partA$ are generated by forming groups of $2^d$ adjacent basic cells; the cells of $\partB$ are generated analogously, but with an offset of 1 basic cell in every direction. (Of course, composite cells may contain less basic cells at the boundaries of $X$). As in Definition \ref{def:Partitions} we set
	\begin{align*}
	X^n_J & \assign \bigcup_{i \in J} X^n_i, &
	\mu^n_J & \assign \sum_{i \in J} \mu^n_i=\mu^n \restr X_J,
	\end{align*}
	for $J \in \partAn$ or $J \in \partBn$. Again, Remark \ref{rem:BasicCellOverlap} remains applicable.

	\item The \emph{mass} of a composite cell $J$ is $m_J^n \assign \sum_{i\in J}m_i^n$. For an $A$ (resp. $B$) composite cell $J$, we will define its center $x_J^n$ as the unique point on the regular grid
	\begin{equation}
		\left\{\frac{1}{n}, \frac{3}{n},...,\frac{n-1}{n} \right\}^d
		,\qquad
		\left(\text{respectively}	\left\{0, \frac{2}{n},...,\frac{n-2}{n}, 1 \right\}^d \right)
		,
	\end{equation}
	that is contained in $X_J^n$. For $A$ composite cells and $B$ composite cells that do not lie at the boundary of $X$, it coincides with the average of the centers of their basic cells, cf. Figure \ref{fig:basicandcompositecells}.

	\item Two distinct composite cells $J\in \partAn$ and $\hat{J}\in \partBn$ are said to be \emph{neighboring} if they share a basic cell. The set of neighbouring composite cells for a given composite cell $J$ is denoted by $\neigh(J)$.
	By construction, the shared basic cell is unique, and we denote it by $i(J,\hat{J})$. For compactness, instead of writing, for instance, $m^n_{i(J,\hat{J})}$, we often merely write $m^n_{J,\hat{J}}$.
	\item Two composite cells $J, \hat{J}\in \partAn$ or $J, \hat{J}\in \partBn$ are \emph{adjacent} if the sets  $X_J^n$ and $X_{\hat{J}}^n$ share a boundary.
	\label{item:NotationAdjacent}
	\begin{figure}[bt]
		\centering
		\includegraphics[width=\linewidth]{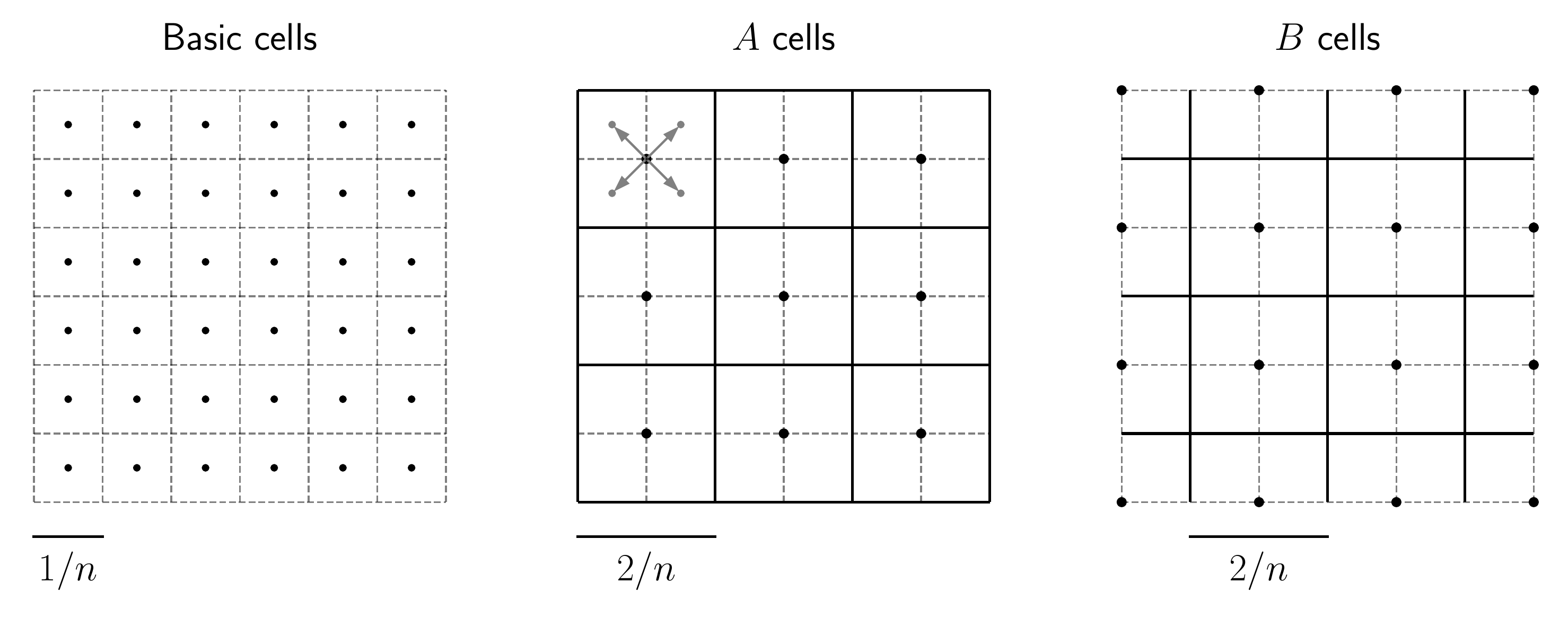}
		\caption{Close-up of the decomposition of $X = [0,1]^2$ into basic and composite cells. Left, basic cells and their centers. Center and right, respectively $A$ and $B$ composite cells and their centers. The top left $A$ composite cell shows that the basic cell center of cell $j\in J$ is of the form $x_J^n + b/2n$, for some $b\in \{-1,1\}^d$.}
		\label{fig:basicandcompositecells}
	\end{figure}
	\item The \emph{transport cost function} $c$ is some $\cont^1$ function on $X\times Y$.
	At scale $n$ it is approximated by a function $c^n$, to account for discretization steps or grid artifacts. We assume that there exists a sequence of real valued functions $(f^n)_n$ on $X\times Y$ uniformly converging to zero such that
	\begin{equation}\label{eq:assumption_modified_cost}
		c^n(x,y) = c(x,y) + \frac{f^n(x,y)}{n}
	\end{equation}
	for all $(x,y)\in X\times Y$.
	\label{item:ModifiedCost}

	\item The entropic \emph{regularization parameter} at level $n$ is $\veps^n \in [0,\infty)$.
	While \cite{BenamouPolarDomainDecomposition1994} was restricted to $\veps=0$ and \cite{BoSch2020} to $\veps>0$, in this article we consider both cases.
	We assume that the sequence $(n \cdot \veps^n)_n$ is convergent as $n \to \infty$ and set
	\begin{align}
		\eta \assign \lim_{n \to \infty} n \cdot \veps^n,
		\label{eq:EpsEta}
	\end{align}
	where we explicitly allow for the case $\eta=\infty$.
	\label{item:EpsEta}
\end{enumerate}
Now, for given $n \in 2\N$, we apply Algorithm \ref{alg:DomDec} to the discrete problem
\begin{align}
	\inf \left\{ \int_{X \times Y} c^n(x,y)\,\diff \pi(x,y) + \veps^n\,\KL(\pi|\mu^n \otimes \nu^n)
	\middle| \pi \in \Pi(\mu^n,\nu^n) \right\},
\end{align}
where we use the basic partition $\partBasicn$, composite partitions $\partAn$ and $\partBn$ and the initial coupling $\piInit^n$.

\begin{enumerate}[resume]
	\item At scale $n$, the $k$-th \emph{iterate} will be denoted by $\iter{\pi}{n,k}$ for $k \in \N$ and one has $\iter{\pi}{n,k} \in \Pi(\mu^n, \nu^n)$. The composite partition used during that iteration will be denoted by $\partnk$, which is either $\partAn$ or $\partBn$, depending on whether $k$ is odd or even.
	\item Based on line \ref{line:YMarginal} of Algorithm \ref{alg:DomDec} we introduce the partial $Y$-marginals of the iterates $\iter{\pi}{n,k}$ when restricted to basic or composite cells:
	\begin{align*}
		\iter{\nu}{n,k}_i & \assign \proj_Y \big(\iter{\pi}{n,k} \restr (X_i \times Y) \big)
		\quad \tn{for} \quad i \in \partBasicn, \\
		\iter{\nu}{n,k}_J & \assign \sum_{i \in J} \iter{\nu}{n,k}_i = \proj_Y \big(\iter{\pi}{n,k} \restr (X_i \times Y) \big) \quad \tn{for} \quad J \in \partAn \tn{ or } \partBn.
	\end{align*}
	It will also be convenient to introduce the normalized versions of $\iter{\nu}{n,k}_i$ and $\iter{\nu}{n,k}_J$,
	\begin{align*}
		\iter{\rho}{n,k}_i & \assign \iter{\nu}{n,k}_i/m_i^n, &
		\iter{\rho}{n,k}_J & \assign \iter{\nu}{n,k}_J/m_J^n.
	\end{align*}
	From Algorithm \ref{alg:DomDec}, lines \ref{line:YMarginal} and \ref{line:CellProblem} we conclude
	\begin{align}
		\label{eq:CompCellMarginalPreservation}
		\nu^{n,k+1}_J = \proj_Y \pi^{n,k+1}_J = \nu^{n,k}_J
	\end{align}
	for $J\in \partnkk$.
	\label{item:NotationPartialMarginals}
\end{enumerate}

\subsection{Discrete trajectories and momenta}\label{sec:DiscreteTrajectories}
	At discretization level $n \in 2\N$, we now associate one iteration with a time-step of $\Delta t = 1/n$, i.e.~iterate $k$ is associated with the time $t = k/n$. Loosely speaking, we now want to consider the family of trajectories $(\R_+ \ni t \mapsto \iter{\pi}{n,\floor{t \cdot n}})_{n \in 2\N}$ and then study their asymptotic behaviour as $n \to \infty$.
	However, we find that the measures $\iter{\pi}{n,k}$ can oscillate very strongly at the level of basic cells, allowing at best for weak* convergence to some limit coupling, cf. Figure \ref{fig:example_iterations}, top.
	In contrast, our hypothesis for the dynamics of the limit trajectory requires a stronger `fiber-wise' convergence of the disintegration against the $X$-marginal.
	We observe that the oscillations in the couplings become much weaker if we average the couplings $\iter{\pi}{n,k}$ at the level of composite cells first, cf. Figure \ref{fig:example_iterations}, bottom.
	Motivated by this we now introduce discrete trajectories of approximate couplings, averaged over composite cells.
	We rely on the following conventions for notation.

	\begin{figure}[bt]
		\centering
		\includegraphics[width=\linewidth]{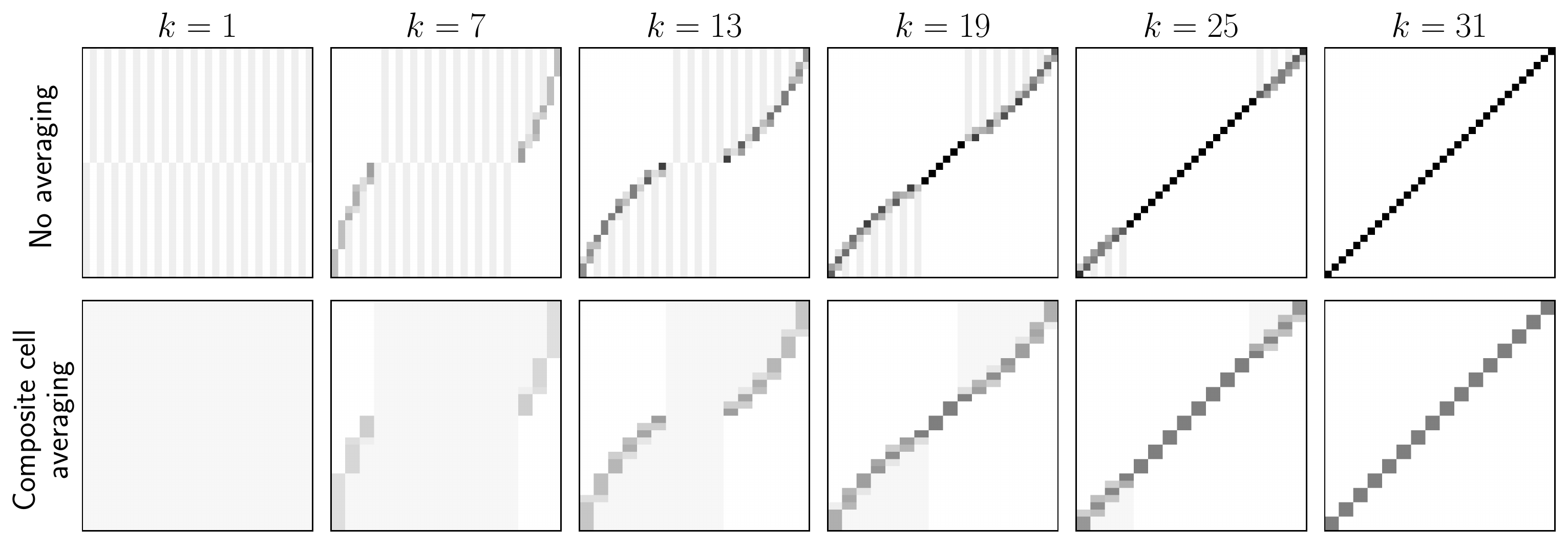}
		\caption{Top, discrete iterates $\iter{\pi}{n,k}$ for $n = 32$, $\mu^n = \nu^n$ discretized Lebesgue, $\piInit^n = \mu^n \otimes \nu^n$, $c^n = c$ the quadratic cost, $\veps^n = 0$. Bottom, same iterates but averaged at the composite cell level. Note how the oscillations in space become much weaker.
		See also Figure \ref{fig:example_iterations_detail} for a similar illustration.}
		\label{fig:example_iterations}
	\end{figure}

	\begin{remark}[Disintegration notation and measure trajectories]\hfill
	\label{rem:Disintegration}
	\begin{enumerate}

	\item We will represent trajectories of measures on $X \times Y$ as measures on $\R_+ \times X \times Y$ and use bold symbols to denote them.
	For a non-negative measure $\bm{\lambda} \in \measp(\R_+ \times X \times Y)$ with $\projTime \bm{\lambda} \ll \Lebesgue$ we write $(\bm{\lambda}_{t})_{t \in \R_+}$ for its disintegration with respect to $\Lebesgue$ on the time-axis such that
	\begin{align*}
		\int_{\R_+ \times X \times Y} \phi(t,x,y) \cdot \diff \bm{\lambda}(t,x,y) = \int_{\R_+} \left[ \int_{X \times Y} \phi(t,x,y) \cdot \diff \bm{\lambda}_t(x,y) \right]\,\diff t,
	\end{align*}
	for $\phi \in \cont_c(\R_+ \times X \times Y)$. This disintegration is well defined for $\sigma$-finite measures: $\bmnu$ can be written as a countable sum of finite measures $(\bmnu^k)_{k\in \N}$ (which have a well-defined disintegration) where $\bmnu^k$ is concentrated on $U_k \times X \times Y$ for a disjoint family of sets $(U_k)_{k\in \N}$ in $\R_+$, so $\bmnu_t$ is given simply by $\bmnu^k_t$, with $k$ the (Lebesgue-a.e.)~unique index satisfying $t\in U_k$.

	\item When $\proj_X \bm{\lambda}_t \ll \mu$ for $\Lebesgue$-a.e.~$t \in \R_+$, we write $(\bm{\lambda}_{t,x})_{(t,x) \in \R_+ \times X}$ for the disintegration in time and $X$ such that
	\begin{align*}
		\int_{\R_+ \times X \times Y} \phi(t,x,y) \cdot \diff \bm{\lambda}(t,x,y) =
		\int_{\R_+} \left[ \int_X \left[ \int_{Y} \phi(t,x,y) \cdot \diff \bm{\lambda}_{t,x}(y) \right]\,\diff \mu(x)\right]\,\diff t,
	\end{align*}
	for $\phi \in \cont_c(\R_+ \times X \times Y)$.
	\item Conversely, a (measurable) family of signed measures $(\lambda_x)_{x \in X}$ in $\meas(Y)$ with uniformly bounded variation (i.e. $\sup_{x\in X}|\lambda_x|(Y) < \infty$) can be glued together along $X$ with respect to $\mu$ to obtain a measure $\lambda \in \meas(X \times Y)$ via
	\begin{align*}
		\int_{X \times Y} \phi(x,y)\,\diff \lambda(x,y) \assign \int_{X} \left[ \int_Y \phi(x,y)\,\diff \lambda_x(y) \right] \diff \mu(x)
	\end{align*}
	for $\phi \in \cont(X \times Y)$. The uniform bounded variation is merely a sufficient condition for the $\lambda$ to have finite mass and that suffices for our purposes. We will denote $\lambda$ as $\mu \otimes \lambda_x$. Similarly, we can glue families over $(t,x) \in \R_+ \times X$ to obtain a $\sigma$-finite measure on $\R_+\times X \times Y$, which we denote by $\Lebesgue \otimes \mu \otimes \lambda_{t,x}$. 
	\item The above points extend to vector measures by component-wise application.
	\end{enumerate}
	\end{remark}

\begin{definition}[Discrete trajectories and momenta]\label{def:DiscreteTrajectories}
	The \emph{discrete trajectory and momentum}, $\bmpi^n \in \measp(\R_+ \times X \times Y)$ and $\bmomega^n \in \meas(\R_+ \times X \times Y)^d$ are defined via their disintegration with respect to $\Lebesgue\otimes \mu$ at $t \in \R_+$, $x\in X$ as:
	\begin{align}
		\bmpi_{t,x}^n
		&\assign
		\frac{1}{m_J^n}
		\iter{\nu_J}{n,k}
		=
		\iter{\rho_J}{n,k},
		\label{eq:pi_discrete}
		\\
		\bmomega_{t,x}^n
		& \assign
		\frac{1}{m_J^n}
		\sum_{\hat{J} \in \neigh(J)}
		\iter{\nu_{J, \hat{J}}}{n,k} \cdot (x_{\hat{J}}^n-x_J^n) \cdot n
		=
		\frac{1}{m_J^n}
		\sum_{b\in \{-1,+1\}^d}  \iter{\nu_{i(J,b)}}{n,k} \cdot b,
		\label{eq:omega_discrete}
	\end{align}
	where we set $k = \floor{nt}$, and $i(J,b)$ is the basic cell contained in $J$ whose center sits at $X_J^n + b/2n$. 
	The vectors $b$ are illustrated in Figure \ref{fig:basicandcompositecells}. 
	For composite $B$ cells at the boundary some of the basic cells $i(J,b)$ might lie outside of $X$ and we ignore the corresponding terms in the sum.
\end{definition}
Note that we use the composite cell marginals $\iter{\nu_J}{n,k}$ in the definition of $\bmpi_{t,x}^n$, hence this implements the aforementioned averaging over composite cells.
An intuitive interpretation of the discrete momentum $\bmomega^n$ is that mass in the basic cell $i(J,\hat{J})$ will travel from $x_J^n$ to $x_{\hat{J}}^n$ during iteration $k$ within a time span of $1/n$.
We will show in Section \ref{sec:ContinuityEquation} that $\bmpi^n$ and $\bmomega^n$ approximately solve a continuity equation on the product space $X \times Y$ in a distributional sense, where $\bmomega^n$ encodes a `horizontal' flow (i.e.~only along the $X$-component). Formally, this can be written as
\begin{align}
	\label{eq:DiscreteCE}
	\partial_t \bmpi^n_t + \ddiv_X \bmomega^n_t = o(1) \qquad \tn{for $t>0$ and}
		\qquad \bmpi^n_{t=0} \to \piInit \qquad \tn{as $n \to \infty$}
\end{align}
and it is to be interpreted via integration against test functions in $\cont^1_c(\R_+ \times X \times Y)$ where the $\R_+$ factor corresponds to time $t$ (cf.~Proposition \ref{prop:CE}).

The discrete trajectory $\bmpi_t^n$ is illustrated in the bottom row of Figure \ref{fig:example_iterations}. The corresponding momentum $\bmomega_t^n$ is visualized in Figure \ref{fig:example_iterations_omega}. A detailed view of the iterations, trajectories and momentum for a small $n$ is given in Figure \ref{fig:example_iterations_detail}.

\begin{figure}[bt]
	\centering
	\includegraphics[width=\linewidth]{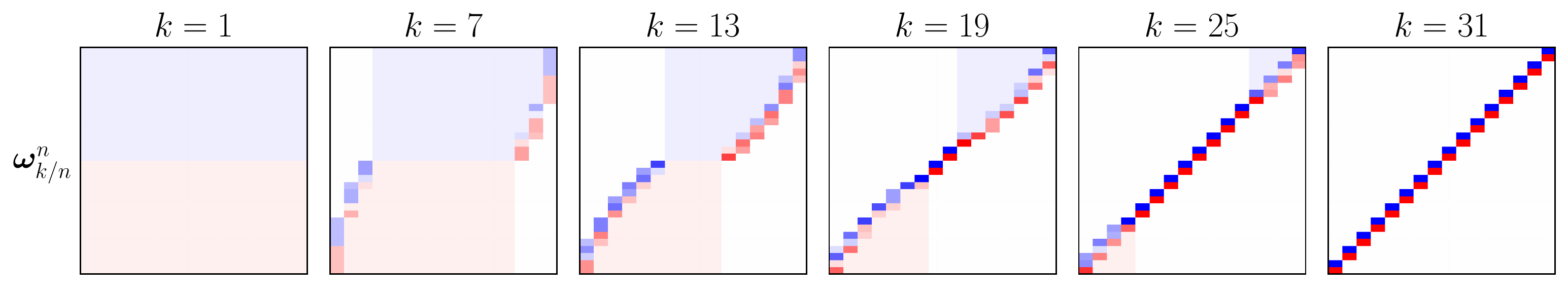}
	\caption{Discrete momentum field for the trajectory shown in Figure \ref{fig:example_iterations}. Blue shade indicates positive velocity (i.e., mass moving towards the right); red shade indicates a negative velocity. The intensity marks the amount of mass that is transported. Note that the momentum does not vanish completely even when the algorithm has converged. This is related to the finite discretization scale. For $n \to \infty$ we anticipate that the momentum (after convergence of the algorithm) converges weak* (but not in norm) to zero.}
	\label{fig:example_iterations_omega}
\end{figure}

\begin{figure}[bt]
	\centering
	\includegraphics[width=\linewidth]{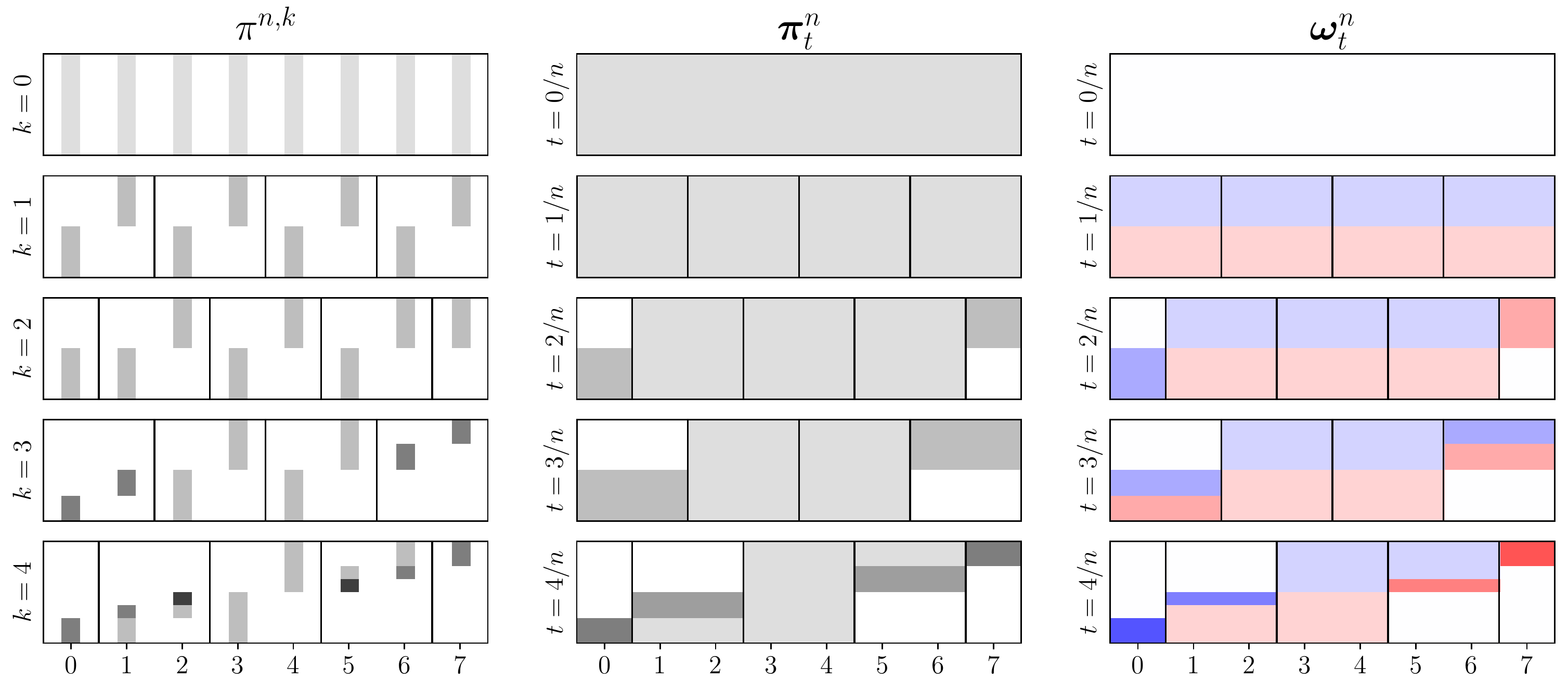}
	\caption{Left, example of discrete iterations $\iter{\pi}{n,k}$ for $n = 8$, $\mu^n = \nu^n=\tn{discretized Lebesgue}$, $\piInitN = \mu^n \otimes \nu^n$, $\veps^n = 0$. The vertical lines represent the extent of the composite cells where the subproblems are solved in each iteration. Center, discrete trajectories $\bmpi^n_t$ for the same setting. These have $\mu$ as $X$-marginal. The averaging procedure reduces oscillations in time as well as in space; the relation between these two kind of oscillations will be given by Proposition \ref{prop:WTVb_implies_WY}. Right, momentum $\bmomega_{t}$. In blue (resp.~red) regions where $\bmomega_t$ is positive (resp.~negative), the intensity marks its magnitude. }
	\label{fig:example_iterations_detail}
\end{figure}

\begin{remark}
	$\bmpi_{t}^n$ is generated from $\pi^{n,\floor{nt}}$ by averaging over the composite cells. Therefore, for all $n \in 2\N$ and $t \in \R_+$, it holds $W_{X \times Y}(\bmpi_{t}^n, \pi^{n,\floor{nt}}) \le 2\sqrt{d}/n$.
	Consequently the sequences $(\bmpi_{t}^n)_n$ and $(\pi^{n,\floor{nt}})_n$ have the same weak* limits or cluster points.
\end{remark}

\subsection{Overview of following sections}
\label{sec:Preview}
The rest of this paper is now dedicated to the study of the asymptotic behaviour of the trajectories $\bmpi^n$ and momenta $\bmomega^n$.

\paragraph{Convergence of $\bmpi^n$.} Taking inspiration from the numerical results shown in Figure \ref{fig:example_iterations_several_n} we expect that the trajectories $\bmpi^n$ converge (under suitable conditions, and up to subsequences) to some limit $\bmpi$. This convergence seems to be stronger than weak*. It appears to be approximately `fiber-wise', i.e.~the disintegrations $\bmpi^n_{t,x}$ converge weak* on $Y$ to some limit $\bmpi_{t,x}$ for `most' $(t,x)$.
A suitable metric for this is given in Section \ref{sec:WY} (Definition \ref{def:WY}).
Convergence will then be studied in Section \ref{sec:OscillationsConvergence} with an Ascoli--Arzelà argument (Proposition \ref{prop:WYUniformConvergence}), for which we establish a suitable notion of equicontinuity (Proposition \ref{prop:WTVb_implies_WY}), as well as pointwise compactness of the trajectories $(\bmpi^n_t)_n$ for all times $t$ (Proposition \ref{prop:WYPrecompacness}).
The main result is summarized in Proposition \ref{prop:WTVbGivesWYConditions}.

The results hinge on a regularity assumption on the discrete iterates $\pi^{n,k}$.
Unfortunately, we are not able to establish validity of this assumption in general. In Section \ref{sec:OscillationsBound} we prove it for the special case of $d=1$, $\mu=\Lebesgue \restr [0,1]$ and $\veps^n=0$ (covering, for instance, the setting of Figures \ref{fig:example_iterations_detail} and \ref{fig:example_iterations_several_n}). Based on our numerical experiments, the assumption seems to hold in the overwhelming majority of cases, although we believe that there are also counter-examples (see Section \ref{sec:Numerics} for details).
In this respect, our situation is comparable to that of \cite{LeclercFlows2020} where the asymptotic convergence of a Lagrangian discretization scheme for minimizing movements in Wasserstein space is established and the result also hinges on a regularity condition on the discrete solutions that could only be proved to hold in one dimension but seems to hold in many cases, based on numerical evidence.

\paragraph{Convergence of $\bmomega^n$.}
The vector measure $\bmomega^n$ (approximately) encodes the evolution of $\bmpi^n$, see the `horizontal' continuity equation \eqref{eq:DiscreteCE}.
It is constructed from the basic cell marginals of $\bmpi^n$, i.e.~from the solutions to the cell-wise problems of the domain decomposition algorithm.
In Section \ref{sec:Convergence} we study the limit $\bmomega$ (up to subsequences) of the $\bmomega^n$ and show that it can be constructed from solutions to a problem that is the $\Gamma$-limit of the cell-wise domain decomposition problems where the cells have collapsed to single points $x \in X$ (Proposition \ref{prop:MomentumConvergence}).
In addition, the limit pair $(\bmpi,\bmomega)$ solves the `horizontal' continuity equation on $X \times Y$, cf.~\eqref{eq:DiscreteCE} (Proposition \ref{prop:CE}).

In summary, the limit of trajectories generated by the domain decomposition algorithm can be associated with a limit notion of the domain decomposition algorithm (Theorem \ref{thm:main}).

\begin{figure}[bt]
	\centering
	\includegraphics[width=0.9\linewidth]{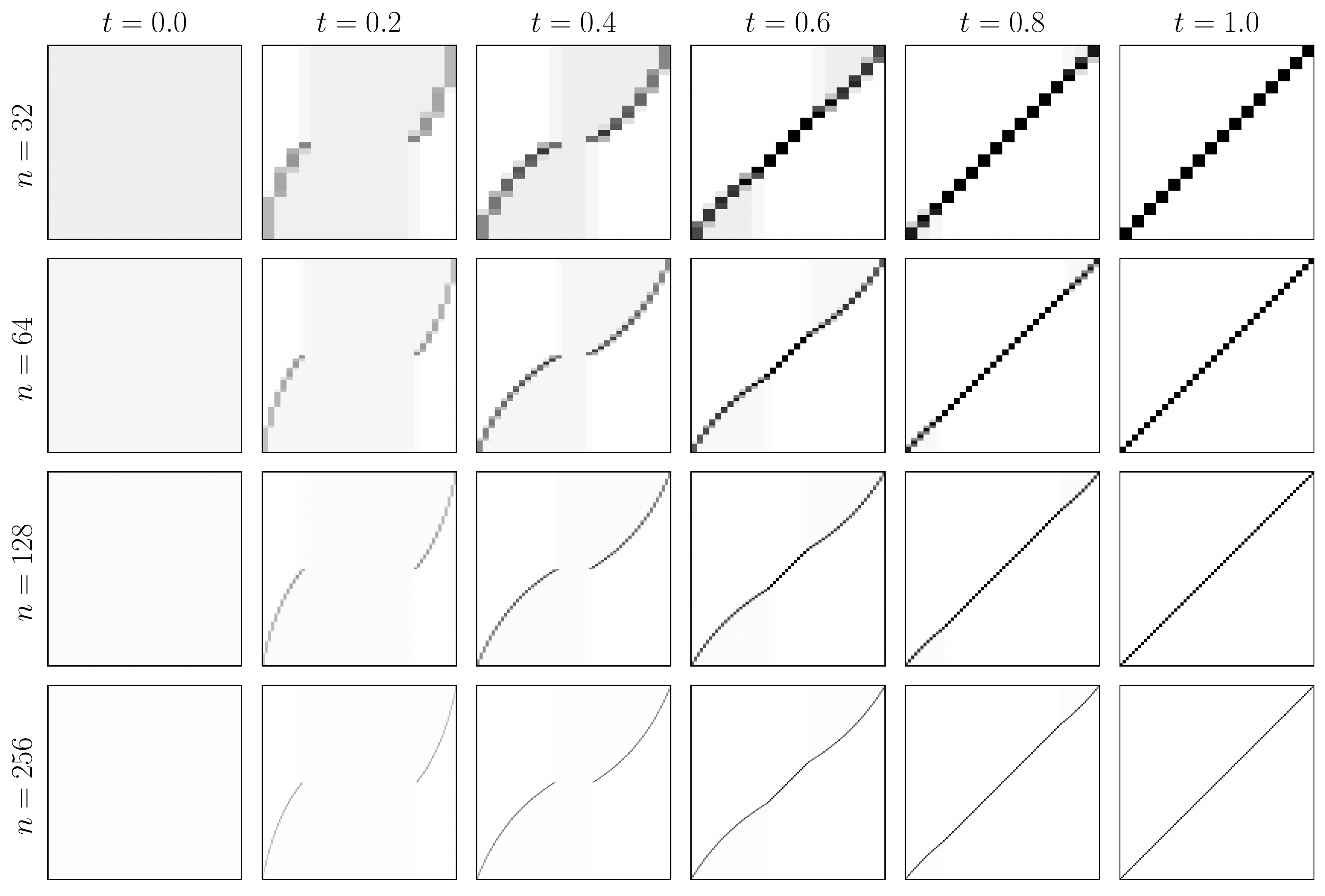}
	\caption{Comparison of discrete trajectories $\bmpi^n$ for increasing $n$ for the same setting as in Figure \ref{fig:example_iterations_detail}. It is tempting to conjecture that the disintegrations $\bmpi^n_{t,x}$ converge weak* for almost all $(t,x)$. 
	}
	\label{fig:example_iterations_several_n}
\end{figure}

\paragraph{Role of regularization parameter $\veps^n$.}
We expect that the behaviour of the limit trajectory and momentum $(\bmpi,\bmomega)$ depends on the behaviour of the sequence of regularization parameters $\veps^n$.
This is motivated by the numerical simulations illustrated in Figure \ref{fig:example_iterations_flipped_bottleneck}. The upper rows show the evolution under the domain decomposition algorithm on two kinds of initial data, for $d=1$, $n = 64$, $\veps^n=0$. The setting dubbed ``flipped'' has again $\mu^n = \nu^n= \tn{discretized Lebesgue}$, but the initial plan is the `flipped' version of the optimal (diagonal) plan. The setting named ``bottleneck'' has as $\mu^n$ a measure with piecewise constant density that features a low-density region (the bottleneck) around $x = 0.5$, while $\nu^n$ is again discretized Lebesgue and $\piInitN= \mu^n \otimes \nu^n$. This bottleneck slows down exchange of mass between the two sides of the $X$ domain, and thus two shocks appear (cf.~$t= 0.4$), which slowly merge as mass traverses the bottleneck.

These evolution examples for $\veps^n=0$ serve as reference for comparison with the regularized cases that are illustrated in the bottom plots of Figure \ref{fig:example_iterations_flipped_bottleneck} for fixed $t$ and various $n$. We examine three different `schedules' for the regularization parameter: $\veps^n=2/n^2$ (left), $\veps^n=2/(64 n)$ (middle) and $\veps^n=2/64^2$ (right). The values were chosen so that the regularization at scale $n=64$ is the same for all schedules. 
\begin{itemize}
	\item For $\veps^n \sim 1/n^2$, the trajectories become increasingly `crisp' and seem to be very close to the unregularized ones.
	\item For $\veps^n \sim 1/n$, the trajectories appear slightly blurred and seem to lag slightly behind the unregularized ones, but still evolve consistently as $n \to \infty$.
	\item For $\veps^n \sim 1$ blur and lag seem to increase with $n$.
\end{itemize}
The three schedules yield $\eta=\lim_{n \to \infty} n \cdot \veps^n = 0$, $\eta \in (0,\infty)$ and $\eta=+\infty$ respectively, see \eqref{eq:EpsEta}. Based on Figure \ref{fig:example_iterations_flipped_bottleneck} we conjecture that for $\eta=0$, the problem describing the limit dynamics of $(\bmpi,\bmomega)$ does not contain an entropic term. For $\eta \in (0,\infty)$, there will be an entropic term. For $\eta=\infty$ the entropic term will dominate and the trajectory $t \mapsto \bmpi_t=\piInit$ will be stationary. This conjecture will be confirmed in Section \ref{sec:Convergence}.

\paragraph{Open questions.}
In this article we are concerned with the convergence of the trajectories $\bmpi^n$ and momenta $\bmomega^n$ as $n \to \infty$ and the dynamics of this limit in $t$. This leads to several natural and important follow-up questions: Does the curve $t \mapsto \bmpi_t$ have a limit as $t \to \infty$? What is the form of this limit (e.g.~does it live on the graph of a map)? When is this limit a minimizer of the (unregularized) optimal transport problem? How fast is the convergence in $t$?
These are beyond the scope of the current article but we consider the present results to be a crucial step on the way.
We will return to this discussion in Section \ref{sec:Conclusion}.

\begin{figure}[p]
	\centering
	\includegraphics[width=0.90\linewidth]{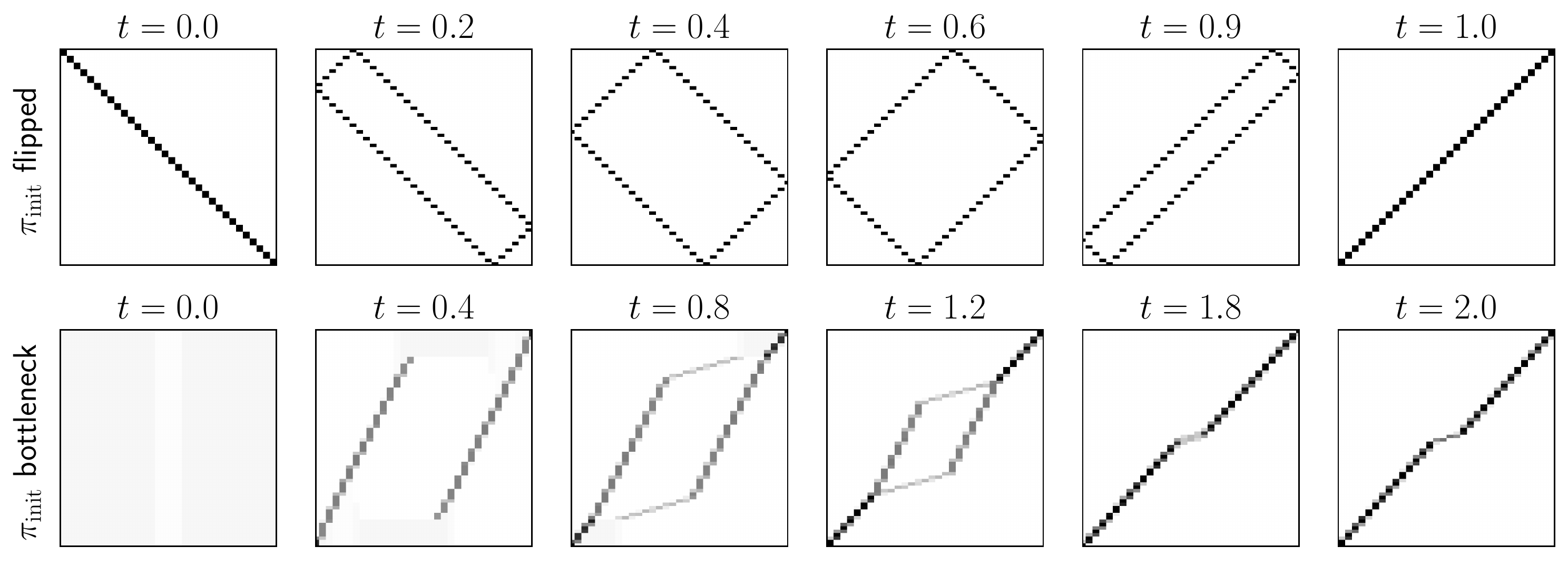}

	\vspace{5mm}

	\begin{minipage}{.5\textwidth}
		\centering
		\includegraphics[width=0.95\linewidth]{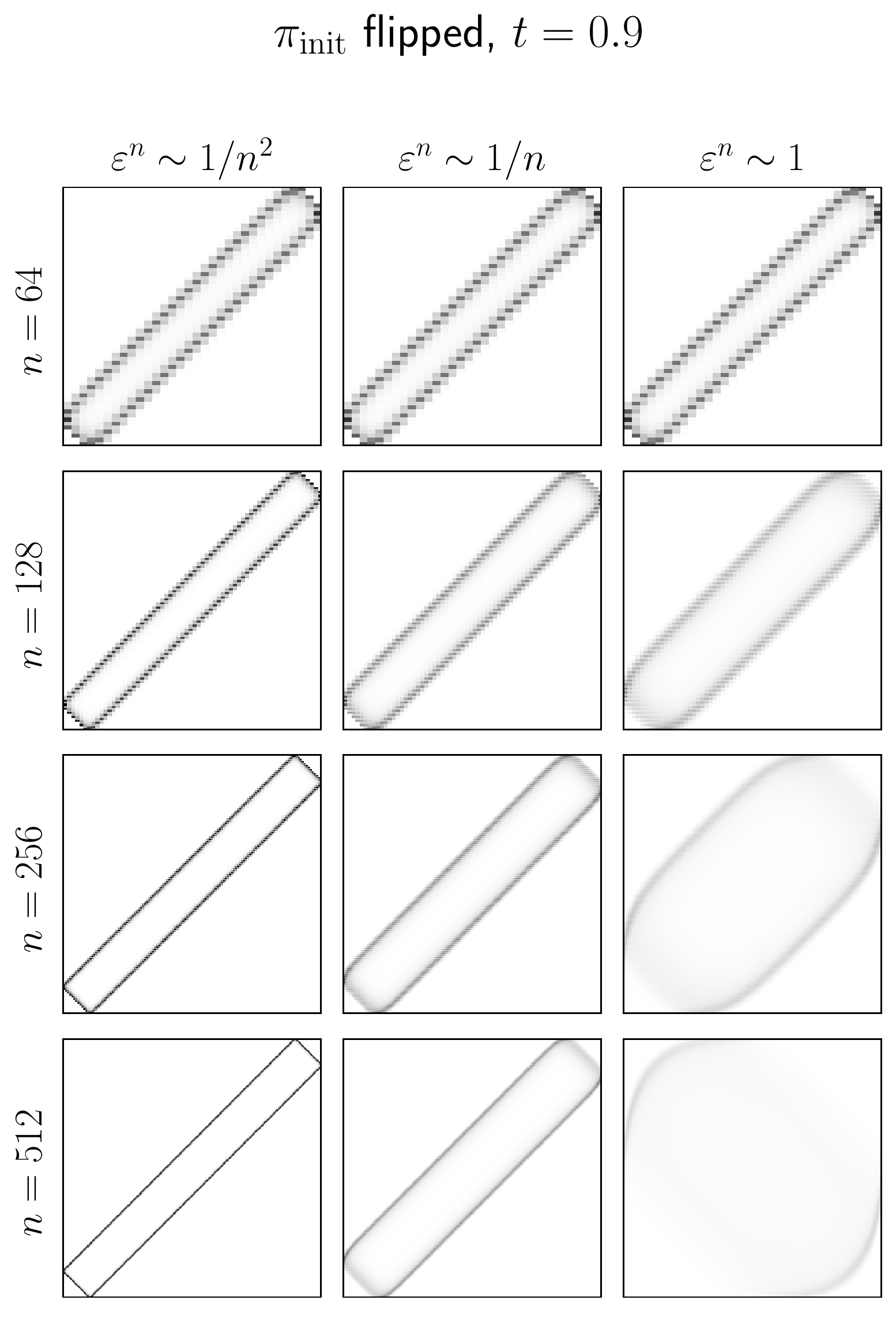}
	\end{minipage}%
	\begin{minipage}{0.5\textwidth}
		\centering
		\includegraphics[width=0.91\linewidth]{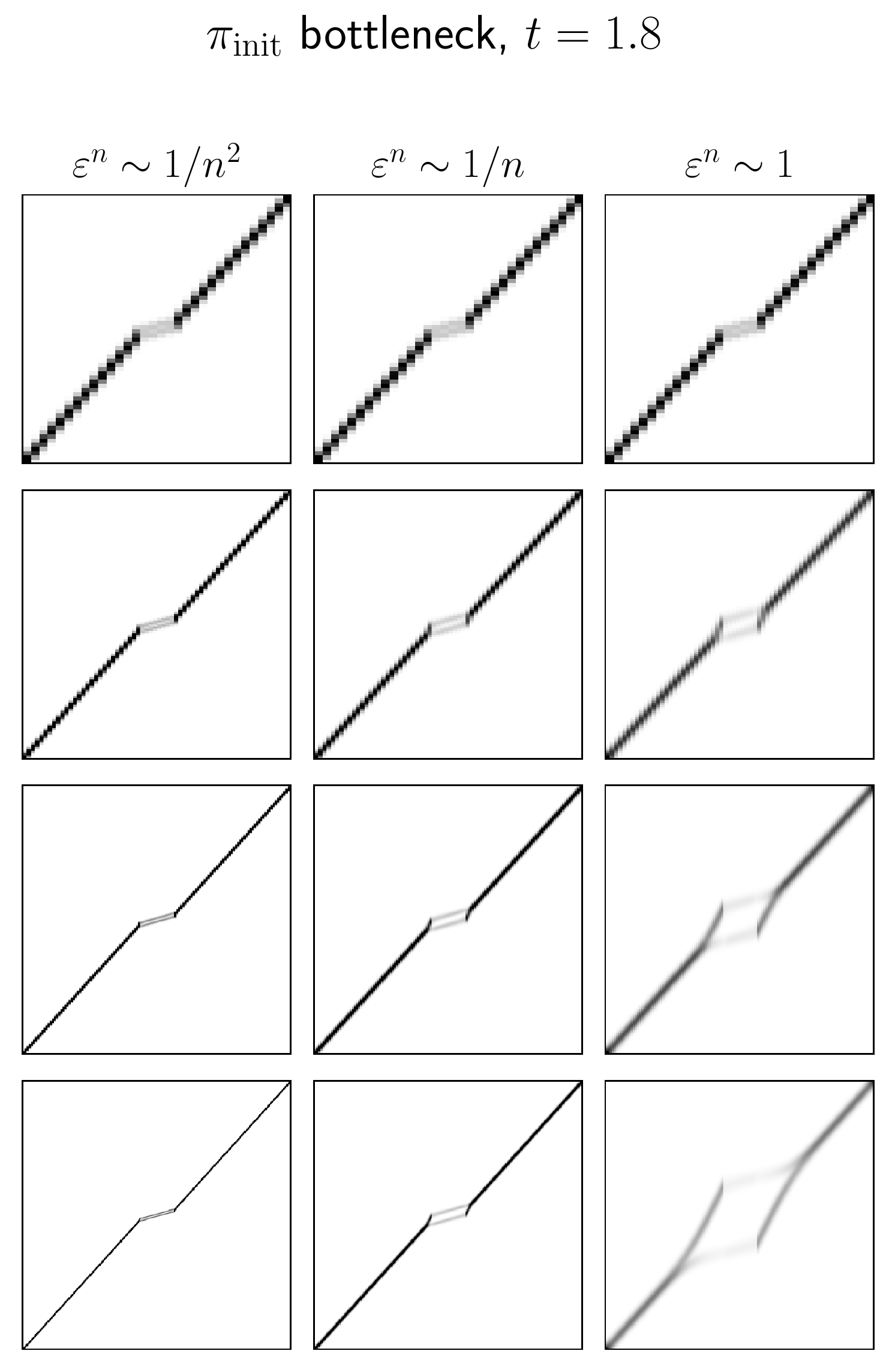}
	\end{minipage}
	\caption{Top, evolution of the discrete trajectories for  $\piInit$ of type ``flipped'' and ``bottleneck'', $\veps=0$. Bottom, snapshot of the discrete trajectories at a fixed time (left $t=0.9$, right $t=1.8$), for different values of $n$ and scaling behaviour of the regularization parameter $\veps^n$.}
	\label{fig:example_iterations_flipped_bottleneck}
\end{figure}

\section{Convergence of trajectories}
\label{sec:Oscillations}
We study in this section convergence of the discrete trajectories $(\bmpi^n)_n$ under the assumption of a uniform bound for a particular notion of spatial oscillations of the iterates (Definition \ref{def:WTVb}).
The proper notion of convergence will be introduced in Section \ref{sec:WY} (Definition \ref{def:WY}), which will be sufficiently strong for the subsequent asymptotic analysis of the momenta $(\bmomega^n)_n$ in Section \ref{sec:Convergence}.
In Section \ref{sec:OscillationsConvergence} we then employ an Ascoli--Arzelà argument.
A priori bounds on the oscillations in the general case are still an open problem. We provide a bound for the special case $d=1$, $\mu=\Lebesgue \restr{[0,1]}$, $\veps^n=0$ in Section \ref{sec:OscillationsBound} and give a discussion.

\subsection{Vertical transport distance}
\label{sec:WY}
In the following we view each $\bmpi^n$ as a curve $t \mapsto \bmpi^n_t$ in the set
\begin{align}
	\label{eq:HalfCouplings}
	\Pi(\mu) \assign \left\{ \pi \in \measp(X \times Y) \, \middle|\, \proj_X \pi = \mu \right\}
\end{align}
which we equip with the following notion of ``fiber-wise vertical convergence''.
\begin{definition}[Vertical transport metric $\WY$]\label{def:WY}
For $\pi, \pi' \in \Pi(\mu)$ we set
\begin{align}
	\WY(\pi,\pi') \assign \int_X \WoY(\pi_{x},\pi_{x}')\,\diff \mu(x).
\end{align}
\end{definition}
\begin{remark}[Interpretation and motivation of $\WY$]
	\label{rem:WYInterpretation}
	$\WY$ can be interpreted as $L^1$-type metric on functions $X \to \prob(Y)$ with reference measure $\mu$ on $X$ and pointwise distance $\WoY$. A measure $\pi \in \Pi(\mu)$ is then interpreted as function $x \mapsto \pi_x$.
	Alternatively, $\WY$ can also be interpreted as an optimal transport metric on $X \times Y$ which only allows `vertical' transport, i.e.~along the $Y$-component. So in any fiber $\{x\} \times Y$ a transport plan from $\pi_{x}$ onto $\pi'_{x}$ must be sought, whereas transport between different $x,x'\in X$ is not allowed.
	From this intuition we deduce the alternative formulation
	\begin{multline}
	\label{eq:WYAlternative}
	\WY(\pi,\pi') = \inf \Bigg\{ \int_{(X \times Y)^2} \|(x,y)-(x',y')\|\,\diff \gamma((x,y),(x',y')) \Bigg| \\
	\gamma \in \Pi(\pi,\pi') \tn{ with $x=x'$ $\gamma((x,y),(x',y'))$-a.e.} \Bigg\},
	\end{multline}
	the relation $\WY(\pi,\pi') \geq W_{X \times Y}(\pi,\pi')$ and that convergence in the former implies convergence in the latter, which is equivalent to weak* convergence (by compactness of $X \times Y$) \cite[Theorem 6.9]{Villani-OptimalTransport-09}.
\end{remark}
\begin{remark}
	\label{rem:WYGeodesic}
	The space $(\Pi(\mu),\WY)$ is geodesic, i.e.~for any pair $\pi^0, \pi^1 \in \Pi(\mu)$ there exists a curve $[0,1] \ni s \mapsto \pi(s) \in \Pi(\mu)$ such that
	\begin{align}
	\label{eq:WYGeodesic}
	\WY(\pi(s),\pi(s')) = |s-s'| \cdot \WY(\pi^0,\pi^1) \qquad \tn{for} \qquad s,s' \in [0,1].
	\end{align}
	This property is inherited from $(\prob(Y),\WoY)$, which is well-known to be geodesic \cite[Chapter 6]{Villani-OptimalTransport-09}, and a geodesic $s \mapsto \pi(s)$ between $\pi^0$ and $\pi^1$ can be written as
	$\pi(s)=\mu \otimes \pi_x(s)$ (see Remark \ref{rem:Disintegration}) where $\mu(x)$-a.e.~$\pi_x(s)$ needs to be a point on a (constant speed) geodesic between $\pi^0_x$ and $\pi^1_x$ such that one has
	$$
		\WoY(\pi_x(s),\pi_x(s')) = |s-s'| \cdot \WoY(\pi^0_x,\pi^1_x) \qquad \tn{for} \qquad s,s' \in [0,1].
	$$
	This readily implies \eqref{eq:WYGeodesic}.
	A geodesic can be obtained from minimizers of \eqref{eq:WYAlternative} as $\pi(s) \assign (T_s)_\sharp \gamma$ where for $s \in [0,1]$ we set
	\begin{align*}
		T_s & : (X \times Y)^2 \to X \times Y, & (x,y,x',y') \mapsto (x,(1-s) \cdot y + s \cdot y').
	\end{align*}		
\end{remark}

\begin{remark}
	The metric space $(\Pi(\mu),\WY)$ is complete, but not compact.
	In Figure \ref{fig:WY_compactness_counterexample} we construct a sequence with no Cauchy subsequence.
\end{remark}

\begin{figure}[bt]
	\centering
	\includegraphics[width=\linewidth]{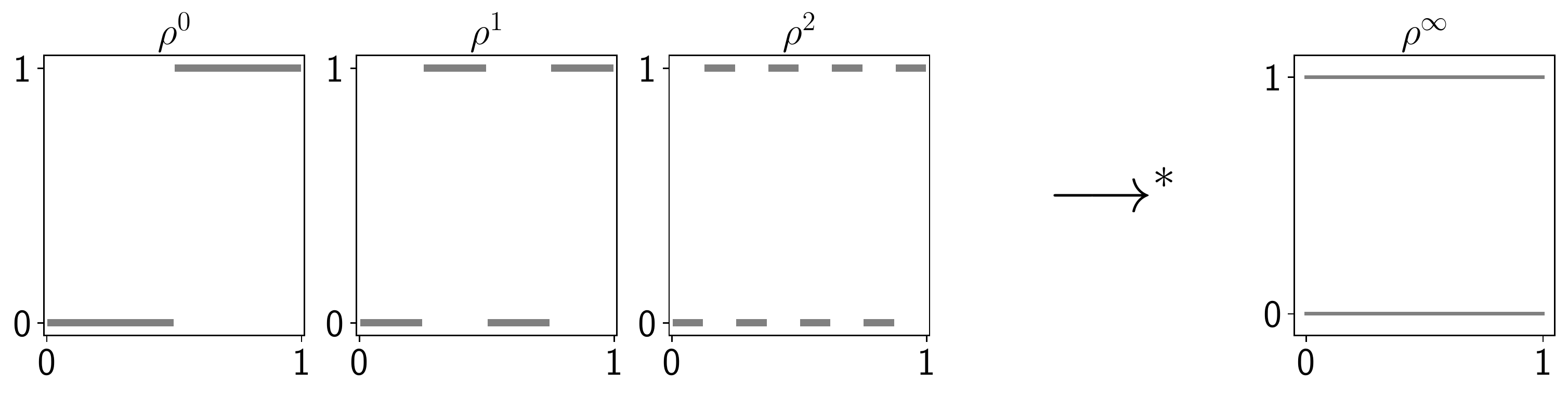}
	\caption{Let $X=Y=[0,1]$. For an integer $n$, construct the measure $\rho^n\in \meas_1([0,1]^2)$ by taking the restriction of the one-dimensional Hausdorff measure $\mathcal{H}^1$ to the alternating gray line segments shown above. The sequence $(\rho^n)_n$ has $\rho^\infty = \tfrac{1}{2}\Lebesgue \otimes (\delta_0 + \delta_1)$ as its weak* limit, but has no Cauchy subsequence in $(\Pi(\mu),\WY)$. This is because $\WY(\rho^n, \rho^m) = \tfrac{1}{2}$ for all $n\neq m$.}
	\label{fig:WY_compactness_counterexample}
\end{figure}

\subsection{Wasserstein total variation and Ascoli--Arzelà argument}
\label{sec:OscillationsConvergence}
\begin{figure}[hbt]
	\centering
	\includegraphics[width=\linewidth]{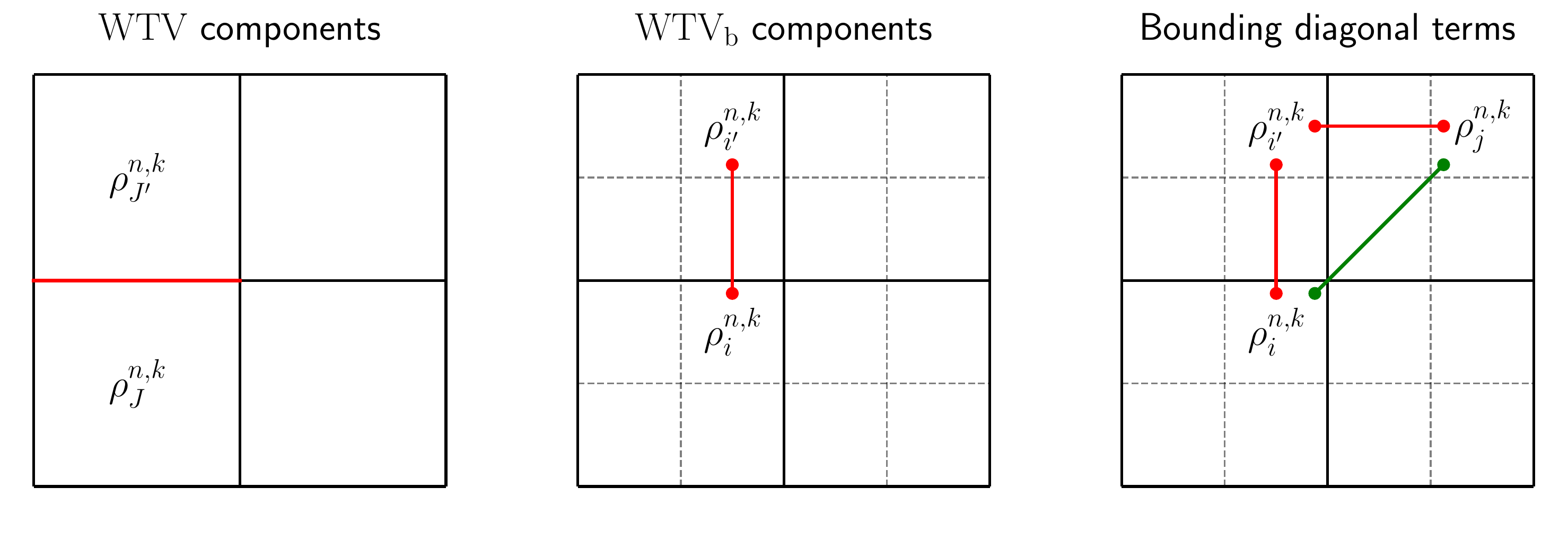}
	\caption{Left, the contribution of the pair of adjacent composite cells $(J, J')$ to $\WTV(\bmpi_{t}^n)$ is the distance $\WoY(\iter{\rho_J}{n,k}, \iter{\rho_J}{n,k})$, weighted by the area of their shared boundary. Center, contribution of a pair of basic cells $(i, i')$ to $\WTVb(\pi^{n,k})$. Note that $i$ and $i'$ are in the same relative position within their respective composite cells. Right, if $i$ and $j$ are in the same relative position in their composite cells, we can bound $\WoY(\iter{\rho_i}{n,k}, \iter{\rho_j}{n,k})$ by a sum of Wasserstein distances along coordinate directions, all of which appear in $\WTVb(\pi^{n,k})$. The number of terms required is precisely $||j-i||_1 / 2$.}
	\label{fig:WTV_figure}
\end{figure}
To apply an Ascoli--Arzelà argument to curves in $(\Pi(\mu),\WY)$ we need to establish equicontinuity of trajectories and pointwise precompactness. In this perspective, two notions of \emph{total variation} for couplings will play a central role (Definitions \ref{def:WTV} and \ref{def:WTVb}).

For given $n \in 2\N$, $t \in \R_+$, recall that the map $x \mapsto \bmpi^n_{t,x}$ is piecewise constant on composite cells $X^n_J$ where $J \in \partnk$ with $k=\lfloor t\,n\rfloor$. Thus, we can define a \emph{Wasserstein total variation} for this function by summing the jump distances of $x \mapsto \bmpi^n_{t,x}$ with respect to $\WoY$ at composite cell boundaries, weighted by boundary areas.
This is formalized in the subsequent definition.
\begin{definition}[$\WTV$: Wasserstein total variation for discrete trajectories]\label{def:WTV}
Consider a discrete trajectory $\bmpi^n$ at time $t$, and let $k = \lfloor t\,n\rfloor$. The \emph{Wasserstein total variation} of $\bmpi^n_t$ is defined as
\begin{equation}
\WTV(\bmpi^n_{t})
\assign
\sum_{\substack{J, J' \in \partnk,\\\tn{adjacent}}}
\mathcal{H}^{d-1}\left( X_J^n \cap X_{J'}^n \right)
\cdot
\WoY(\rho_J^{n,k}, \rho_{J'}^{n, k} ),
\label{eq:WTV_def}
\end{equation}
where $\mathcal{H}^{d-1}$ is the $(d-1)$-dimensional Hausdorff measure. See Section \ref{sec:DomDecNotation}, items \ref{item:NotationAdjacent} and \ref{item:NotationPartialMarginals} for the notion of \emph{adjacent} composite cells and the definition of $\rho_J^{n,k}$. Contributions to $\WTV$ are illustrated in Figure \ref{fig:WTV_figure}, left.
\end{definition}

\begin{remark}[$\WTV$: generalization to general measures $\pi \in \Pi(\mu)$]
\label{remark:WTVIsTotalVariation}
Definition \ref{def:WTV} is a particular case of the total variation of $(\meas_1(Y), \WoY)$-valued functions, as introduced in \cite{AmbrosioMetricBVFunctions1990} for general metric spaces. Indeed, for a discrete trajectory $\bmpi^n$ at time $t$, the function $x\mapsto \bmpi_{t,x}^n$ is a simple function (i.e., it only takes a finite number of values) and it is constant on composite cells. Hence, our definition of $\WTV$ coincides with the formula given in \cite[Proposition 3.1]{AmbrosioMetricBVFunctions1990} for the total variation of metric space-valued simple functions. We will leverage that sequences with bounded $\WTV$ enjoy some compactness \cite[Theorem 2.4(i)]{AmbrosioMetricBVFunctions1990}.
\end{remark}

\begin{remark}[Relation to $\TV$]
For a coupling that is concentrated on the graph of a transport map $S$, $\pi=(\id,S)_\sharp \mu$ one has $\pi_x = \delta_{S(x)}$ for $\mu$-a.e.~$x$. Consequently, $\WTV(\pi)=\TV(S)$ where $\TV(S)$ is the total variation of the map $S$ restricted to $X$.
\end{remark}

$\WTV$ measures oscillations at the composite cell level. We also need a related notion for oscillations at the level of basic cells.

\begin{definition}[Basic cell `skip one' $\WTVb$]\label{def:WTVb}
For a discrete iterate $\pi^{n,k}$, the \emph{basic-cell $\WTV$} is defined as
\begin{equation}\label{eq:WTVb_def}
\WTVb(\pi^{n,k})
\assign
\frac{1}{n^{d-1}}
\sum_{\substack{
		i, i' \in I_n, \\
		i - i' \in \{2\mathbf{e}_1,...,2\mathbf{e}_d\}
	}}
\WoY(\rho_i^{n,k}, \rho_{i'}^{n,k} ).
\end{equation}
with $\bm{e}_i$ the $i$-th coordinate unit vector.
\end{definition}
In $\WTVb$ we compare the (normalized) $Y$-marginals of basic cells with those of basic cells that lie two cells apart in every coordinate direction. This is illustrated in Figure \ref{fig:WTV_figure}. The intuition behind this definition is that basic cell marginals with the same relative position inside their respective composite cells ``play a similar role'' within their composite cells. For example, for $d=1$, $\veps=0$, the right-hand side basic cells always hold the upper part of the mass (cf. Figure \ref{fig:example_iterations_detail}, left).
We will find that a bound on $\WTVb(\iter{\pi}{n,\lfloor t \cdot n \rfloor})$, uniform in $n$ and $t$, provides a uniform bound on $\WTV(\bmpi^n_t)$ (Proposition \ref{prop:WTVb_implies_WTV}) and an equicontinuity result for the discrete trajectories (Proposition \ref{prop:WTVb_implies_WY}). Eventually this will lead to the convergence of trajectories $t \mapsto \bmpi^n_t$ to $t \mapsto \bmpi_t$ in $\WY$, uniformly on compact time intervals.
The relation between the partial results is illustrated in Figure \ref{fig:FlowchartPartialResults}.
We start with a couple of auxiliary lemmas.

\begin{figure}[hbt]
	\begin{tikzpicture}[node distance=25mm]
		\node (WTVb) [result] {$\WTVb(\pi^{n,k})$ \\ uniformly bounded};
		\node (WY) [result, right=of WTVb] {$\WY$-equicontinuity \\ of $(t\mapsto \bmpi_t^n)_n$};
		\node (WTV) [result, below of=WTVb] {$\WTV(\bmpi_t^n)$ \\ uniformly bounded};
		\node (precompactness) [result, below of=WY] {$\WY$-precompactness\\ of $(\bmpi_t^n)_n$};
		\node (main_result) [mainresult, right=of WY, yshift = -10mm, xshift=-10mm] {$\WY$-uniform convergence \\ of $(\bmpi^n)_n$ to $\bmpi$};

		\draw [arrow] (WTVb) -- node[anchor=west] {Prop.~\ref{prop:WTVb_implies_WTV}} (WTV);
		\draw [arrow] (WTV) -- node[anchor=north] {Prop.~\ref{prop:WYPrecompacness}} (precompactness);
		\draw [arrow] (WTVb) -- node[anchor=north] {Prop.~\ref{prop:WTVb_implies_WY}} (WY);
		\draw [arrow] (precompactness) |- (main_result);
		\draw [arrow] (WY) |- node[anchor=north, xshift = 18mm] {Prop.~\ref{prop:WYUniformConvergence}} (main_result);
	\end{tikzpicture}
	\caption{Relationship between partial results. Proposition~\ref{prop:WTVbGivesWYConditions} serves as a summary.}
	\label{fig:FlowchartPartialResults}
\end{figure}
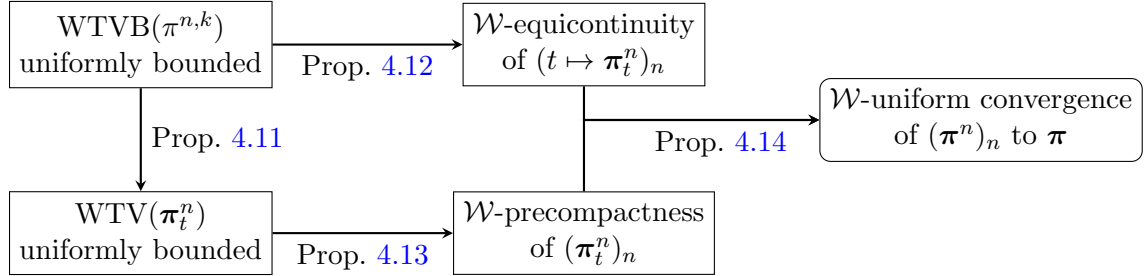

\begin{lemma}\label{lemma:IntegralDifferenceTV}
	Let $\Omega\subset\R^d$ be a bounded open set with Lipschitz boundary. For any $u \in \BV(\Omega)$ and $K \subset \Omega$, one has
	\begin{equation}
	\label{eq:IntegralDifferenceTV}
	\int_{K} |u(x+y) - u(x)| \diff x
	\le
	\|y\| \TV(u)
	\quad
	\text{for all } y \text{ s.t. } \|y\| < \textup{dist}(K, \partial \Omega). 
	\end{equation}
\end{lemma}

\begin{proof}
	Assume $u \in C^\infty(\Omega) \cap \BV(\Omega)$, then
	\begin{align*}
		\int_{K} &|u(x+y)-u(x)| \,\diff x
		=
		\int_{K} \int_0^1 \frac{\diff}{\diff t} u(x+ty)\,\diff t\,\diff x
		=
		\int_0^1 \int_{K} \nabla u(x+ty) \cdot y \,\diff x\,\diff t
		\\
		& \leq \|y\| \int_0^1 \int_{K} \|\nabla u(x+ty)\| \,\diff x\,\diff t
		\leq
		\|y\| \int_0^1 \int_{\Omega} \|\nabla u(x)\| \,\diff x\,\diff t
		=
		\|y\|\TV(u).
	\end{align*}
	The extension to $u \in \BV(\Omega)$ follows by approximation (see, e.g., \cite[Remark 3.25]{ambrosio2000functions}).
\end{proof}

\begin{lemma}
	\label{lem:MassBalanceError}
	Let the density of $\mu$ fulfill $\muLow \le \muDens$ for some $\muLow>0$, and let $\muDens$ have bounded variation. Then, for $n \in 2\N$, $\partJ=\partAn$ or $\partBn$, it holds
	\begin{align}
		\label{eq:MassBalanceErrorLemmaEq}
		\frac{1}{n^{d-1}} \sum_{J \in \partJ} \sum_{i \in J} \left|\frac{m^n_i}{m^n_J}-\frac{1}{2^d}\right|
		\leq
		\frac{\sqrt{d}}{\muLow} \cdot \TV(\muDens).
	\end{align}
\end{lemma}
The expression on the left is zero if the mass within every composite cell is equally distributed over its basic cells. The key step in the proofs of Propositions \ref{prop:WTVb_implies_WTV} and \ref{prop:WTVb_implies_WY} hinges on this equal distribution.
The above Lemma asserts that deviations are bounded by the total variation of $\muDens$ and thus the error inflicted in the proofs below can be controlled.

\begin{proof}
	We use the lower bound of $\muDens$ to write:
	\begin{align}
		\label{eq:BeginningSumBound}
		& \overbrace{
			\muLow\frac{2^d}{n^d}
		}^{\le m_J^n}
		\sum_{J \in \partJ}
		\sum_{i \in J}
		\left|\frac{m^n_i}{m^n_J}-\frac{1}{2^d}\right|
		\le
		\sum_{J \in \partJ} m_J^n
		\sum_{i \in J}
		\left|\frac{m^n_i}{m^n_J}-\frac{1}{2^d}\right|
		\\
		& \quad
		=
		\sum_{J \in \partJ}
		\sum_{i \in J}
		\left|m^n_i-\frac{m_J^n}{2^d}\right|
		\le
		\frac{1}{2^d}
		\sum_{J \in \partJ}
		\sum_{i \in J}
		\sum_{j \in J}
		\left|m^n_i-m_j^n\right|
		\nonumber
		\\
		&\quad
		=
		\frac{1}{2^d}
		\sum_{J \in \partJ}
		\sum_{b \in B}
		\sum_{b' \in B}
		\left|
		\int_{x_J^n + b/2n + [-1/2n,1/2n]^d} \MuDens(x) - \MuDens(x
		+ \tfrac{b'-b}{2n})\diff x
		\right|,
		\nonumber
		\\
		\intertext{with $B \assign \{-1,+1\}^d$. Setting $\mc{X}^n_b \assign \bigcup_{J\in \partJ} (x_J^n + b/2n + [-1/2n,1/2n]^d)$, and using Lemma \ref{lemma:IntegralDifferenceTV},}
		&\quad \le
		\frac{1}{2^d}
		\sum_{b \in B}
		\sum_{b' \in B}
		\int_{\mc{X}^n_b}
		\left|\MuDens(x) - \MuDens(x
		+ \tfrac{b'-b}{2n})
		\right|\diff x
		\le
		\frac{1}{2^d}
		\sum_{b \in B}
		\sum_{b' \in B}
		\frac{||b'-b||}{2n} \TV(\muDens).
	\end{align}
	Using that the cardinality of $|B|$ is $2^d$, that $||b'-b||\le 2 \sqrt d$ and combining the end of the expression with its start at \eqref{eq:BeginningSumBound}, we obtain \eqref{eq:MassBalanceErrorLemmaEq}.
\end{proof}

\begin{proposition}[$\WTVb$-bound implies $\WTV$-bound.]\label{prop:WTVb_implies_WTV}
	Let the density of $\mu$ fulfill $\muLow \le \muDens$ for some $\muLow>0$, and let $\muDens$ have bounded variation.
	Then, for $n \in 2\N$, $k \in \N$, it holds
	\begin{align}
		\WTV(\bmpi_{k/n}^n)
		\leq
		\frac12 \WTVb(\iter{\pi}{n,k})
		+
		d\cdot \diam Y \left[
		2^{2d} + 2^d\frac{\sqrt{d}}{\muLow} \TV(\muDens)
		\right].
	\end{align}
\end{proposition}

\begin{proof}
	Since $x \mapsto \bmpi_{k/n,x}^n$ is piecewise constant on composite cells, with $\bmpi_{k/n,x}^n=\rho^{n,k}_J$ where $J \in \partnk$ is $\mu(x)$-almost surely uniquely determined by the condition $x \in X^n_J$, one has
	\begin{align}
		\label{eq:WTVBoundFirstLine}
		\WTV(\bmpi_{k/n}^n)
		& =
		\sum_{\substack{J, J' \in \partnk,\\\tn{adjacent}}}
		\mathcal{H}^{d-1}\left( X_J^n \cap X_{J'}^n \right)
		\cdot
		\WoY(\rho_J^{n,k},\rho_{J'}^{n,k}).
	\end{align}
	The $B$ composite cells at the boundary have an intersection area smaller than the rest of $B$ cells or $A$ cells. For simplicity, we bound the contributions of the boundary for $A$ or $B$ iterations: there are at most $2d(\tfrac{n}{2}+1)^{d-1}$ boundary composite cells, each with at most $2^d$ neighbors, and the interface area is at most of $(2/n)^{d-1}$. On the other hand, the interface area of the rest of pairs of cells is precisely $(2/n)^{d-1}$. Thus we can bound the expression above by:
	\begin{align}
		\tn{\eqref{eq:WTVBoundFirstLine}} \le
		\left[2d\left( \frac{n}{2} + 1  \right)^{d-1}\right] \cdot 2^d \cdot
		\frac{2^{d-1}}{n^{d-1}}
		\cdot \diam Y
		+
		\frac{2^{d-1}}{n^{d-1}}
		\sum_{\substack{J,J' \in \partnk,\\\tn{adjacent and interior}}} \WoY(\rho_J^{n,k},\rho_{J'}^{n,k}).
		\label{eq:WTVBoundSecondLine}
	\end{align}
	The first term is bounded by $d2^{2d}\diam Y$. We now focus on the second term: recall that for any composite cell $J \in \partAn$ or $\partBn$ we set (Section \ref{sec:DomDecNotation}, item \ref{item:NotationPartialMarginals})
	\begin{align}
		\iter{\rho_J}{n,k} & = \sum_{i \in J} \frac{m^n_i}{m^n_J} \iter{\rho_i}{n,k}. \nonumber \\
		\intertext{We now define}
		\label{eq:BalancedP}
		\iter{\hat{\rho}_J}{n,k} & \assign \sum_{i \in J} \frac{1}{2^d} \iter{\rho_i}{n,k}, \\
		\intertext{and find that}
		\label{eq:MassBalanceError}
		\WoY(\iter{\rho_J}{n,k},\iter{\hat{\rho}_J}{n,k})
		& \leq
		\diam(Y) \cdot \sum_{i \in J} \left|\frac{m^n_i}{m^n_J}-\frac{1}{2^d}\right|.
	\end{align}
	In addition, each term in the sum of the second term of \eqref{eq:WTVBoundSecondLine} may be bounded as
	\begin{equation}\label{eq:AdjacentCellsBound}
	\WoY(\rho_J^{n,k},\rho_{J'}^{n,k})
	\le
	\WoY(\hat{\rho}_J^{n,k},\hat{\rho}_{J'}^{n,k}) + \WoY(\hat{\rho}_J^{n,k},\rho_J^{n,k})
		+ \WoY(\hat{\rho}_{J'}^{n,k},\rho_{J'}^{n,k}).
	\end{equation}
	Let us focus on the first term on the right hand side of \eqref{eq:AdjacentCellsBound}: for two adjacent composite cells $J, J'$ there always exist a coordinate direction $\bm{e}$ such that $J\ni i \mapsto i+2\bm{e}$ is a bijection between cells in $J$ and cells in $J'$. Thus,
	\begin{equation}
	\WoY(\hat{\rho}_J^{n,k},\hat{\rho}_{J'}^{n,k})
	=
	\WoY\left(
		\frac{1}{2^d} \sum_{i \in J} \iter{\rho_i}{n,k},
		\frac{1}{2^d} \sum_{i \in J} \iter{\rho_{i+2\bm{e}}}{n,k}
		\right)
	\le
	\frac{1}{2^d}
	\sum_{i \in J}
	\WoY(\iter{\rho_i}{n,k}, \iter{\rho_{i+2\bm{e}}}{n,k}).
	\end{equation}
	Considering all possible $J$ and $J'$, each contribution of the form $\WoY(\iter{\rho_i}{n,k}, \iter{\rho_{i+2\bm{e}}}{n,k})$ can only appear once, since otherwise we would be counting the same pair of cells $J, J'$ repeatedly. This yields the bound:
	\begin{equation}
	\frac{2^{d-1}}{n^{d-1}} \sum_{\substack{J,J' \in \partnk,\\\tn{adjacent and interior}}}
	\WoY(\hat{\rho}_J^{n,k},\hat{\rho}_{J'}^{n,k})
	\le
	\frac{1}{2}\frac{1}{n^{d-1}}
	\sum_{\substack{i, i' \in I_n, \\ i - i' \in \{2\mathbf{e}_1,...,2\mathbf{e}_d\}}}
	\WoY(\rho_i^{n,k}, \rho_{i'}^{n,k} ),
	\end{equation}
	which is precisely (up to a factor of $1/2$) our definition of $\WTVb(\iter{\pi}{n,k})$.

	The second and third term on the right hand side of \eqref{eq:AdjacentCellsBound} can be bounded by using \eqref{eq:MassBalanceError}, Lemma \ref{lem:MassBalanceError}, and the fact that each composite cell admits at most $2d$ adjacent cells, so each composite cell appears at most $2d$ times in the global sum. Introducing these considerations into \eqref{eq:WTVBoundSecondLine} we obtain:
	\begin{align}
		\WTV(\bmpi_{k/n}^n)
		&\leq
		d2^{2d} \diam Y
		+
		\frac12\WTVb(\iter{\pi}{n,k})
		+
		\frac{2^{d-1}}{n^{d-1}}2d
		\sum_{J\in \partnk}
		\diam(Y) \cdot \sum_{i \in J} \left|\frac{m^n_i}{m^n_J}-\frac{1}{2^d}\right|
		\nonumber
		\\
		&\le
		\frac12 \WTVb(\iter{\pi}{n,k})
		+
		d\cdot \diam Y \left[
		2^{2d} + 2^d\frac{\sqrt{d}}{\muLow} \TV(\muDens)
		\right].
		\nonumber
		\qedhere
	\end{align}
\end{proof}

\begin{proposition}[$\WTVb$-bound implies $\WY$-equicontinuity.]
	\label{prop:WTVb_implies_WY}
	Let the density of $\mu$ fulfill $\muLow \le \muDens \leq \muHi$ for some $\muLow>0$, $\muHi < \infty$, and let $\muDens$ have bounded variation. Then there exists a constant $C<\infty$ (possibly depending on the dimension $d$) such that for $n \in 2\N$, $k \in \N$, it holds
	\begin{align}
		\WY(\bmpi^n_{k/n},\bmpi^n_{(k+1)/n})
		\leq
		\frac{ M_u}{n}
		\left[
		C \cdot \WTVb(\pi^{n,k})
		+
		2^{d+1} \diam Y \frac{\sqrt{d}}{\muLow} \TV(\muDens)
		+
		2d\diam Y
		\right].
	\end{align}
\end{proposition}

\begin{proof}
	Fix $n, k$.
	For $i \in I^n$ denote by $J(i)$ the unique element of $\partnk$ with $i \in J(i)$.  Likewise, $\hat{J}(i)$ is the unique element of $\partnkk$ containing $i$. We also define $\boundaryIn$ as the set of basic cells $i\in I^n$ whose extent $X_i^n$ has non-empty intersection with $\partial X$, and $\interiorIn \assign I^n \setminus \boundaryIn$.
	Further, observe that $\iter{\rho}{n,k+1}_{\hat{J}(i)} = \iter{\rho}{n,k}_{\hat{J}(i)}$ since $\hat{J}(i) \in \partnkk$, see \eqref{eq:CompCellMarginalPreservation}.
	Then,
	\begin{align}
		\WY(\bmpi^n_{k/n},\bmpi^n_{(k+1)/n})
		&=
		\sum_{i\in I^n}
		m_i^n
		\WoY(\iter{\rho}{n,k}_{J(i)} ,\iter{\rho}{n,k+1}_{\hat{J}(i)} )
		=
		\sum_{i\in I^n}
		\underbrace{m_i^n}_{\le \muHi n^{-d}}
		\WoY(\iter{\rho}{n,k}_{J(i)} ,\iter{\rho}{n,k}_{\hat{J}(i)} )
		\nonumber
		\\
		&\le
		\muHi n^{-d}
		\left(
		\sum_{i\in \boundaryIn}
		\WoY(\iter{\rho}{n,k}_{J(i)} ,\iter{\rho}{n,k}_{\hat{J}(i)} )
		+
		\sum_{i\in \interiorIn}
		\WoY(\iter{\rho}{n,k}_{J(i)} ,\iter{\rho}{n,k}_{\hat{J}(i)} )
		\right)
		\nonumber
		\\
		&\le
		\frac{2d \muHi \diam Y}{n}
		+
		\muHi n^{-d}
		\sum_{i\in \interiorIn}
		\WoY(\iter{\rho}{n,k}_{J(i)} ,\iter{\rho}{n,k}_{\hat{J}(i)} ),
		\label{eq:LastLineWYBound}
	\end{align}
	where we have used that the cardinality of $\boundaryIn$ is at most $2d n^{d-1}$.

	Using again \eqref{eq:BalancedP}, we can bound the last term of \eqref{eq:LastLineWYBound} by
	\begin{align}\label{eq:WTV_contributions}
		\sum_{i\in \interiorIn}
		\WoY(\iter{\rho}{n,k}_{J(i)} ,\iter{\rho}{n,k}_{\hat{J}(i)} )
		&\le
		\sum_{i \in \interiorIn} \WoY(\hat{\rho}_{J(i)}^{n,k},\hat{\rho}_{\hat{J}(i)}^{n,k})
		+ \WoY(\rho^{n,k}_{J(i)},\hat{\rho}^{n,k}_{J(i)})
		+ \WoY(\rho^{n,k}_{\hat{J}(i)},\hat{\rho}^{n,k}_{\hat{J}(i)}).
	\end{align}
	The first term can be bounded by $C  n^{d-1} \WTVb(\pi^{n,k})$ for a certain $C$, let us see how. First, for $i\in \interiorIn$, $j\in J(i)$, introduce the \textit{pivoted cell} $\pivot(i,j) \assign j + 2(i-j)$. It is easy to check that $j\mapsto \pivot(i,j)$ is a bijection between basic cells in $J(i)$ and $\hat{J}(i)$. Then a first step is to bound, for each $i\in \interiorIn$,

	\begin{equation}\label{eq:CompositeCellWTVb}
	\WoY(\hat{\rho}_{J(i)}^{n,k},\hat{\rho}_{\hat{J}(i)}^{n,k})
	\le
	\sum_{j \in J(i)} \frac{1}{2^d}\WoY(\iter{\rho_j}{n,k}, \iter{\rho_{\pivot(i,j)}}{n,k}).
	\end{equation}

	Now we introduce the ``\emph{$\WTVb$ graph}'' with $I^n$ as set of vertices and
	\begin{equation*}
	E \assign \{(j, j') \subset I_n \mid
	j - j' \in \{2\mathbf{e}_1,...,2\mathbf{e}_d\}
	\}
	\end{equation*}
	as its set of edges where each edge corresponds to one term in the definition of $\WTVb$. Then for each $i,j$ we can find a path in the graph between $j$ and $\pivot(i,j)$ consisting of at most $d$ edges --- this is because $j$ and $\pivot(i,j)$ can be regarded as vertices of a coordinate hypercube with edges contained in $E$ (as exemplified in Figure \ref{fig:WTV_figure}, right). We name this path $E(i,j)  \subset E$ and so each distance in \eqref{eq:CompositeCellWTVb} can be bounded as
	\begin{equation*}
	\WoY(\iter{\rho_j}{n,k}, \iter{\rho_{\pivot(i,j)}}{n,k})
	\le
	\sum_{(\ell, \ell') \in E(i,j)}
	\WoY(\iter{\rho_\ell}{n,k}, \iter{\rho_{\ell'}}{n,k} ).
	\end{equation*}
	Collecting these considerations one arrives at the bound
	\begin{equation}
	\sum_{i \in \interiorIn} \WoY(\hat{\rho}_{J(i)}^{n,k},\hat{\rho}_{\hat{J}(i)}^{n,k})
	\le
	\frac{1}{2^d}
	\sum_{i \in \interiorIn}
	\sum_{j\in J(i)}
	\sum_{(\ell, \ell') \in E(i,j)}
	\WoY(\iter{\rho_\ell}{n,k}, \iter{\rho_{\ell'}}{n,k} ),
	\end{equation}
	where all the distances are computed along edges in $E$. Now, each edge $e$ of $E$ is used at most as many times as $j$ and $\pivot(i,j)$ are vertices of a hypercube that has $e$ as one of its edges. This number of events is of course finite and independent of $n$ and $k$, so there exists a constant $C$ (that may depend on the dimension $d$) such that
	\begin{equation}
	\label{eq:WTVPartWY}
	\sum_{i \in \interiorIn} \WoY(\hat{\rho}_{J(i)}^{n,k},\hat{\rho}_{\hat{J}(i)}^{n,k})
	\le
	C
	\sum_{(\ell, \ell') \in E}
	\WoY(\iter{\rho_\ell}{n,k}, \iter{\rho_{\ell'}}{n,k} )
	=
	Cn^{d-1} \WTVb(\pi^{n,k}).
	\end{equation}
	
	In the second and third term of \eqref{eq:WTV_contributions}, note that each composite cell appears at most $2^d$ times (since each inner composite cell contains $2^d$ basic cells). After using this fact, we apply \eqref{eq:MassBalanceError} and Lemma \ref{lem:MassBalanceError} to bound the second term of \eqref{eq:WTV_contributions} by
	\begin{align}
		\sum_{i \in \interiorIn} \WoY(\rho^{n,k}_{J(i)},\hat{\rho}^{n,k}_{J(i)})
		&\leq
		2^d \sum_{J \in \partnk} \WoY(\rho^{n,k}_{J},\hat{\rho}^{n,k}_{J})
		\leq
		2^d \sum_{J \in \partnk}
		\diam(Y) \cdot \sum_{i \in J} \left|\frac{m^n_i}{m^n_J}-\frac{1}{2^d}\right|
		\nonumber
		\\
		&\le
		2^d n^{d-1} \diam Y
		\frac{\sqrt{d}}{\muLow} \cdot
		\TV(\muDens).
		\label{eq:MassBalancingPartWY}
	\end{align}
	The same bound applies to the third term. Inserting \eqref{eq:WTVPartWY} and \eqref{eq:MassBalancingPartWY} into \eqref{eq:LastLineWYBound} yields
	\begin{equation*}
	\WY(\bmpi^n_{k/n},\bmpi^n_{(k+1)/n})
	\le
	\frac{ M_u}{n}
	\left[
	C \cdot \WTVb(\pi^{n,k})
	+
	2^{d+1} \diam Y \frac{\sqrt{d}}{\muLow} \TV(\muDens)
	+
	2d\diam Y
	\right]
	\end{equation*}
	where $C$ is the constant in \eqref{eq:WTVPartWY}.
\end{proof}

For fixed $t>0$, a uniform $\WTV$ bound for the sequence $(\bmpi_t^n)_n$ provides compactness of the sequence itself within $(\Pi(\mu), \WY)$. This is a direct consequence of \cite[Theorem 2.4(i)]{AmbrosioMetricBVFunctions1990} combined with the boundedness of $Y$. 
\begin{proposition}
	\label{prop:WYPrecompacness}
	Let $(\gamma^n)_n$ be a sequence in $\Pi(\mu)$ with uniformly bounded $\WTV$. Then the sequence is precompact in $\Pi(\mu)$ with respect to $\WY$, and any cluster point $\gamma$ satisfies
	\begin{equation}
	\WTV(\gamma) \le \liminf_{n\to \infty} \WTV(\gamma^n).
	\end{equation}
\end{proposition}

\begin{proof}
	Since $(\gamma^n)_n$ is a sequence of bounded $\WTV$, the family of disintegration maps $(X\ni x\mapsto \gamma^n_{x}\in \prob(Y))_n$ have uniformly bounded variation with respect to $\WoY$  (cf. Remark \ref{remark:WTVIsTotalVariation}). Besides, since the space $(\prob(Y),\WoY)$ has finite diameter, by assumption for any $\rho \in \prob(Y)$ we have that the sequence
	\begin{align*}
		\WTV(\gamma^n) + \int_X \WoY(\gamma^n_{x},~\rho)\,\diff x
	\end{align*}
	is bounded.
	Therefore, by \cite[Theorem 2.4(i)]{AmbrosioMetricBVFunctions1990}, there is some map $x\mapsto \gamma_x$ (that can be identified with some element $\gamma\in\Pi(\mu)$) with $\WTV(\gamma) \leq \liminf_n \WTV(\gamma^n)$, such that up to selection of a subsequence, $(\gamma^n_{x})_n$ converges to $\gamma_x$ for almost all $x \in X$.
	Since $Y$ has finite diameter, by dominated convergence this implies that $\gamma^n \to \gamma$ in $\WY$.
\end{proof}

\begin{proposition}
	\label{prop:WYUniformConvergence}
	Assume the discrete trajectories $\bmpi^n$ satisfy the `almost-equicontinuity' condition
	\begin{align}\label{eq:almost_equicontinuity}
		\WY(\bmpi^n_{k/n},\bmpi^n_{(k+1)/n}) \leq C/n \quad \tn{for all } n \in 2\N,\,k \in \N
	\end{align}
	for some $C \in \R_+$ that does not depend on $k$ or $n$.
	In addition, assume that the set $\{\bmpi^n_t | n \in 2\N\}$ is precompact in $(\Pi(\mu),\WY)$ for all $t \in \R_+$.
	Then there exists a subsequence $(\bmpi^{n_l})_l$ and a trajectory $\bmpi \in \measp(\R_+ \times X \times Y)$ with $\bmpi_t \in \Pi(\mu,\nu)$ for all $t \in \R_+$, such that for every $T \in \R_+$, $\bmpi^{n_l}_t$ converges to $\bmpi_t$ in $\WY$ uniformly for $t \in [0,T]$.
\end{proposition}

\begin{proof}
	Equation \eqref{eq:almost_equicontinuity} states that the family $(t \mapsto \bmpi^n_t)_n$ (which is piecewise constant in $t$) is close to being equicontinuous in the $\WY$ metric. Since our purpose is to use Ascoli--Arzelà, we will construct equicontinuous versions of $\bmpi^n$ by using that $(\Pi(\mu),\WY)$ is geodesic (Remark \ref{rem:WYGeodesic}).	
	For every $n \in 2\N$ we introduce the trajectory $\R_+ \ni t \mapsto \tilde{\bmpi}^n_t$ (and the corresponding measure in $\measp(\R_+ \times X \times Y)$) by setting
	$$\tilde{\bmpi}^n_t \assign \bmpi^n_t \qquad \tn{for} \qquad t=k/n,\,k \in \N$$
	and on every interval $[k/n,(k+1)/n]$, $k \in \N$, we set $\tilde{\bmpi}^n_t$ to a constant speed geodesic with respect to $\WY$ between $\bmpi^n_{k/n}$ and $\bmpi^n_{(k+1)/n}$.
	\eqref{eq:almost_equicontinuity} then implies that the curve $t \mapsto \tilde{\bmpi}^n_t$ is Lipschitz with Lipschitz constant $C$ on each interval $[k/n,(k+1)/n]$ and thus on $\R_+$. Consequently, the family $(t \mapsto \tilde{\bmpi}^n_t)_n$ is equi-Lipschitz and thus equicontinuous.

	In addition, the construction also implies for all $t \in \R_+$ that
	\begin{align}
	\label{eq:distance_rho_pi}
	\WY(\tilde{\bmpi}^n_t,\bmpi^n_t) = \WY(\tilde{\bmpi}^n_t,\bmpi^n_{k/n})= \WY(\tilde{\bmpi}^n_t,\tilde{\bmpi}^n_{k/n}) \leq C/n \qquad \tn{where} \qquad k = \lfloor t \cdot n\rfloor.
	\end{align}
	So any cluster point of $(\bmpi^n_t)_n$ is also a cluster point of $(\tilde{\bmpi}^n_t)_n$ and thus precompactness of the former implies precompactness of the latter.

	We are now in the position to invoke an Ascoli--Arzelà theorem (in particular \cite[Chapter 7, Theorem 17]{kelley1975general}) to conclude that $(t \mapsto \tilde{\bmpi}^n_t)_n$ converges (up to subsequences) to some continuous $t \mapsto \bmpi_t$, uniformly in $\WY$ on any compact $[0,T]$. By \eqref{eq:distance_rho_pi}, the same subsequence of $\bmpi^n$ also converges uniformly in $\WY$ to $\bmpi$.
\end{proof}

We conclude the section by stating our main convergence result, which is given under the assumption of bounded basic cell oscillations and whose proof follows combining previous results as illustrated in the flow chart of Figure \ref{fig:FlowchartPartialResults}.
\begin{proposition}
	\label{prop:WTVbGivesWYConditions}
	Assume that $\mu \ll \Lebesgue$, with density bounded away from $0$ and $+\infty$ and of bounded variation (Assumption \ref{assumption:RegularityMu}).
	Assume the discrete iterates $\pi^{n,k}$ have uniformly bounded basic cell variation, i.e.
	\begin{align}
	\sup_{n \in 2\N,k \in \N} \WTVb(\pi^{n,k}) < +\infty.
	\end{align}
	Let $(\bmpi^n)_n$ be the discrete trajectories defined in Definition \ref{def:DiscreteTrajectories}.
	Then, up to a subsequence, there exists a trajectory $\bmpi \in \measp(\R_+ \times X \times Y)$ with $\bmpi_t \in \Pi(\mu,\nu)$ for all $t \in \R_+$, such that $\bmpi^{n}_t$ converges to $\bmpi_t$ in $\WY$ uniformly in $t \in [0,T]$, for all $T > 0$.
\end{proposition}

\subsection{Oscillation bound in one dimension, with Lebesgue-marginal, without regularization}
\label{sec:OscillationsBound}

\begin{proposition}\label{prop:OscillationsBoundPartialResult}
	Let $d=1$, $c(x,y)\assign h(x-y)$ for $h$ strictly convex, $\mu^n=\mu=\Lebesgue \restr [0,1]$, $\veps^n=0$, $\piInitN=\piInit$. 
	Then for $n \in 2\N$, $k \in \N$,
	\begin{align*}
		\WTVb(\iter{\pi}{n,k}) \leq 2\WTV(\piInit) + 4\,\diam Y.
	\end{align*}
\end{proposition}
The proof is given in Appendix \ref{app:ProoWTVbBound}.
It is based on the monotonicity of unregularized optimal transport in one dimension.
We show that contributions to $\WTVb$ in the bulk of $X$ are locally non-increasing. Increase can only be generated at the boundaries and is controlled by $\diam Y$.
Numerical evidence suggests that more generally, if $\mu \neq \Lebesgue$ (but otherwise still the setting of Proposition \ref{prop:OscillationsBoundPartialResult}), increase is generated by variations in the density of $\mu$, and $\TV(\muDens)<\infty$ would probably be a necessary condition for a generalization. (The boundaries of $X$ can be interpreted as variations of $\muDens$, where the density drops to zero.)

The monotonicity argument breaks down in higher dimensions and in the presence of entropic regularization.
We observe numerically that in these cases the local non-increasing property does no longer hold exactly. But in the overwhelming majority of numerical examples it seems clear that $\WTVb(\pi^{n,k})$ is uniformly bounded.
In fact it was non-trivial to find a conjectured counter-example. One such example is presented in Section \ref{sec:WTVExample}.

Generally, while regularization does disrupt the strict monotonicity of optimal transport, we observed numerically that it seems to have a regularizing effect on $\WTVb$.
In case of the conjectured counter-example $\WTVb$ is much lower in the presence of entropic regularization but based on numerical evidence it is still unclear whether it is bounded or not.

\section{Convergence of momenta and dynamics}
\label{sec:Convergence}

So far we have treated the convergence of discrete trajectories $(\bmpi^n)_n$ in the vertical transport metric $\WY$.
Now we focus on the convergence of the momenta. The discrete momenta $(\bmomega^n)_n$ are constructed from basic cell marginals, see \eqref{eq:omega_discrete}, that are obtained by solving the cell-wise transport problems in Algorithm \ref{alg:DomDec}, line \ref{line:CellProblem}. They approximately describe the temporal evolution of the discrete trajectories via a horizontal continuity equation \eqref{eq:DiscreteCE}.
We will now show that there exists a limit momentum $\bmomega$, which is constructed from the solutions to fiber-wise optimal transport problems (entropic regularization strength given by $\eta = \lim_{n \to \infty} n \cdot \veps^n$), which are the $\Gamma$-limit of the (re-scaled) discrete cell problems. Convergence of $(\bmpi^n)_n$ in $\WY$ is required to obtain meaningful fiber-wise convergence of the problems, weak* convergence would not be sufficient. Finally, the limit trajectory $\bmpi$ and momentum $\bmomega$ solve the horizontal continuity equation.
We can think of this limit as a continuum limit of the domain decomposition algorithm.
	
This section is structured as follows: In Section \ref{sec:ProblemRescaling} we introduce the re-scaled versions of the domain decomposition cell problems which have a meaningful $\Gamma$-limit.
In Section \ref{sec:GluedProblems} we state the supposed limit functional and prepare the proof.
Liminf and limsup conditions are provided in Sections \ref{sec:Liminf} and \ref{sec:Limsup}.
The continuity equation is addressed in Section \ref{sec:ContinuityEquation}.

\begin{assumption}\label{assumption:PitConvergence}
	Assume that there is a trajectory $\bmpi=\Lebesgue \otimes \bmpi_t \in \measp(\R_+ \times X \times Y)$, $\bmpi_t \in \Pi(\mu,\nu)$ for a.e.~$t \in \R_+$, such that, up to extraction of a subsequence $\nset \subset 2\N$, the discrete trajectories $(\bmpi^n)_{n}$, defined in \ref{def:DiscreteTrajectories}, converge for a.e.~$t \in \R_+$ to $\bmpi$ in $\WY$. More precisely,
	\begin{equation}
	\lim_{n\in \nset,\ n\rightarrow\infty}\WY(\bmpi^n_t, \bmpi_t) = 0,
	\qquad
	\tn{for $\Lebesgue$-a.e.~$t$.}
	\end{equation}
	Recall that this implies $\bmpi^n_t \rightweaks \bmpi_t$ for a.e.~$t$ (Remark \ref{rem:WYInterpretation}).
\end{assumption}

\begin{remark}
	If Assumption \ref{assumption:RegularityMu} holds and $\sup_{n,k} \WTVb(\pi^{n,k})<\infty$ then Assumption \ref{assumption:PitConvergence} holds by virtue of Proposition \ref{prop:WTVbGivesWYConditions} (the latter implies a slightly stronger notion of convergence).
\end{remark}

\subsection{Re-scaled discrete cell problems}
\label{sec:ProblemRescaling}
Recall that at resolution $n\in2\N$, during iteration $k\in \N$, in a composite cell $J \in \partnk$ we need to solve the following regularized optimal transport problem (Algorithm \ref{alg:DomDec}, line \ref{line:CellProblem}):
\begin{align}
	\label{eq:ExampleCellProblem}
	\inf\left\{ \int_{X_J^n \times Y} c^n\,\diff \pi + \veps^n \cdot \KL(\pi|\mu^n_J \otimes \nu^n)
	\,\middle|\, \pi \in \Pi(\mu^n_J,\nu^{n,k}_J) \right\}
\end{align}
For the limiting procedure we will map $X^n_J$ to a reference hyper-cube $Z=[-1,1]^d$, normalize the cell marginals $\mu^n_J$ and $\nu^{n,k}_J$ (the latter will then become $\bmpi^n_{k/n,x}$ for $x \in X^n_J$). We will subtract some constant contributions from the transport and regularization terms and re-scale the objective such that the dominating contribution is finite in the limit (the proper scaling will depend on whether $\eta$ is finite).
In addition, as $n \to \infty$, the cells $X^n_J$ become increasingly finer, we thus expect that we can replace the cost function $c^n$ by a linear expansion along $x$. These transformations are implemented in the two following definitions, yielding the functional \eqref{eq:Fn}. Equivalence with \eqref{eq:ExampleCellProblem} is then established in Proposition \ref{prop:DomDecGeneratesMinimizerFn}.

\begin{definition}\label{def:reference_cell}
	We define the \textit{scaled composite reference cell} as $Z = [-1,1]^d$. Let $n \in \nset$.
	\begin{itemize}
		\item For a composite cell $J \in \partnk$, $k > 0$, the \textit{scaling map} of cell $J$ is given by
		\[
		S_J^n : X_J^n \to Z, \quad x \mapsto n(x - x_J^n).
		\]
		\item For $J \in \partnk$, $k > 0$, we define the \textit{scaled $X$-marginal} as
		\[
		\sigma_J^n \assign (S_J^n)_\sharp \mu_J^n / m_J^n \in \measp(Z).
		\]
		\item For $t > 0$, $x \in X$, let $k = \floor{tn}$ and let $J \in \partnk$ be the $\mu$-a.e.~unique composite cell in $\partnk$ such that $x \in X_J^n$. We will write
		\begin{align*}
		J_{t,x}^n & \assign J, &
		\xJn & \assign x_J^n, &
		S_{t,x}^n & \assign S_J^n, &
		\sigma_{t,x}^n & \assign \sigma_{J}^n.
		\end{align*}
		This will allow us to reference more easily between the continuum limit problem in fiber $x \in X$ and its corresponding family of discrete problems at finite scale $n$.
	\end{itemize}
\end{definition}

\begin{definition}[Discrete fiber problem]
	For each $n\in 2\N$, $t\in \R_+$ and $x\in X$, we define the following functional over $\meas_1(Z\times Y)$:
	\begin{align}\label{eq:Fn}
		F^n_{t,x}(\vecnu) & \assign \begin{cases}
			C^n_{t,x}(\vecnu) & \tn{if  $\vecnu\in\Pi(\sigma_{t,x}^n, \bmpi_{t,x}^n)$,} \\
			+ \infty & \tn{otherwise,}
		\end{cases}
	\end{align}
	where for $\eta<\infty$,
	\begin{align}
		C^n_{t,x}(\vecnu)
		& \assign
		\int_{Z\times Y} \left[\inner{\nabla_X c(\xJn, y)}{z}
		+
		\Delta^n(\xJn,z,y)  \right] \diff \vecnu(z,y)
		+
		n\,\veps^n \cdot \KL(\vecnu | \sigma_{t,x}^n\otimes \bmpi_{t,x}^n),
		\label{eq:Cn}
		\\
		\intertext{and if $\eta =\infty$,}
		C^n_{t,x}(\vecnu)
		& \assign
		\KL(\vecnu | \sigma_{t,x}^n\otimes \bmpi_{t,x}^n)
		+
		\frac{1}{n\,\veps^n}
		\int_{Z\times Y}
		\left[\inner{\nabla_X c(\xJn, y)}{z} + 	\Delta^n(\xJn,z,y)  \right] \diff \vecnu(z,y),
		\label{eq:CnInfty}
		\end{align}
		with
		\begin{equation}
		\Delta^n(x,z,y)
		\assign
		n \cdot [ c^n(x + z/n, y) - c(x, y)
		-
		\inner{\nabla_X c(x, y)}{z/n} ].
		\label{eq:Deltan}
		\end{equation}	
\end{definition}
We now establish equivalence between \eqref{eq:ExampleCellProblem} and minimizing \eqref{eq:Fn}.
\begin{proposition}[Domain decomposition algorithm generates a minimizer of $F_{t,x}^n$]
	\label{prop:DomDecGeneratesMinimizerFn}
	Let $n \in 2\N$, $t > 0$, $x \in X$, $k = \lfloor t\,n \rfloor$ and $J = J_{t,x}^n$.
	Then problem \eqref{eq:ExampleCellProblem} is equivalent to minimizing $F^n_{t,x}$, \eqref{eq:Fn}, over $\measp(Z \times Y)$ in the sense that the latter is obtained from the former by a coordinate transformation, a positive re-scaling and subtraction of constant terms. The minimizers $\pi^{n,k}_J$ of \eqref{eq:ExampleCellProblem} are in one-to-one correspondence with minimizers $\bmnu^n_{t,x}\in \measp(Z \times Y)$ of \eqref{eq:Fn} via the bijective transformation
	\begin{align}\label{eq:MinimizerFnDomDec}
		\bmnu^n_{t,x}
		\assign
		(S_{J}^n, \id)_{\sharp} \pi^{n,k}_{J}/m_{J}^n.
	\end{align}
\end{proposition}
Note that $m^n_J>0$ is a consequence of the fundamental property of basic partitions that each cell carries non-zero mass (Definition \ref{def:Partitions}).

\begin{proof}
We subsequently apply equivalent transformations to \eqref{eq:ExampleCellProblem} to turn it into \eqref{eq:Fn}, while keeping track of the corresponding transformation of minimizers. We start with the case $\eta<\infty$.

First, we multiply the objective of \eqref{eq:ExampleCellProblem} by $n$ and re-scale the mass of $\pi$ by $1/m^n_J$ such that it becomes a probability measure. We obtain that \eqref{eq:ExampleCellProblem} is equivalent to
\begin{align}
	\label{eq:ExampleCellProblemB}
	\inf\left\{ \int_{X_J^n \times Y} (n \cdot c^n)\,\diff \pi + n \cdot \veps^n \cdot \KL(\pi|\tfrac{\mu^n_J}{m^n_J} \otimes \nu^n)
	\,\middle|\, \pi \in \Pi(\tfrac{\mu^n_J}{m^n_J},\bmpi^n_{t,x}) \right\},
\end{align}
where we used that $\KL(\cdot|\cdot)$ is positively 1-homogeneous under joint re-scaling of both arguments and that $\nu^{n,k}_J/m^n_J=\bmpi^n_{t,x}$ by \eqref{eq:pi_discrete} (and the relation of $t$, $x$, $n$, $k$ and $J$). Minimizers of \eqref{eq:ExampleCellProblemB} are obtained from minimizers $\pi$ of \eqref{eq:ExampleCellProblem} as $\pi/m^n_J$.

Second, we transform the cell $X^n_J$ to the reference cell $Z$ via the map $S^n_J$.
For the transport term in \eqref{eq:ExampleCellProblemB} we find
\begin{align*}
	\int_{X_J^n \times Y} (n \cdot c^n)\,\diff \pi = \int_{Z \times Y} (n \cdot c^n) \circ (S^n_J,\id)^{-1}\,\diff (S^n_J,\id)_\sharp \pi.
\end{align*}
Using that $(S^n_J,\id)$ is a homeomorphism one gets that
\begin{align*}
	\RadNikD{(S^n_J,\id)_\sharp \pi}{\left((S^n_J)_\sharp \tfrac{\mu^n_J}{m^n_J} \otimes \bmpi^n_{t,x}\right)} \circ (S^n_J,\id) = \RadNikD{\pi}{\left(\tfrac{\mu^n_J}{m^n_J} \otimes \bmpi^n_{t,x}\right)}
	\qquad \tn{$\left(\tfrac{\mu^n_J}{m^n_J} \otimes \bmpi^n_{t,x}\right)$-almost everywhere.}
\end{align*}
With this we can transform the entropy term of \eqref{eq:ExampleCellProblemB} to
\begin{align*}
	\KL(\pi|\tfrac{\mu^n_J}{m^n_J} \otimes \nu^n) = \KL((S^n_J,\id)_\sharp\pi|(S^n_J)_\sharp\tfrac{\mu^n_J}{m^n_J} \otimes \nu^n)
	= \KL((S^n_J,\id)_\sharp\pi|\sigma^n_J \otimes \nu^n).
\end{align*}
Finally, using once more that $(S^n_J,\id)$ is a homeomorphism one finds
\begin{align*}
	\left[ \pi \in \Pi(\tfrac{\mu^n_J}{m^n_J},\bmpi^n_{t,x}) \right] \qquad \Leftrightarrow \qquad
	\left[ (S^n_J,\id)_\sharp \pi \in \Pi(\sigma^n_J,\bmpi^n_{t,x}) \right].
\end{align*}
Consequently, \eqref{eq:ExampleCellProblemB} is equivalent to
\begin{align}
	\label{eq:ExampleCellProblemC}
	\inf\left\{ \int_{Z \times Y} (n \cdot c^n) \circ (S^n_J,\id)^{-1}\,\diff \pi + n \cdot \veps^n \cdot \KL(\pi|\sigma^n_J \otimes \nu^n)
	\,\middle|\, \pi \in \Pi(\sigma^n_J,\bmpi^n_{t,x}) \right\},
\end{align}
with minimizers to the latter obtained from minimizers $\pi$ of the former as $(S^n_J,\id)_\sharp \pi$. 

Third, we subtract a constant term from the transport part of \eqref{eq:ExampleCellProblemC}. Recalling \eqref{eq:Deltan} one quickly finds that
\begin{multline*}
	\int_{Z\times Y} \left[\inner{\nabla_X c(\xJn, y)}{z}
		+
		\Delta^n(\xJn,z,y)  \right] \diff \vecnu(z,y)
	= \\
	\int_{Z \times Y} (n \cdot c^n) \circ (S^n_J,\id)^{-1}\,\diff \lambda
	- \int_{Z \times Y} n \cdot c(\xJn,y) \diff \lambda(z,y)
\end{multline*}
where the left hand side is the transport term of \eqref{eq:Cn}, the first term on the right hand side is the transport term in \eqref{eq:ExampleCellProblemC} and the second term is constant for all $\lambda \in \Pi(\sigma^n_J,\bmpi^n_{t,x})$ and thus has no influence on the minimization.

Fourth, we subtract constant parts of the entropy term. Recall that $\bmpi^n_{t,x}=\nu^{n,k}_J/m^n_J$ with $\sum_{J' \in \partnk} \nu^{n,k}_{J'}=\nu^n$ and all partial marginals are non-negative. This implies that $\bmpi^n_{t,x} \ll \nu^n$ with the density $\RadNik{\bmpi^n_{t,x}}{\nu^n}$ lying in $[0,1/m^n_J]$ $\nu^n$-almost everywhere.
Consequently, if $\lambda \ll \sigma^n_J \otimes \bmpi^n_{t,x}$ then $\lambda \ll \sigma^n_J \otimes \nu^n$ and the densities satisfy
\begin{align*}
	\RadNikD{\lambda}{\sigma^n_J \otimes \nu^n}(z,y) = \RadNikD{\lambda}{\sigma^n_J \otimes \bmpi^n_{t,x}}(z,y)
	\cdot \RadNikD{\bmpi^n_{t,x}}{\nu^n}(y)
	\qquad \tn{$\sigma^n_J(z)$-$\nu^n(y)$-almost everywhere.}
\end{align*}
Since $\proj_Y \lambda = \bmpi^n_{t,x}$ for all feasible $\lambda \in \Pi(\sigma^n_J,\bmpi^n_{t,x})$, one also has [$\lambda \ll \sigma^n_J \otimes \nu^n$] $\Rightarrow$ [$\lambda \ll \sigma^n_J \otimes \bmpi^n_{t,x}$] and the same relation between the densities.
Using this one finds that when either of the two entropic terms in \eqref{eq:Cn} or \eqref{eq:ExampleCellProblemC} is finite, so is the other one where one has the relation
\begin{align*}
	\KL(\lambda|\sigma^n_J \otimes \nu^n) = \KL(\lambda | \sigma^n_J \otimes \bmpi^n_{t,x}) + \KL(\bmpi^n_{t,x}|\nu^n).
\end{align*}
Here, the second term on the right hand side is finite (due to the bound on the density $\RadNik{\bmpi^n_{t,x}}{\nu^n}$) and does not depend on $\lambda$. Hence, the entropic terms in \eqref{eq:Cn} and \eqref{eq:ExampleCellProblemC} are identical up to a constant and in conclusion, for $\eta<\infty$, both minimization problems are equivalent with the prescribed relation between minimizers.
The adaption to the case $\eta=\infty$ is trivial since \eqref{eq:CnInfty} is just a positive re-scaling of \eqref{eq:Cn}.
\end{proof}

\begin{remark}
	\label{rem:MomentumFieldFromLambda}
	For $\bmnu_{t,x}^n$ a minimizer of $F_{t,x}^n$ constructed as in \eqref{eq:MinimizerFnDomDec}, the discrete momentum field disintegration $\bmomega_{t,x}^n$ \eqref{eq:omega_discrete} can be written in terms of $\bmnu_{t,x}^n$ as:
	\begin{multline}
		(\bmomega^n_{t,x})_\ell
		=
		\proj_Y (\bmnu^n_{t,x}\restr Z^\ell_+ \times Y)
		-
		\proj_Y (\bmnu^n_{t,x}\restr  Z^\ell_- \times Y)
		\quad
		\text{for } \ell = 1,...,d,
		\quad
		\\
		\tn{with}\quad
		Z^\ell_\pm = \{z \in Z \mid \pm z_\ell > 0\}.
		\label{eq:MomentumFieldFromLambda}
	\end{multline}
	To see this, fix $J=J_{t,x}^n$. Then, for each $b\in \bitset \assign \{-1,+1\}^d$ define $Z_b \assign \{z\in Z \mid \sign(z_\ell) = b_\ell \tn{ for all } \ell = 1,...,d\}$. Further define $i(J,b)$ as the basic cell in composite cell $J$ whose center lies at $x_J^n + b/2n$. Then
	\begin{align}
		\proj_Y(\bmnu_{t,x}^n \restr Z_+^\ell \times Y)
		-
		\proj_Y(\bmnu_{t,x}^n \restr Z_-^\ell \times Y)
		&=
		\sum_{b\in \bitset}
		b_\ell \cdot
		\proj_Y(\bmnu_{t,x}^n \restr Z_b \times Y)
		=
		\sum_{b\in \bitset}
		b_\ell \cdot
		\iter{\rho_{i(J,b)}}{n,k},
	\end{align}
	which is precisely the $\ell$-th component in \eqref{eq:omega_discrete}.
\end{remark}

\subsection{Limit fiber problems and problem gluing}
\label{sec:GluedProblems}
In this section we state the expected limit of the discrete fiber problems \eqref{eq:Fn} as $n\to \infty$.
For this we need a sufficiently regular sequence of first marginals $(\sigma^n_{t,x})_n$, which will be dealt with in the first part of this section (Definition \ref{def:regular_discretization}, Assumption \ref{assumption:RegularDiscretization}, Lemma \ref{lem:RegularityMuN}).
Sufficient regularity of the second marginal constraint will be provided by the $\WY$-convergence of $(\bmpi^n)_n$ (Assumption \ref{assumption:PitConvergence}).
The conjectured limit problem is introduced in Definition \ref{def:LimitFiberProblem}.
Instead of proving $\Gamma$-convergence on the level of single fibers, we first `glue' the problems together (Definition \ref{def:GluedProblems}) along $t \in \R_+$ and $x \in X$ and then establish $\Gamma$-convergence for the glued problems (Proposition \ref{prop:GammaConvergence}).
This avoids issues with measurability and the selection of convergent subsequences.
Finally, from this we can deduce the convergence of the momenta to a suitable limit (Proposition \ref{prop:MomentumConvergence}).

\begin{definition}\label{def:regular_discretization}
	We say that $(\mu^n)_n$ is a \emph{regular discretization sequence for the $X$-marginal} if there is some $\sigma \in \prob(Z)$ such that for $\Lebesgue\otimes \mu$ almost all $(t,x) \in \R_+ \times X$ the sequence $(\sigma_{t,x}^n)_n$ converges weak* to $\sigma$ and $\sigma$ does not give mass to any coordinate axis, i.e.,	
	\begin{equation}
	\label{eq:SigmaNoMassCoordinateAxis}
	\sigma(\{z\in Z \mid z_\ell = 0\}) = 0
	\qquad
	\text{for } \ell = 1,...,d.
	\end{equation}
\end{definition}

\begin{assumption}\label{assumption:RegularDiscretization}
	From now on, we assume that $(\mu^n)_n$ is a regular discretization sequence.
\end{assumption}

\begin{remark}\label{remark:AdaptiveSigma}
	More generally one could consider the scenario where the limit of $(\sigma_{t,x}^n)_n$ depends on $x$, which could be useful for describing adaptive discretization schemes. For simplicity, this article is restricted to the constant case.
\end{remark}

\begin{lemma}[Regularity of discretization schemes]
\label{lem:RegularityMuN}
Prototypical choices for $\mu^n$ are:
\begin{enumerate}[label=(\roman*)]
	\item Collapsing all the mass within each basic cell to a Dirac at its center, $\mu^n=\sum_{i \in I^n} m_i^n \delta_{x_i^n}$. One obtains $\sigma=2^{-d} \sum_{b \in \{-1,+1\}^d} \delta_{b/2}$.
		\label{item:MuNDirac}
	\item Using the measure $\mu$ itself, without discretization, $\mu^n=\mu$. One obtains $\sigma=2^{-d}\cdot \Lebesgue \restr Z$.
		\label{item:MuNSelf}
	\item At every $n \in 2\N$ we collapse the mass of $\mu$ onto Diracs on a regular Cartesian grid such that every basic cell contains a sub-grid of $s^n$ points along each dimension, for a sequence $(s^n)_n$ in $\N$ with $ s^n \to +\infty$. One obtains $\sigma=2^{-d} \cdot \Lebesgue \restr Z$. A related refined discretization scheme was considered in \cite[Section 5.3]{BenamouPolarDomainDecomposition1994} where $n$ was kept fixed but $s^n$ was sent to $+\infty$ and it was shown that the sequence of fixed-points of the algorithm converges to the globally optimal solution.
		\label{item:MuDiracFine}
\end{enumerate}
	For $\mu \ll \Lebesgue$, the above schemes yield regular discretization sequences in the sense of Definition \ref{def:regular_discretization}.
\end{lemma}
The proof, reported in Appendix \ref{sec:RegularityMuN}, is based on the Lebesgue differentiation theorem \cite[Theorem 7.10]{RudinRealAndComplexAnalysis} for $L^1$ functions and leverages the fact that $\mu$ is Lebesgue absolutely continuous.

\begin{definition}[Limiting fiber problem]
	\label{def:LimitFiberProblem}
	For each $t\in \R_+$, $x\in X$ we define the following functional over $\meas_1(Z\times Y)$:
	\begin{align}
	\label{eq:LimitFiberProblem}
	F_{t,x}(\vecnu) & \assign \begin{cases}
	C_{t,x}(\vecnu) & \tn{if } \vecnu\in\Pi(\sigma, \bmpi_{t,x}), \\
	+ \infty & \tn{otherwise,}
	\end{cases}
	\intertext{where}
	C_{t,x}(\vecnu) & \assign
	\begin{cases}
	\displaystyle
	\int_{Z\times Y}
	\inner{\nabla_X c(x, y)}{z} \, \diff \vecnu(z,y)
	+
	\eta \cdot \KL(\vecnu | \sigma\otimes \bmpi_{t,x})
	& \tn{if $\eta <\infty$,} \\
	\KL(\vecnu | \sigma\otimes \bmpi_{t,x}) & \tn{if $\eta =\infty$.}
	\end{cases}
	\end{align}
\end{definition}

\begin{definition}[Glued problems (discrete and limiting)]
	\label{def:GluedProblems}
	Fix $T>0$. We define
	\begin{align}
	\nuset_T \assign \left\{ \bmnu \in \measp([0,T] \times X \times Z \times Y)
	\mid
	\projTX \bmnu = (\Lebesgue \restr [0,T]) \otimes \mu \right\},
	\end{align}
	where $\projTX$ takes non-negative measures on $\R_+ \times X \times Z \times Y$ to their marginal on $\R_+ \times X$ (cf.~Section \ref{sec:Notation}).
	In particular, any $\bmnu \in \nuset_T$ can be disintegrated with respect to $\Lebesgue \otimes \mu$, i.e.~there is a measurable family of probability measures $(\bmnu_{t,x})_{t,x}$  such that $\bmnu = (\Lebesgue \restr [0,T]) \otimes \mu \otimes \bmnu_{t,x}$ and
	\begin{align*}
	\int_{[0,T] \times X \times Z\times Y} \phi(t,x,z,y)\,\diff \bmnu(t,x,z,y)
	= \int_{[0,T] \times X} \int_{Z\times Y} \phi(t,x,z,y)\,\diff \bmnu_{t,x}(z,y)\,\diff \mu(x)\,\diff t
	\end{align*}
	for all measurable $\phi$ from $[0,T] \times X \times Z \times Y \to \R_+$ (see Remark \ref{rem:Disintegration}).

	For $\bmnu \in \nuset_T$, $n \in 2\N$ we define the glued discrete and limiting functionals
	\begin{align}
		\label{eq:FGlued}
		F^n_T(\bmnu) & \assign \int_0^T \int_{X} F^n_{t,x}(\bmnu_{t,x})\,\diff \mu(x)\,\diff t, &
		F_T(\bmnu) & \assign \int_0^T \int_{X} F_{t,x}(\bmnu_{t,x})\,\diff \mu(x)\,\diff t.
	\end{align}
\end{definition}
The finite time horizon $T$ is necessary since otherwise the infima of the glued functionals \eqref{eq:FGlued} might be infinity.

\begin{remark}
	\label{rem:DiscreteGluedMinimizer}
	For any $n \in 2\N$, $T >0$ a minimizer $\bmnu^n \in \nuset$ for $F^n_T$ can be obtained via Proposition \ref{prop:DomDecGeneratesMinimizerFn} (and hence via the domain decomposition algorithm) by gluing together discrete fiber-wise minimizers $\bmnu_{t,x}^n$ of \eqref{eq:Fn} given by \eqref{eq:MinimizerFnDomDec} to obtain $\bmnu^n \assign (\Lebesgue \restr [0,T]) \otimes \mu \otimes \bmnu^n_{t,x}$ (see Remark \ref{rem:Disintegration}). The obtained $\bmnu^n$ clearly lies in $\nuset_T$ and minimizes $F^n_T$ because each $\bmnu_{t,x}^n$ minimizes the fiberwise functional $F_{t,x}^n$.
	Due to the discreteness at scale $n$, only a finite number of minimizers must be chosen (one per discrete time-step and composite cell) and thus no measurability issues arise.
\end{remark}

\begin{proposition}\label{prop:GammaConvergence}
	Under Assumptions \ref{assumption:PitConvergence} and \ref{assumption:RegularDiscretization},
	for any $T>0$, $F^n_T$ $\Gamma$-converges to $F_T$ with respect to weak* convergence on $\nuset_T$ on the subsequence $n \in \nset$.
\end{proposition}
The proof is divided into liminf and limsup condition that are given in Sections \ref{sec:Liminf} and \ref{sec:Limsup}.

Based on this we can now extract cluster points from the minimizers to the discrete fiber problems that converge to minimizers of the limit fiber problems and also get convergence for the associated momenta.

\begin{proposition}[Convergence of fiber-problem minimizers and momenta]
	\label{prop:MomentumConvergence}
	Let $\bmnu^n\in \measp(\R_+\allowbreak \times X \times Z \times Y)$ be constructed from the discrete iterates $\iter{\pi}{n,k}$ as shown in Proposition \ref{prop:DomDecGeneratesMinimizerFn}.
	Under Assumptions \ref{assumption:PitConvergence} and \ref{assumption:RegularDiscretization} there is a subsequence $\nsetb \subset \nset \subset 2\N$ and a measure $\bmnu \in \measp(\R_+ \times X \times Z \times Y)$ such that for all $T \in (0,\infty)$,
	\begin{align*}
		\bmnu^n \restr ([0,T] \times X \times Z \times Y) \rightweaks \bmnu \restr ([0,T] \times X \times Z \times Y) \in \nuset_T
	\end{align*}
	on the subsequence $\nsetb$ and the limit is a minimizer of $F_T$.
	In addition, analogous to Remark \ref{rem:MomentumFieldFromLambda}, we introduce the limit momentum field $\bmomega \assign \Lebesgue \otimes \mu \otimes \bmomega_{t,x} \in \meas(\R_+ \times X \times Y)^d$ via
	\begin{multline*}
		(\bmomega_{t,x})_\ell
		\assign
		\proj_Y (\bmnu_{t,x}\restr Z^\ell_+ \times Y)
		-
		\proj_Y (\bmnu_{t,x}\restr  Z^\ell_- \times Y)
		\quad
		\text{for } \ell = 1,...,d,
		\\
		\text{with}\quad
		Z^\ell_\pm = \{z \in Z \mid \pm z_\ell > 0\}.
	\end{multline*}
	Then $\bmomega^n$, $n \subset \nsetb$, converges weak* to $\bmomega$ on any finite time interval $[0,T]$.
\end{proposition}
\begin{proof}	
	By weak* compactness, for any $T>0$ one can extract a subsequence $\nset' \subset \nset$ such that $(\bmnu^n \restr ([0,T] \times X \times Z \times Y))_{n \in \nset'}$ converges weak* to some $\bmnu_T \in \measp([0,T] \times X \times Z \times Y)$.
	By a diagonal argument, we can choose a further subsequence $\nsetb \subset \nset$ such that $(\bmnu^{n})_{n \in \nsetb}$ converges to some $\bmnu\in \meas_+(\R_+ \times X \times Z \times Y)$ when restricted to $[0,T] \times X \times Z \times Y$ for any $T>0$.
	By construction (Proposition \ref{prop:DomDecGeneratesMinimizerFn}, Remark \ref{rem:DiscreteGluedMinimizer}), $\bmnu^{n}$, $n \in \nsetb$, is a minimizer of $F^{n}_T$ for any choice of $T$. Thus, by $\Gamma$-convergence (Proposition \ref{prop:GammaConvergence}) $\bmnu$ is also a minimizer of $F_T$ for all choices of $T$.

	It remains to be shown that the construction of $\bmomega^{n}$ from $\bmnu^{n}$ (for the discrete and the limit case) is a weak* continuous operation.
	The fiber-wise construction can be written at the level of the whole measures as
	\begin{align*}
		(\bmomega^{n})_\ell = \proj_{\R_+ \times X \times Y} (\bmnu^{n} \restr (\R_+ \times X \times Z^\ell_+ \times Y))
		-
		\proj_{\R_+ \times X \times Y} (\bmnu^{n} \restr (\R_+ \times X \times Z^\ell_- \times Y))
	\end{align*}
	for $\ell = 1,...,d$. Marginal projection is a weak* continuous operation, and so is the addition (subtraction) of two measures. Let us therefore focus on the restriction operation. In general, restriction is not weak* continuous, but it is under our regularity assumptions on $\sigma^n_{t,x}$ and $\sigma$ (Section \ref{sec:DomDecNotation}, item \ref{item:MuN} and Assumption \ref{assumption:RegularDiscretization}). None of these measures carry mass on the boundaries between the $Z^\ell_{\pm}$ (and these sets are relatively open in $Z$).
	For simplicity, we will now show that under these conditions, [$\sigma_{t,x}^n \rightweaks \sigma$] $\Rightarrow$ [$\sigma_{t,x}^n \restr Z^\ell_{\pm} \rightweaks \sigma \restr Z^\ell_{\pm}$] for any $\ell \in \{1,\ldots,d\}$. The same argument (but with heavier notation) will then apply to the convergence of the restrictions of $\bmnu^n$.
	By weak* compactness we can select a subsequence such that
	\begin{align*}
		\sigma_{t,x}^n \restr Z^\ell_{\pm} \rightweaks \sigma_\pm
	\end{align*}
	for two measures $\sigma_\pm \in \measp(Z)$, 
	and by the Portmanteau theorem for weak convergence of measures \cite[Theorem 2.1]{Billingsley1999} we have
	\begin{align*}
		\sigma_{\pm}(Z^\ell_{\mp}) \leq \liminf_n (\sigma_{t,x}^n \restr Z^\ell_{\pm})(Z^\ell_{\mp}) = 0.
	\end{align*}
	Now observe
	\begin{align*}
		\sigma = \lim_{n \to \infty} \sigma_{t,x}^n = 
		\lim_{n \to \infty} ( \sigma_{t,x}^n \restr Z^\ell_{+} + \sigma_{t,x}^n \restr Z^\ell_{-})
		= \sigma_+ + \sigma_-
	\end{align*}		
	where in the second equality we used that $\sigma_{t,x}^n$ does not carry mass on the set $\{z_\ell=0\}$.
	Using that $\sigma$ carries no mass on $\{z_\ell=0\}$ we conclude that $\sigma_\pm = \sigma \restr Z^\ell_\pm$.
	This holds for any convergent subsequence and thus by weak* compactness the whole sequences of restrictions converge to $\sigma \restr Z_{\pm}^\ell$.
	As indicated, the same argument will apply to the restriction of $\bmnu^{n}$ to the sets $[0,T] \times X \times Z^\ell_{\pm} \times Y$ for any finite horizon $T \in (0,\infty)$.
\end{proof}

\subsection{Liminf condition}
\label{sec:Liminf}

We start by establishing that the transport cost contribution in $F^n_T$ converges to that of $F_T$. We do so by gathering all transport contributions in the fibers $F^n_{t,x}$ into a single integral, and likewise for $F_{t,x}$.
\begin{lemma}[Convergence of the transport cost] \label{lemma:convergence_objective_function}
Let $T>0$ and $(\bmnu^n)_n$ be a weak* convergent sequence in $\nuset_T$ with limit $\bmnu \in \nuset_T$. Then the transport part of the glued functional $F^n_T$ converges to that of $F_T$. More specifically,
\begin{multline}
	\label{eq:convergence_objective_function}
		\lim_{n\rightarrow \infty}
		\int_{[0,T]\times X \times Z \times Y}
		\left[\inner{\nabla_X c(\xJn, y)}{z} +  \Delta^n(\xJn,z,y)  \right] \diff \bmnu^n(t,x,z,y)
		\\
		=
		\int_{[0,T]\times X\times Z\times Y}
		 \inner{\nabla_X c(x,y)}{z}\,\diff \bmnu(t,x,z,y).
\end{multline}
\end{lemma}

\begin{proof}
	We have to verify that the following expression tends to zero:
	\begin{align}
		\label{eq:LiminfTransportTerm}
		&
		\int_{[0,T]\times X \times Z \times Y}
		\inner{\nabla_X c(\xJn, y) - \nabla_X c(x, y)}{z}
		\diff \bmnu^n(t,x,z,y)
		\nonumber \\
		+&
		\int_{[0,T]\times X \times Z \times Y}
		\inner{\nabla_X c(x, y)}{z}
		\diff (\bmnu^n - \bmnu)(t,x,z,y)
		\nonumber \\
		+&
		\int_{[0,T]\times X \times Z \times Y}
		\Delta^n(\xJn,z,y)
		\diff \bmnu^n(t,x,z,y).
	\end{align}
	Since $\nabla_X c(x,y)$ is uniformly continuous, the integrand in the first term converges to zero uniformly (since $\|\xJn - x\| \le \sqrt{d}/n$), and since the masses of $\bmnu^n$ are uniformly bounded, the integral goes to zero.
	The second term converges to zero by weak* convergence of $\bmnu^n$ to $\bmnu$.

	Recalling the definition of $\Delta^n$, \eqref{eq:Deltan}, and $c^n$, \eqref{eq:assumption_modified_cost}, the integrand in the third term is given by the function
	\begin{align*}
		(t,x,z,y) & \mapsto \Delta^n(\xJn,z,y)
		= n \cdot [ c^n(\xJn + z/n, y) - c(\xJn, y)
		- \inner{\nabla_X c(\xJn, y)}{z/n} ] \\
		& = f^n(\xJn + z/n, y) + n \cdot [ c(\xJn + z/n, y) - c(\xJn, y)
		- \inner{\nabla_X c(\xJn, y)}{z/n} ]
	\end{align*}
	The first term converges to zero uniformly by assumption on $f^n$ (cf. Section \ref{sec:DomDecNotation}, Item \ref{item:ModifiedCost}).
	For the second term, by the mean value theorem, there exists a point $\xi$ on the segment $[\xJn,\xJn+z/n]$ (depending on $t,x,z,y,n$) such that it can be written as
	\begin{align*}
	\inner{\nabla_X c(\xi, y) - \nabla_X c(\xJn, y)}{z},
	\end{align*}
	which converges to zero uniformly since $\|\xi-\xJn\|\leq \sqrt{d}/n$ and $\nabla_X c$ is uniformly continuous.
	Finally, this implies that the third term in \eqref{eq:LiminfTransportTerm} converges to zero since the integrand converges to zero uniformly and the masses of $\bmnu^n$ are uniformly bounded.
\end{proof}

\begin{lemma}[Liminf inequality]\label{lemma:gamma_liminf}
	Let $T>0$ and $(\bmnu^n)_{n\in \nset}$ be a weak* convergent sequence in $\nuset_T$ with limit $\bmnu \in \nuset_T$. Then
	\begin{equation}\label{eq:gamma_liminf}
		\liminf_{n\in \nset,\ n\rightarrow\infty} F^n_T(\bmnu^n) \geq F_T(\bmnu).
	\end{equation}
\end{lemma}

\begin{proof}
	By disintegration $\bmnu^n$ and $\bmnu$ can be written as
	\begin{align*}
		\bmnu^n & = (\Lebesgue \restr [0,T]) \otimes \mu \otimes \bmnu^n_{t,x}, &
		\bmnu & = (\Lebesgue \restr [0,T]) \otimes \mu \otimes \bmnu_{t,x}
	\end{align*}
	for suitable families $(\bmnu^n_{t,x})_{t,x}$ and $(\bmnu_{t,x})_{t,x}$.
	If the liminf is $+\infty$ there is nothing to prove. So we may limit ourselves to study subsequences with finite limit (and assume that we have extracted and relabeled such a sequence as $\nsetliminf \subset \nset$). 
	Unless otherwise stated, all limits in the proof are taken on this subsequence $\nsetliminf$, though we may not always state it to avoid overloading the notation.
	
	\smallskip
	\noindent \textit{Step 1: marginal constraints.}
	$F^n_T(\bmnu^n)$ can only be finite if $\bmnu^n_{t,x} \in \Pi(\sigma^n_{t,x},\bmpi^n_{t,x})$ for $\Lebesgue \otimes \mu$ almost all $(t,x) \in [0,T] \times X$. We find that this implies $\proj_Z \bmnu_{t,x}=\sigma$ for almost all $(t,x)$ by observing that for any $\phi \in \cont([0,T] \times X \times Z)$ one has
	\begin{align*}
	&\int_{[0,T]\times X} \int_{Z \times Y} \phi(t,x,z)  \,\diff\bmnu_{t,x}(z,y) \,\diff\mu(x)\,\diff t
	=
	\lim_{n \to \infty}
	\int_{[0,T]\times X} \int_{Z \times Y} \phi(t,x,z)  \,\diff\bmnu_{t,x}^n(z,y) \,\diff\mu(x)\,\diff t
	\\
	&=
	\lim_{n \to \infty}
	\int_{[0,T]\times X} \int_Z \phi(t,x,z)  \,\diff\sigma_{t,x}^n(z) \,\diff\mu(x)\,\diff t
	=
	\int_{[0,T]\times X} \int_Z \phi(t,x,z)  \,\diff\sigma(z) \,\diff\mu(x)\,\diff t
	\end{align*}
	where the first equality follows from $\bmnu^n \rightweaks \bmnu$ and the third one from dominated convergence since the inner integral converges pointwise almost everywhere (see Assumption \ref{assumption:RegularDiscretization}) and is uniformly bounded.
	
	The argument for $\proj_Y \bmnu_{t,x}=\bmpi_{t,x}$ is slightly more involved since Assumption \ref{assumption:PitConvergence} only provides that $\bmpi^n_t \rightweaks \bmpi_t$ for almost all $t$, but pointwise weak* convergence does not necessarily hold at the level of disintegrations in $(t,x)$.
	For any $\phi \in \cont([0,T] \times X \times Y)$ one has
	\begin{align*}
	&\int_{[0,T]\times X} \int_{Z \times Y} \phi(t,x,y)  \,\diff\bmnu_{t,x}(z,y) \,\diff\mu(x)\,\diff t
	=
	\lim_{n \to \infty}
	\int_{[0,T]\times X} \int_{Z \times Y} \phi(t,x,y)  \,\diff\bmnu_{t,x}^n(z,y) \,\diff\mu(x)\,\diff t
	\\
	&=
	\lim_{n \to \infty}
	\int_{[0,T]\times X} \int_Y \phi(t,x,y)  \,\diff\bmpi_{t,x}^n(y) \,\diff\mu(x)\,\diff t
	=
	\lim_{n \to \infty}
	\int_{[0,T]} \int_{X \times Y} \phi(t,x,y)  \,\diff\bmpi_{t}^n(x,y) \,\diff t
	\\
	&= \int_{[0,T]} \int_{X \times Y} \phi(t,x,y)  \,\diff\bmpi_{t}(x,y) \,\diff t
	=\int_{[0,T]\times Y} \int_Z \phi(t,x,y)  \,\diff\bmpi_{t,x}(y) \,\diff\mu(x)\,\diff t
	\end{align*}	
	where we argue again via dominated convergence from the second to the third line.
	Hence, $\bmnu_{t,x} \in \Pi(\sigma, \bmpi_{t,x})$ for almost all $(t,x)$.

	\smallskip
	\noindent \textit{Step 2: transport contribution.}
	The transport cost contributions of $F^n_T(\bmnu^n)$ converge to that of $F_T(\bmnu)$ by Lemma \ref{lemma:convergence_objective_function}.
	
	\noindent \textit{Step 3: entropy contribution.}
	Assume first $\eta<\infty$.
	Introducing the measures
	\begin{equation}
		\bmnu^n_\otimes
		\assign
		(\Lebesgue \restr [0,T]) \otimes \mu \otimes \sigma_{t,x}^n \otimes \bmpi_{t,x}^n
		\quad \text{ and } \quad
		\bmnu_\otimes
		\assign
		(\Lebesgue \restr [0,T]) \otimes \mu \otimes \sigma \otimes \bmpi_{t,x},
	\end{equation}
	the entropic terms of $F^n_T$ and $F_T$ can be written as $n\, \veps^n \KL(\bmnu^n \mid \bmnu^n_\otimes)$ and $\eta \KL(\bmnu \mid \bmnu_\otimes)$ respectively, since
	\begin{align*}
	&\int_{[0,T]\times X}
	\KL(\bmnu_{t,x}^n \mid \sigma_{t,x}^n \otimes \bmpi_{t,x}^n) \diff \mu(x) \diff t
	\\
	& \qquad =
	\int_{[0,T]\times X\times Z \times Y}
	\varphi\left(
		\RadNikD{\bmnu^n_{t,x}(z,y)}{(\sigma_{t,x}^n \otimes \bmpi^n_{t,x})}
	\right)
	\diff \sigma_{t,x}^n(z) \diff \bmpi^n_{t,x}(y)
	\diff \mu(x) \diff t
	\\
	& \qquad =
	\int_{[0,T]\times X\times Z \times Y}
	\varphi\left(
	\RadNikD{\bmnu^n(t,x,z,y)}{\bmnu^n_\otimes}
	\right)
	\diff \bmnu^n_\otimes(t,x,z,y)
	=
	\KL(\bmnu^n \mid \bmnu^n_\otimes),
	\end{align*}
	because $\bmnu^n$ and $\bmnu^n_\otimes$ have the same marginals in time and $X$. Analogously, the entropic contribution in $F_T$ is $\eta\KL(\bmnu \mid \bmnu_\otimes)$. Thus, by joint lower semicontinuity of $\KL$ (where we use $\bmnu^n_\otimes \rightweaks \bmnu_\otimes$, which follows from Assumptions \ref{assumption:PitConvergence} and \ref{assumption:RegularDiscretization}), convergence of $n\veps^n$ to $\eta$ and the fact that we selected a subsequence with finite limit (such that $\KL(\bmnu^n \mid \bmnu^n_\otimes)$ is uniformly bounded) we find
	\begin{equation}
		\liminf_{n\in \nsetliminf,\, n\rightarrow\infty}
		n\,\veps^n\,\KL(\bmnu^n \mid \bmnu^n_\otimes)
		\ge
		\eta \KL(\bmnu \mid \bmnu_\otimes).
	\end{equation}
	This shows that, for any subsequence $\nsetliminf$ with finite limit, $\liminf_{n\in \nsetliminf,\, n\rightarrow\infty} F^n_T(\bmnu^n) \geq F_T(\bmnu)$, so
	
	\begin{equation}
		\liminf_{n\in \nset,\, n\rightarrow\infty}
		F^n_T(\bmnu^n)
		= 
		\inf_{\nsetliminf \subset \nset} \liminf_{n\in \nsetliminf,\, n\rightarrow\infty}
		F^n_T(\bmnu^n)
		\geq 
		F_T(\bmnu).
	\end{equation}
	This concludes the proof for $\eta<\infty$. The case $\eta=\infty$ is analogous.
\end{proof}

\subsection{Limsup condition}
\label{sec:Limsup}

\begin{lemma}[Limsup inequality]\label{lemma:gamma_limsup}
	Let $T>0$, $\bmnu \in \nuset_T$. Then, there exists a sequence $(\bmnu^n)_{n\in\nset}$ in $\nuset_T$,
	converging weak* to $\bmnu$ such that
	\begin{equation}\label{eq:gamma_limsup}
		\limsup_{n\in\nset,\ n\rightarrow\infty} F^n_T(\bmnu^n) \leq F_T(\bmnu).
	\end{equation}
\end{lemma}

\begin{proof}
Since $\bmnu \in \nuset_T$ it can be disintegrated into $\bmnu = \Lebesgue \otimes \mu \otimes \bmnu_{t,x}$ for a family $(\bmnu_{t,x})_{t,x}$ in $\prob(Z \times Y)$. We may assume that $F_T(\bmnu)<\infty$, as otherwise there is nothing to prove. Hence, $\bmnu_{t,x} \in \Pi(\sigma, \bmpi_{t,x})$ for $\Lebesgue \otimes \mu$-almost all $(t,x) \in [0,T] \times X$. We will build our recovery sequence by gluing, setting $\bmnu^n \assign \Lebesgue \otimes \mu \otimes \bmnu_{t,x}^n$ where we construct the fibers $\bmnu_{t,x}^n \in \Pi(\sigma_{t,x}^n, \bmpi_{t,x}^n)$ by tweaking the measures $\bmnu_{t,x}$.

\def\gammaol{\ol{\gamma}}
\smallskip
\noindent
\textit{Step 1: construction of the recovery sequence.}
For every $n \in \nset$, let $(\gamma^n_{t,x})_{t,x}$ be a family of measures in $\prob(Y \times Y)$ where $\gamma^n_{t,x}\in \Pi(\bmpi_{t,x}^n, \bmpi_{t,x})$ is an optimal transport plan for $\WoY(\bmpi_{t,x}^n, \bmpi_{t,x})$. (Measurability of this family can be obtained, for instance, by disintegration of a minimizer of \eqref{eq:WYAlternative}.)
Likewise, let $(\gammaol^n_{t,x})_{t,x}$ be a (measurable) family of measures in $\prob(Z \times Z)$ where $\gammaol^n_{t,x}\in \Pi(\sigma_{t,x}^n, \sigma)$ is an optimal transport plan for $W_X(\sigma_{t,x}^n, \sigma_{t,x})$.
We then define $\bmnu_{t,x}^n$ for $n\in \nset$ by integration against $\phi\in \cont(Z\times Y)$ via
\begin{equation}\label{eq:TestLambdaNReconstructing}
\int_{Z\times Y}
\phi(z,y)\,\diff \bmnu_{t,x}^n (z,y)
\assign
\int_{Z^2 \times Y^2}
\phi(z,y)\,
\diff (\gammaol_{t,x}^n)_{z'}(z)\,
\diff (\gamma_{t,x}^n)_{y'}(y)\,
\diff \bmnu_{t,x}(z',y'),
\end{equation}
where $(\gammaol_{t,x}^n)_{z'}$ denotes the disintegration of $\gammaol_{t,x}^n$ with respect to its second marginal (namely $\sigma$) at point $z'$ (and analogously for $(\gamma_{t,x}^n)_{y'}$).

\smallskip
\noindent
\textit{Step 2: correct marginals along the recovery sequence.}
Let us check that $\bmnu_{t,x}^n \in \Pi(\sigma_{t,x}^n, \bmpi_{t,x}^n)$. First, for any $\phi\in \cont(Y)$,
\begin{align*}
	\int_{Z\times Y} \phi(y) \diff \bmnu_{t,x}^n(z,y)
	&=
	\int_{Z^2 \times Y^2}
	\phi(y)
	\diff (\gammaol_{t,x}^n)_{z'}(z)
	\diff (\gamma_{t,x}^n)_{y'}(y)
	\diff \bmnu_{t,x}(z',y')
	\\
	&=
	\int_{Y^2 }
	\phi(y)
	\diff (\gamma_{t,x}^n)_{y'}(y)
	\diff \bmpi_{t,x}(y')
	=
	\int_{Y^2 }
	\phi(y)
	\diff \gamma_{t,x}^n(y, y')
	=
	\int_{Y}
	\phi(y) \diff \bmpi_{t,x}^n(y)
\end{align*}
where we used that $(\gammaol_{t,x}^n)_{z'}$ is a probability measure.
The same argument applies to the $Z$-marginal.

\smallskip
\noindent
\textit{Step 3: convergence of the recovery sequence.}
Now we show that $\bmnu^n \rightweaks \bmnu$ for $n\in \nset$. For this we will use the Kantorovich--Rubinstein duality for the Wasserstein-1 distance \eqref{eq:KantRubin}
\begin{align}
	\label{eq:LimSupRubinsteinA}
	W_{[0,T] \times X \times Z \times Y}(\bmnu^n,\bmnu) = \sup_{\phi \in \Lip_1} \int_{[0,T] \times X \times Z \times Y} \phi\,\diff(\bmnu^n - \bmnu).
\end{align}
where we abbreviate $\Lip_1\assign \Lip_1([0,T] \times X \times Z \times Y)$.
In the following, let $\phi \in \Lip_1$. We find
\begin{align}
	& \left| \int_{[0,T] \times X \times Z \times Y} \phi\,\diff(\bmnu^n  -\bmnu) \right|
	\nonumber \\
	\leq & \int_{[0,T] \times X \times (Z \times Y)^2} \left| \phi(t,x,z,y)-\phi(t,x,z',y')\right|
	\diff(\gammaol_{t,x}^n)_{z'}(z)
	\diff (\gamma_{t,x}^n)_{y'}(y)
	\diff \bmnu_{t,x}(z',y') \diff \mu(x) \diff t
	\label{eq:LimSupRubinsteinB}
\end{align}

Using the Lipschitz continuity of $\phi$ we can bound
\begin{multline*}
	\left| \phi(t,x,z,y)-\phi(t,x,z',y')\right|
	\leq \left| \phi(t,x,z,y)-\phi(t,x,z',y)\right| + \left| \phi(t,x,z',y)-\phi(t,x,z',y')\right| \\
	\leq \|z-z'\| + \|y-y'\|.
\end{multline*}
Thus, we can continue
\begin{align*}
	\tn{\eqref{eq:LimSupRubinsteinB}} & \leq 
	\int_{[0,T] \times X \times Z^2} \|z-z'\| \diff(\gammaol_{t,x}^n)_{z'}(z)\,
	\diff \sigma(z')\,\diff \mu(x)\,\diff t \\
	& \qquad +
	\int_{[0,T] \times X \times Y^2} \|y-y'\| \diff(\gamma_{t,x}^n)_{y'}(y)\,
	\diff \bmpi_{t,x}(y')\,\diff \mu(x)\,\diff t \\
	& \leq 
	\int_{[0,T] \times X} W_{Z}(\sigma^n_{t,x},\sigma)\,
	\diff \mu(x)\,\diff t
	+
	\int_{[0,T] \times X} \WoY(\bmpi^n_{t,x},\bmpi_{t,x})\,
	\diff \mu(x)\,\diff t \\
	& =
	\int_{[0,T] \times X} W_{Z}(\sigma^n_{t,x},\sigma)\,
	\diff \mu(x)\,\diff t
	+
	\int_{[0,T]} \WY(\bmpi^n_{t},\bmpi_{t})\,\diff t,
\end{align*}
where we have used optimality of the plans $\gammaol^n_{t,x}$ and $\gamma^n_{t,x}$.
Using Assumptions \ref{assumption:PitConvergence} and \ref{assumption:RegularDiscretization} and dominated convergence (where we exploit that $Z$ and $Y$ are compact, hence $W_Z$ and $\WY$ are bounded), we find that this tends to zero as $\nset\ni n \to \infty$.
Plugging this into \eqref{eq:LimSupRubinsteinA}, we find that $W_{[0,T] \times X \times Z \times Y}(\bmnu^n,\bmnu) \to 0$ and since Wasserstein distances metrize weak* convergence on compact spaces, we obtain $\bmnu^n \rightweaks \bmnu$ for $n\in\nset$.

\smallskip
\noindent
\textit{Step 4: $\limsup$ inequality.}
Now we have to distinguish between different behaviors of $(n \cdot \veps^n)_n$.
\begin{itemize}
	\item{} \textit{[$\eta=0$, $\veps^n = 0$ for all $n\in\nset$, with only a finite number of exceptions]}
	The exceptions have no effect on the lim sup, hence we may skip them.
	By Lemma \ref{lemma:convergence_objective_function} the transport contribution to the functional converges, and so we obtain that 
	$\displaystyle
	\lim_{n\in\nset, n\rightarrow\infty} F^n_T(\bmnu^n) = F_T(\bmnu)$.

	\item{} \textit{[$\eta>0$]} We have that $\veps^n>0$ for all $n$ up to a finite number of exceptions, which we may again skip.
	In this case, the limit cost has an entropic contribution, and thus for a.e.~$t \in [0,T]$ and $\mu$-a.e.~$x \in X$, $\bmnu_{t,x}$ has a density with respect to $\sigma \otimes \bmpi_{t,x}$, that we denote by $u_{t,x}$. Then, as we will show below, $\bmnu_{t,x}^n$ also has a density with respect to $\sigma_{t,x}^n \otimes \bmpi_{t,x}^n$, that is given by:
	\begin{equation}\label{eq:ReconstructedDensityn}
	u^n_{t,x}(z,y)
	\assign
	\int_{Z\times Y}u_{t,x}(z', y')\diff (\gammaol_{x,t}^n )_z(z')\diff(\gamma_{x,t}^n )_y(y'),
	\end{equation}
	where we use again the transport plans $\gammaol_{x,t}^n \in \Pi(\sigma^n_{t,x},\sigma)$ and $\gamma^n_{t,x} \in \Pi(\bmpi^n_{t,x},\bmpi_{t,x})$ and this time their disintegrations against the first marginals. 
	Let us prove that $u^n_{t,x}$ is indeed the density of $\bmnu_{t,x}^n$ with respect to $\sigma_{t,x}^n \otimes \bmpi_{t,x}^n$:
	\begin{align*}
		&\int_{Z\times Y}
		\phi(z,y)\,u_{t,x}^n(z,y)\,
		\diff \sigma_{t,x}^n(z)\, \diff \bmpi_{t,x}^n(y)
		\\
		&\quad=
		\int_{Z^2 \times Y^2}
		\phi(z,y)\,u_{t,x}(z', y')\,
		\underbrace{
			\diff (\gammaol_{x,t}^n)_{z}(z')\, \diff \sigma^n_{t,x}(z)
		}_{
			= \diff (\gammaol_{x,t}^n)_{z'}(z)\, \diff \sigma(z')\,
		}
		\underbrace{
			\diff (\gamma_{x,t}^n)_{y}(y')\,\diff \bmpi^n_{t,x}(y)
		}_{
			= \diff(\gamma_{x,t}^n)_{y'}(y)\, \diff \bmpi_{t,x}(y')
		},
		\intertext{where we switched the disintegration from the first to the second marginals. Now use that $\bmnu_{t,x} = u_{t,x} \cdot (\sigma \otimes \bmpi_{t,x})$,}
		&\quad=
		\int_{Z^2 \times Y^2}\phi(z,y)\,
		\diff (\gammaol_{x,t}^n)_{z'}(z)\,
		\diff(\gamma_{x,t}^n)_{y'}(y)\,
		\diff\bmnu_{t,x} (z',y')
		=
		\int_{Z\times Y}
		\phi(z,y)\,
		\diff \bmnu_{t,x}^n(z,y).
		\nonumber
	\end{align*}

	Regarding the entropic regularization, notice that $\varphi(s) = s\log(s) - s + 1$ is a convex function, so using Jensen's inequality we obtain:
	\begin{align*}
		\varphi(u_{t,x}^n(z,y) )
		&=
		\varphi\left(
		\int_{Z\times Y} u_{t,x}(z', y')
		\diff (\gammaol_{x,t}^n )_z(z')\diff(\gamma_{x,t}^n )_y(y')
		\right)
		\\
		&\le
		\int_{Z\times Y}
		\varphi(u_{t,x}(z', y') )
		\diff (\gammaol_{x,t}^n )_z(z')\diff(\gamma_{x,t}^n )_y(y'),
	\end{align*}
	so the entropic term can be bounded as
	\begin{align*}
		& \int_{[0,T]\times X}
		\KL(\bmnu_{t,x}^n | \sigma_{t,x}^n\otimes \bmpi_{t,x}^n )
		\,\diff \mu(x)\,\diff t
		=
		\\
		& \quad =
		\int_{[0,T]\times X}
		\int_{Z\times Y} \varphi(u_{t,x}^n(z,y) )
		\,\diff\sigma_{t,x}^n(z)	\,\diff \bmpi_{t,x}^n(y)
		\,\diff \mu(x) \,\diff t
		\\
		& \quad \le
		\int_{[0,T]\times X}
		\int_{Z^2\times Y^2} \varphi(u_{t,x}(z',y') )\,
		\diff \gammaol_{x,t}^n (z,z')\,\diff\gamma_{x,t}^n (y,y')
		\,\diff \mu(x)\,\diff t
		\\
		& \quad =
		\int_{[0,T]\times X}
		\int_{Z\times Y} \varphi(u_{t,x}(z',y') )
		\,\diff \sigma(z') \,\diff \bmpi_{t,x}(y')
		\,\diff \mu(x)\,\diff t
		\\
		& \quad =
		\int_{[0,T]\times X}
		\KL(\bmnu_{t,x} | \sigma \otimes \bmpi_{t,x})
		\,\diff \mu(x)\,\diff t.
	\end{align*}
	Adding to this the convergence of the transport contribution along weak* converging sequences (Lemma \ref{lemma:convergence_objective_function}) and that $n\veps^n$ converges to $\eta$ it follows that, for both $\eta<\infty$ and $\eta=\infty$, $\displaystyle \limsup_{n\in\nset, n\rightarrow\infty} F^n_T(\bmnu^n) \leq F_T(\bmnu)$.

	\item{} \textit{[$\eta=0$, $\veps_n > 0$ for an infinite number of indices $n$]}
	This case is slightly more challenging since the reconstructed $\bmnu_{t,x}^n$ may not have a density with respect to $\sigma_{t,x}^n \otimes \bmpi_{t,x}^n$, and thus the $\KL$ term at finite $n$ may explode for $\veps^n>0$. Hence, for those $n$ the recovery sequence needs to be adjusted. We apply the \textit{block approximation} technique as in \cite{Carlier-EntropyJKO-2015}, which is summarized in Lemma \ref{lem:block_approximation}. We set $\hat{\bmnu}_{t,x}^n$ to be the block approximation of $\bmnu_{t,x}^n$ at scale $\len_n \assign n\veps^n$ (where we set $\Omega \assign Y \cup Z$).
	Lemma \ref{lem:block_approximation} provides that the marginals are preserved, i.e.~$\hat{\bmnu}_{t,x}^n \in \Pi(\sigma_{t,x}^n,\bmpi_{t,x}^n)$. In addition we find that $W_{Z \times Y}(\bmnu_{t,x}^n, \hat{\bmnu}_{t,x}^n) \leq \ell_n \cdot \sqrt{2d}$ and thus $\hat{\bmnu}^n \rightweaks \bmnu$ (arguing as above, e.g.~via dominated convergence). So by Lemma \ref{lemma:convergence_objective_function} the transport contribution still converges.
	Finally, for the entropic contribution we get from Lemma \ref{lem:block_approximation},
	\begin{align}
		n\veps^n \KL(\hat{\bmnu}_{t,x}^n \mid \sigma_{t,x}^n \otimes \bmpi_{t,x}^n)
		\le
		Cn\veps^n - 2d n\veps^n \log(n\veps^n)
		\xrightarrow{n\rightarrow\infty} 0.
	\end{align}
	Wrapping up, this means that $\limsup_{ n\in\nset, n\rightarrow\infty} F^n_T(\hat{\bmnu}^n) \le F_T(\bmnu)$, and $(\hat{\bmnu}^n)_n$ represents a valid recovery sequence. \qedhere
\end{itemize}
\end{proof}

\subsection{Continuity equation}
\label{sec:ContinuityEquation}
The discrete momenta $\bmomega^n$, \eqref{eq:omega_discrete} have been introduce to approximately describe the `horizontal' mass movement in the discrete trajectories $\bmpi^n$, \eqref{eq:pi_discrete} via a continuity equation on $X \times Y$. We now establish that in the limit the relation becomes exact.
\begin{proposition}
	\label{prop:CE}
	Let Assumption \ref{assumption:PitConvergence} hold. Let $\bmomega \in \meas(\R_+ \times X \times Y)^d$ and $\nsetb \subset \nset \subset 2\N$ be a subsequence on which $\bmomega^n \restr [0,T] \rightweaks \bmomega \restr [0,T]$ for any $T \in (0,\infty)$. Then $\bmpi$ and $\bmomega$ solve the horizontal continuity equation
	\begin{align*}
		\partial_t \bmpi_t + \ddiv_X \bmomega_t = 0 \qquad \tn{for $t>0$ and}
		\qquad \bmpi_{t=0} = \piInit
	\end{align*}
	in a distributional sense. More precisely, for any $\phi \in \cont^1_c(\R_+ \times X \times Y)$ one has
	\begin{align}
		\label{eq:FlowWeakPDE}
		\int_{\R_+ \times X \times Y} \partial_t\phi\,\diff \bmpi
		+ \int_{\R_+ \times X \times Y} \nabla_X \phi \cdot \diff \bmomega
		= - \int_{X \times Y} \phi(0,x,y)\,\diff \piInit(x,y).
	\end{align}
\end{proposition}
\begin{proof}
	Let $\phi \in \cont^1_c(\R_+ \times X \times Y)$. We will show that
	\begin{align}
		\int_{\R_+ \times X \times Y} \partial_t \phi \,\diff \bmpi^n
		+ \int_{\R_+ \times X \times Y} \nabla_X \phi \cdot\diff \bmomega^n
		=
		- \int_{X \times Y} \phi(0,x,y)\,\diff \piInitN(x,y)
		+ o(1)
		\label{eq:FlowWeakPDE_N}
	\end{align}
	for $\nsetb \ni n \to \infty$ and then \eqref{eq:FlowWeakPDE} will follow by weak* convergence of $(\bmpi^n)_n$ to $\bmpi$ and of $(\bmomega^n)_n$ to $\bmomega$ on compact time intervals.

	Since $\phi$ has compact support, there exists some $T\in \R_+$ such that $\phi(t, \cdot, \cdot)  = 0$ for all $t \ge T$. Now fix some $n$, and note that $\partial_t \phi$ and $\nabla_X \phi$ are uniformly continuous and $\bmpi_t^n$ has finite mass on $[0,T]$. Thus, replacing $\partial_t \phi(t,x,y)$ and $\nabla_X \phi(t,x,y)$ on the left hand side of \eqref{eq:FlowWeakPDE_N} by $\partial_t \phi(t,\xJn,y)$ and $\nabla_X \phi(t,\xJn,y)$ only introduces an error of $o(1)$ in the first two terms, since $\|x-\xJn\|\leq \sqrt{d}/n$. Thus the first term of \eqref{eq:FlowWeakPDE_N} becomes

	\begin{align}
		&\int_{0}^{\R_+} \int_{X \times Y}
		\partial_t \phi(t,x,y) \,\diff \bmpi^n_t(x,y)\, \diff t
		=
		\int_{0}^T \int_{X \times Y} \partial_t \phi(t,\xJn,y) \,\diff \bmpi^n_t(x,y)\, \diff t
		+
		o(1).
		\nonumber
		\\
		\intertext{Now take $K = \lceil nT \rceil$, and use that $\bmpi_t^n$ is constant on time intervals of length $1/n$ and on composite cells, so we can continue}
		&\qquad =
		\sum_{k = 0}^{K-1}
		\sum_{J\in \partnk}
		\int_{0}^{1/n}
		\int_{Y}
		\partial_t \phi(\tfrac{k}{n} + s, x_J^n, y)
		\diff \iter{\nu_J}{n,k}(y)\, \diff s
		+
		o(1)
		\nonumber
		\\
		&\qquad =
		\sum_{k = 0}^{K-1}
		\sum_{J\in \partnk}
		\int_{Y}
		[\phi(\tfrac{k+1}{n}, x_J^n, y) - \phi(\tfrac{k}{n}, x_J^n, y)]
		\diff \iter{\nu_J}{n,k}(y)
		+
		o(1).
		\label{eq:PDE_first_term_n}
	\end{align}

	Likewise, in the second term of \eqref{eq:FlowWeakPDE_N} we replace again $x$ by $\xJn$, and also $t$ by $\tfrac{\lceil nt \rceil}{n}$,
	which yields again an error of order $o(1)$. The second term then becomes
	\begin{align}
		& \int_{\R_+} \int_{X \times Y} \nabla_X \phi(t,x,y) \cdot\diff \bmomega^n_t(x,y) \, \diff t
		=
		\int_{0}^T \int_{X \times Y} \nabla_X
		\phi(\tfrac{\lceil nt \rceil}{n},\xJn,y)
		\cdot\diff \bmomega^n_t(x,y) \, \diff t
		+
		o(1)
		\nonumber
		\\
		&\qquad =
		o(1)
		+
		\sum_{k = 0}^{K-1}
		\sum_{J\in \partnk}
		\sum_{\hat{J} \in \neigh(J)}
		\int_{0}^{1/n}
		\int_{Y}
		\nabla_X\phi(\tfrac{k+1}{n}, x_J^n, y) \cdot
		(x_{\hat{J}}^n-x_J^n) \cdot n\,
		\diff \iter{\nu_{J,\hat{J}}}{n,k}(y)\,
		\diff s.
		\nonumber
		\\
		\intertext{Now the integral over $[0,1/n]$ cancels with the factor $n$, and $\nabla_X\phi(\tfrac{k+1}{n}, x_J^n, y) \cdot
			(x_{\hat{J}}^n-x_J^n) = \phi(\tfrac{k+1}{n}, x_{\hat{J}}^n, y) - \phi(\tfrac{k+1}{n}, x_J^n, y) + o(1/n)$. We get}
		&\qquad =
		o(1)
		+
		\sum_{k = 0}^{K-1}
		\sum_{J\in \partnk}
		\sum_{\hat{J} \in \neigh(J)}
		\int_{Y}
		\left[\phi(\tfrac{k+1}{n}, x_{\hat{J}}^n, y) - \phi(\tfrac{k+1}{n}, x_J^n, y) + o(1/n)\right]
		\diff \iter{\nu_{J,\hat{J}}}{n,k}(y).
		\nonumber
		\\
		\intertext{The sum of all $\iter{\nu_{J, \hat{J}}}{n,k}$ over $J$ and $\hat{J}$ has unit mass, so the total contribution of the $o(1/n)$ errors scales like $K\cdot o(1/n) = T \cdot n \cdot o(1/n)$, which is $o(1)$. Thus, we can absorb this error term into the global $o(1)$ error:}
		&=
		o(1)
		+
		\sum_{k = 0}^{K-1}
		\left[
		\sum_{J\in \partnk}
		\sum_{\hat{J} \in \neigh(J)}
		\int_{Y}
		\phi(\tfrac{k+1}{n}, x_{\hat{J}}^n, y)
		\diff \iter{\nu_{J,\hat{J}}}{n,k}(y)
		-
		\sum_{J\in \partnk}
		\sum_{\hat{J} \in \neigh(J)}
		\int_{Y}
		\phi(\tfrac{k+1}{n}, x_J^n, y)
		\diff \iter{\nu_{J,\hat{J}}}{n,k}(y)
		\right]
		\nonumber
		\\
		\intertext{Then, in the second term we can regroup all $ \iter{\nu_{J,\hat{J}}}{n,k}$ with the same $J$ into $\iter{\nu_{J}}{n,k}$. In the first term we can first reverse the order of the sums, and then use that adding up all $\iter{\nu_{J,\hat{J}}}{n,k}$ with the same $\hat{J}$ results in $\iter{\nu_{\hat{J}}}{n,k} = \iter{\nu_{\hat{J}}}{n,k+1}$ (see \eqref{eq:CompCellMarginalPreservation} for the equality). This leaves us with}
		&=
		o(1)
		+
		\sum_{k = 0}^{K-1}
		\left[
		\sum_{\hat{J}\in \partnkk}
		\int_{Y}
		\phi(\tfrac{k+1}{n}, x_{\hat{J}}^n, y)
		\diff \iter{\nu_{\hat{J}}}{n,k+1}(y)
		-
		\sum_{J\in \partnk}
		\int_{Y}
		\phi(\tfrac{k+1}{n}, x_J^n, y)
		\diff \iter{\nu_{J}}{n,k}(y)
		\right].
		\label{eq:PDE_second_term_n}
	\end{align}

 Now we can combine the temporal and spatial parts, noticing that the first term in \eqref{eq:PDE_first_term_n} cancels with the second term in \eqref{eq:PDE_second_term_n}, so the left hand side of \eqref{eq:FlowWeakPDE_N} equals
\begin{align}
	\int_{\R_+} & \int_{X \times Y} \partial_t \phi(t,x,y) \,\diff \bmpi^n_t(x,y)\, \diff t
	+ \int_{\R_+} \int_{X \times Y} \nabla_X \phi(t,x,y) \cdot\diff \bmomega^n_t(x,y) \, \diff t
	=
	\nonumber
	\\
	&=
	o(1)
	+
	\sum_{k = 0}^{K-1}
	\Bigg[
	\sum_{\hat{J}\in \partnkk}
	\int_{Y}
	\phi(\tfrac{k+1}{n}, x_{\hat{J}}^n, y)
	\diff \iter{\nu_{\hat{J}}}{n,k+1}(y)
	-
	\sum_{J\in \partnk}
	\int_{Y}
	\phi(\tfrac{k}{n}, x_J^n, y)
	\diff \iter{\nu_{J}}{n,k}(y)
	\Bigg]
	\nonumber
	\\
	\intertext{which is a telescopic sum. The surviving terms are just}
	&=
	o(1)
	+
	\sum_{\hat{J}\in \partJ^{n,K}}
	\int_{Y}
	\phi(\tfrac{K}{n}, x_{\hat{J}}^n, y)
	\diff \iter{\nu_{\hat{J}}}{n,k+1}(y)
	-
	\sum_{J\in \partJ^{n,0}}
	\int_{Y}
	\phi(0, x_J^n, y)
	\diff \iter{\nu_{J}}{n,0}(y)
	\nonumber
	\\
	\intertext{The first integral vanishes, since $\phi(K/n,\cdot, \cdot) = 0$. In the second integral we first integrate again in space:}
	&=
	o(1)
	-
	\sum_{J\in \partJ^{n,0}}
	\int_{X_J^n\times Y}
	\phi(0, x_J^n, y)
	\diff \mu(x)
	\frac{1}{m_J^n}
	\diff \iter{\nu_{J}}{n,0}(y)
	\nonumber
	\\
	&=
	o(1)
	-
	\sum_{J\in \partJ^{n,0}}
	\int_{X_J^n\times Y}
	\phi(0, x_J^n, y)
	\diff \bmpi_0^n(x,y)
	\nonumber
	\\
	\intertext{We use for the last time that replacing $x_J^n$ by $x$ in $X_J^n$ introduces a global $o(1)$ error (since $\bmpi_0^n$ has unit mass):}
	&=
	o(1)
	-
	\sum_{J\in \partJ^{n,0}}
	\int_{X_J^n\times Y}
	\phi(0,x, y)
	\diff \bmpi_0^n(x,y)
	=
	o(1)
	-
	\int_{X \times Y}
	\phi(0,x,y)
	\diff \bmpi_0^n(x,y)
	\nonumber
	\\
	\intertext{and finally, since $(\bmpi_0^n)_n$ converges weak* to $\piInit$,}
	&=
	-
	\int_{X \times Y}
	\phi(0,x,y)
	\diff \piInit(x,y)
	+
	o(1)
	,
	\nonumber
\end{align}
which is precisely \eqref{eq:FlowWeakPDE_N}.
\end{proof}

The following observation may help in the interpretation of the trajectories generated by domain decomposition.
\begin{proposition}
	\label{prop:OmegaAbsContinuous}
	The sequence momentum fields $(\bmomega^n)_n$ and the limit $\bmomega$ (Proposition \ref{prop:MomentumConvergence}) are absolutely continuous with respect to their respective trajectories $(\bmpi^n)_n$ and $\bmpi$ and the component-wise density is bounded by one. More precisely,
	\begin{equation}
		\left|\RadNikD{(\bmomega^n_{t,x})_\ell}{\bmpi^n_{t,x}} \right| \le 1
		\qquad \tn{for all $n\in 2\N$, and}
		\qquad
		\left|\RadNikD{(\bmomega_{t,x})_\ell}{\bmpi_{t,x}} \right| \le 1.
	\end{equation}
	for a.e.~$t$, $\mu$-a.e.~$x$ and all $\ell \in \{1,\ldots,d\}$.
\end{proposition}
\begin{proof}
For $n\in 2\N$, this is a simple consequence of \eqref{eq:MomentumFieldFromLambda}:
	\begin{align*}
		\left|\RadNikD{(\bmomega^n_{t,x})_\ell}{\bmpi^n_{t,x}} \right|
		&=
		\left|\RadNikD{[\proj_Y (\bmnu_{t,x}^n\restr Z^\ell_+ \times Y)
			-
			\proj_Y (\bmnu_{t,x}^n\restr  Z^\ell_- \times Y)]}{\bmpi^n_{t,x}} \right|
		\\
		& \le
		\left|\RadNikD{\proj_Y (\bmnu_{t,x}^n\restr Z^\ell_+ \times Y)}{\bmpi^n_{t,x}} \right|
		+
		\left|\RadNikD{\proj_Y (\bmnu_{t,x}^n\restr  Z^\ell_- \times Y)}{\bmpi^n_{t,x}} \right|
		\\
		\intertext{and now use that $\bmnu^n$ is a positive measure}
		&=
		\RadNikD{\proj_Y (\bmnu_{t,x}^n\restr Z^\ell_+ \times Y)}{\bmpi^n_{t,x}}
		+
		\RadNikD{\proj_Y (\bmnu_{t,x}^n\restr  Z^\ell_- \times Y)}{\bmpi^n_{t,x}}
		=
		\RadNikD{\proj_Y (\bmnu^n_{t,x}\restr Z \times Y)}{\bmpi^n_{t,x}}
		=
		1.
	\end{align*}
The argument for the limit momentum $\bmomega$ is completely analogous.
\end{proof}
\begin{remark}[Interpretation]
In Section \ref{sec:Oscillations} we have introduced the `vertical transport metric' $\WY$, Definition \ref{def:WY}, which can be interpreted as an optimal transport metric on $X \times Y$ that only allows `vertical transport' of mass along the $Y$-direction. It was established that (under suitable conditions) the discrete and limit trajectories are equicontinuous in this metric. This means, we can interpret the changes in $\bmpi^n_t$ and $\bmpi_t$ over time as being induced by relatively regular movement of mass in the vertical direction (and the corresponding convergence of $\bmpi^n_t$ to $\bmpi_t$ was important for the convergence of the cell problems in Section \ref{sec:Convergence}).

Conversely, Proposition \ref{prop:OmegaAbsContinuous} allows the interpretation of the changes in $\bmpi^n_t$ and $\bmpi_t$ over time as horizontal movement of mass, along the $X$-direction, where particles move at most with velocity 1 along each spatial axis. Hence, if we introduced a `horizontal' analogue of the metric $\WY$, the curves $\bmpi^n_t$ and $\bmpi_t$ would be Lipschitz with respect to that metric with Lipschitz constant $\sqrt{d}$.
This corresponds to the fact that each mass particle can only travel by one basic cell along each axis per iteration.
Unfortunately, this regularity is not suitable for the convergence of the fiber problems. Hence, the `detour' via Section \ref{sec:Oscillations} is necessary.
\end{remark}

\subsection{Main result}
We can now summarize and assemble the results from the two previous sections to arrive at the main result of the article.
\begin{theorem}
	\label{thm:main}
	Assume Assumptions \ref{assumption:PitConvergence} and \ref{assumption:RegularDiscretization} hold.
	Then, up to selection of a subsequence $\nsetb \subset \nset$, the sequences of discrete trajectories $(\bmpi^n)_n$, \eqref{eq:pi_discrete}, and momenta $(\bmomega^n)_n$, \eqref{eq:omega_discrete}, that are generated by the domain decomposition algorithm at scales $n$, converge weak* on compact sets to a limit trajectory $\bmpi$ and momentum $\bmomega$ as $n \to \infty$.
	The limits solve the horizontal continuity equation
	\begin{align*}
		\partial_t \bmpi_t + \ddiv_X \bmomega_t = 0 \qquad \tn{for $t>0$ and}
		\qquad \bmpi_{t=0} = \piInit
	\end{align*}
	on $X \times Y$ in a distributional sense, \eqref{eq:FlowWeakPDE}.
	The limit momentum $\bmomega$ is induced by an asymptotic version of the domain decomposition algorithm. More precisely, for $\Lebesgue \otimes \mu$-almost all $(t,x)$ its disintegration is given by
	\begin{multline*}
		(\bmomega_{t,x})_\ell
		\assign
		\proj_Y (\bmnu_{t,x}\restr Z^\ell_+ \times Y)
		-
		\proj_Y (\bmnu_{t,x}\restr  Z^\ell_- \times Y)
		\quad
		\text{for } \ell = 1,...,d,
		\\
		\text{with}\quad
		Z^\ell_\pm = \{z \in Z \mid \pm z_\ell > 0\}
	\end{multline*}
	where for $\eta<\infty$ the measure $\bmnu_{t,x}$ is given as a minimizer of the asymptotic cell problem
	\begin{align*}
	\inf \left\{
	\int_{Z\times Y}
	\inner{\nabla_X c(x, y)}{z} \, \diff \vecnu(z,y)
	+
	\eta \cdot \KL(\vecnu | \sigma\otimes \bmpi_{t,x})
	\middle|
	\vecnu\in\Pi(\sigma, \bmpi_{t,x})
	\right\}.
	\end{align*}
	For $\eta=\infty$ one finds $\bmnu_{t,x}=\sigma \otimes \bmpi_{t,x}$, which implies $\bmomega_{t,x}=0$ and thus $\bmpi_t=\piInit$ for all $t \in \R_+$. Hence, the algorithm asymptotically freezes.
\end{theorem}
\begin{proof}
Assumptions \ref{assumption:PitConvergence} includes the weak* convergence on compact sets of $\bmpi^n$ to a limit $\bmpi$ for a subsequence $\nset$.
Under Assumptions \ref{assumption:PitConvergence} and \ref{assumption:RegularDiscretization}, Proposition \ref{prop:MomentumConvergence} provides the existence of a subsequence $\nsetb \subset \nset$ such that $\bmomega^n$ converges weak* on compact sets to a limit $\bmomega$ on $\nsetb$, and this limit is of the prescribed form for (almost-everywhere) fiber-wise minimizers $\bmnu_{t,x}$ of a limit fiber problem given in Definition \ref{def:LimitFiberProblem}.

For $\eta<\infty$ the limit fiber problem is as stated. For $\eta=\infty$, the unique minimizer of the limit fiber problem is given by $\bmnu_{t,x}=\sigma \otimes \bmpi_{t,x}$. With Assumption \ref{assumption:RegularDiscretization} this implies $\bmomega_{t,x}=0$ (see Remark \ref{remark:ZQuadrantsSameMass}).

Solution to the continuity equation is provided by Proposition \ref{prop:CE}. For $\eta=\infty$, with $\bmomega=0$ this implies that the limit trajectory $\bmpi_t$ is constant and equal to $\piInit$.
\end{proof}

\begin{remark}
By virtue of Proposition \ref{prop:WTVbGivesWYConditions}, Assumption \ref{assumption:PitConvergence} can be replaced by Assumption \ref{assumption:RegularityMu} and the condition $\sup_{n,k} \WTVb(\pi^{n,k})<\infty$.
\end{remark}

\begin{remark}[Discussion]
\label{rem:MainDiscussion}
We observe that solutions to the limit system given in Theorem \ref{thm:main} are not unique. For instance, if $\piInit=(\id,S)_\sharp \mu$ for some (non-optimal) Monge-map $S: X \to Y$, then a solution to the limit system is given by $\bmpi_t=\piInit$, $\bmomega_t=0$, $\bmnu_{t,x}=\sigma \otimes \delta_{S(x)}$ for all $t \geq 0$ where we used that $\bmpi_{t,x}=(\piInit)_x=\delta_{S(x)}$.
In contrast, limit solutions generated by the domain decomposition algorithm are usually able to leave such a point by making the coupling non-deterministic, since at each $n$ the algorithm has a `non-zero range of vision' (see Figure \ref{fig:example_iterations_flipped_bottleneck}, for an example).
This situation is somewhat analogous to a gradient flow being stuck in a saddle point, whereas a minimizing movement scheme or a proximal point algorithm is able to move on.
		
The algorithm by Angenent, Haker, and Tannenbaum \cite{OptimalTransportWarpingTheory} solves the $W_2$-optimal transport problem (for convex $X$ and $\mu \ll \Lebesgue$) by starting from some feasible Monge map and then subsequently removing its curl in a suitable way. It therefore generates trajectories that lie solely in the subset of Monge couplings (i.e.~concentrated on a map) and the algorithm breaks down when cusps or overlaps form. (\cite{OptimalTransportWarpingTheory} also discusses a regularized version.)
From the previous paragraph we deduce that in asymptotic domain decomposition trajectories, mass can only move when the coupling is not of Monge-type and the algorithm is well-defined on the whole space of Kantorovich transport plans.
As shown in Figure \ref{fig:example_iterations_flipped_bottleneck}, the domain decomposition algorithm can also evolve away from an initial, sub-optimal Monge coupling by making it instantaneously non-deterministic (see also Remark \ref{rem:MainDiscussion}).
\end{remark}

\section{Numerical examples}\label{sec:Numerics}
Numerical examples for the practical efficiency of the method have already been illustrated in \cite{BoSch2020}.
A basic intuition for the asymptotic behaviour as $n \to \infty$ on `well-behaved' examples can be drawn from Figures \ref{fig:introflipped} to \ref{fig:example_iterations_flipped_bottleneck}.
The examples in this section aim at providing some glimpse beyond the theoretical results of this article by illustrating counter-examples and conjectures.
In Section \ref{sec:HessianExample} we show a initialization that is locally optimal on all composite cells but not not globally minimal for $\veps=0$. In Section \ref{sec:SemidiscreteExample} we show a semi-discrete example that takes an increasingly long time for convergence for each $n$, with the limit trajectory remaining stuck at $\piInit$. Finally, in Section \ref{sec:WTVExample} we present a numerical example that suggests that a $\WTVb$ bound may not hold universally.

\subsection{Example for asymptotic sub-optimality: discretization}
\label{sec:HessianExample}
It is well-known that the domain decomposition algorithm on discrete unregularized problems may fail to converge to the globally optimal solution.
Examples are given in \cite[Section 5.2]{BenamouPolarDomainDecomposition1994} and \cite[Example 4.12]{BoSch2020}.
Nevertheless it is instructive to study such an example in the context of the asymptotic behaviour of the algorithm.

\begin{figure}[bt]
	\centering	
	\includegraphics[width=\linewidth]{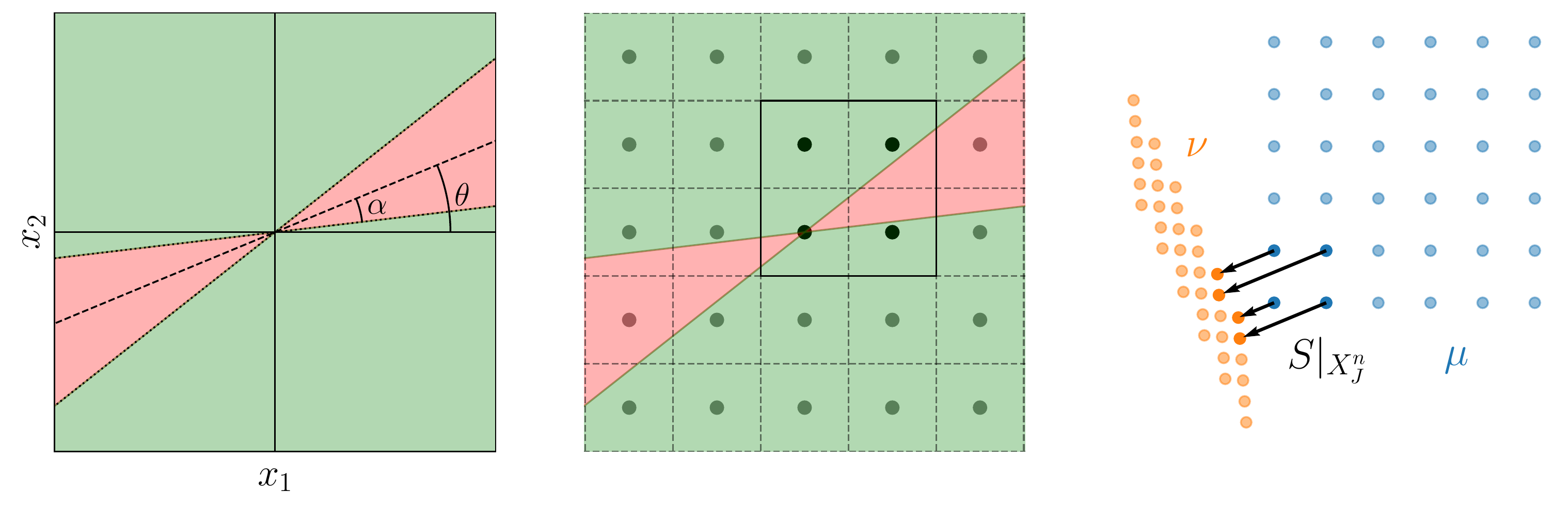}
	\caption{Left, cone where $x^T H x < 0$ in red. Center, for $H$ with a sufficiently narrow, $S\big|_{X_J^n}$ is a monotone arrangement for all composite cells $X_J^n$. Right, the resulting coupling $\piInitN$ is optimal on each composite cell, albeit not globally. 
	}
	\label{fig:hessian_example}
\end{figure}
	
In the following, for $\beta \in S^1$ denote by $e_\beta$ the unit vector in $\R^2$ with orientation $\beta$. Further, let
$$
	H \assign R_\theta^\top \begin{pmatrix} -(\tan \alpha)^2 & 0 \\ 0 && 1 \end{pmatrix} R_\theta
	\qquad \tn{where} \qquad
	R_\theta  \assign
	\begin{pmatrix}
	\cos \theta  & \sin \theta  \\
	- \sin \theta  & \cos \theta 
	\end{pmatrix}$$
for some $\theta \in S^1$ and some $\alpha \in (0,\pi/2)$.
Set
\begin{align*}
	V & : \R^2 \to \R, & x & \mapsto \tfrac12 x^\top H\,x, & S \assign \nabla V.
\end{align*}
Clearly $H$ is indefinite with eigenvalues $-(\tan \alpha)^2$ and $1$ for eigenvectors $e_\theta$ and $e_{\theta+\pi/2}$.
Consequently $V$ is not convex and thus $S$ is in general not an optimal transport map between $\mu \in \measp(\R^2)$ and $S_\sharp \mu$ for the squared distance cost on $\R^2$ by virtue of Brenier's polar factorization \cite{MonotoneRerrangement-91}.
However, for a set $A \subset \R^2$ such that $(x_1-x_2)^\top H (x_1-x_2) \geq 0$ for all $x_1,x_2 \in A$ one quickly verifies that the graph of $S$ over $A$ is $c$-cyclically monotone for the squared distance and therefore $S$ is an optimal transport map between $\mu$ and $S_\sharp \mu$ for $\mu \in \measp(A)$.
One has $e_\beta^\top H e_\beta < 0$ if and only if $\beta \in (\theta-\alpha,\theta+\alpha) \cup (\theta+\pi-\alpha,\theta+\pi+\alpha)$ on $S^1$ (see Figure \ref{fig:hessian_example} for an illustration of this and the subsequent construction).
Therefore, if $\mu^n = \sum_{i \in I^n} m_i^n \cdot \delta_{x_i^n}$ (see Lemma \ref{lem:RegularityMuN}) such that each composite cell is essentially a small $2$ by $2$ Cartesian grid, and $\theta$ and $\alpha$ are chosen carefully, then $S$ will be optimal on each composite cell. But for sufficiently large $n$, some grid points from another cell will eventually lie in the red cone and $S$ is then not globally optimal on $X$.
Hence, if we set $\nu^n \assign S_\sharp \mu^n$ and $\piInitN=(\id,S)_\sharp \mu^n$, $\veps^n=0$, then the discrete trajectory at each $n$ will be stationary and so will be the limit trajectory. But it will not be globally optimal.

Fix now a scale $n$. If each basic cell contains more points, the space where the red cone `remains unnoticed' becomes smaller and thus $\alpha$ must decrease, but it can always be chosen to be strictly positive, i.e.~$S$ will not be globally optimal on a sufficiently large grid.
If we send the number of points per basic cell to infinity (being arranged on a regular Cartesian grid), for fixed $n$, it was shown in \cite[Section 5.2]{BenamouPolarDomainDecomposition1994} that in the limit one recovers a globally optimal coupling.
The behaviour in the case where the number of points per basic cell and $n$ tend to $\infty$ simultaneously remains open (see also Lemma \ref{lem:RegularityMuN}).

Similarly, if we set $\veps^n>0$, then for each fixed $n$ we know by \cite{BoSch2020} that the algorithm converges to the global minimizer. If $\eta=\infty$, then asymptotically the algorithm will freeze in the initial configuration (Theorem \ref{thm:main}).
If one takes $\veps^n$ to zero (for fixed $n$), then the sequence of first iterates (all with the same initialization) will converge (possibly up to subsequences) to a first iterate for the case $\veps^n=0$ (see \cite{Cominetti-ExpBarrierConvergence-1992,LeonardSchroedingerMK2012,Carlier-EntropyJKO-2015}). By stability of optimal transport \cite[Theorem 5.20]{Villani-OptimalTransport-09} this will then extend to a fixed finite number of iterations. That is, the iterate for $\veps^n>0$ after $k$ iterations should converge to a possible trajectory for $\veps^n=0$ after $k$ iterations (for $\veps^n=0$ the trajectory may not be unique, since the cell problems may not always have unique solutions), as $\veps^n \to 0$.
By sending $\veps^n$ to zero sufficiently fast, it seems therefore possible to obtain the same asymptotic behaviour as for $\veps^n=0$, i.e.~potentially we end up in a non-minimal configuration in the limit, even though at each $n$, eventually the globally optimal solution is found (after times that increase exponentially in $n$).
An open question is therefore, if there is an intermediate regime of scaling $(\veps^n)_n$ such that the global minimizer is obtained in the asymptotic trajectory.
Preliminary numerical experiments (and data from \cite{BoSch2020}) suggest that such a regime may exist.

\subsection{Example for asymptotic sub-optimality: semi-discrete transport}
\label{sec:SemidiscreteExample}
\begin{figure}[bt]
	\centering	
	\includegraphics[width=0.5\linewidth]{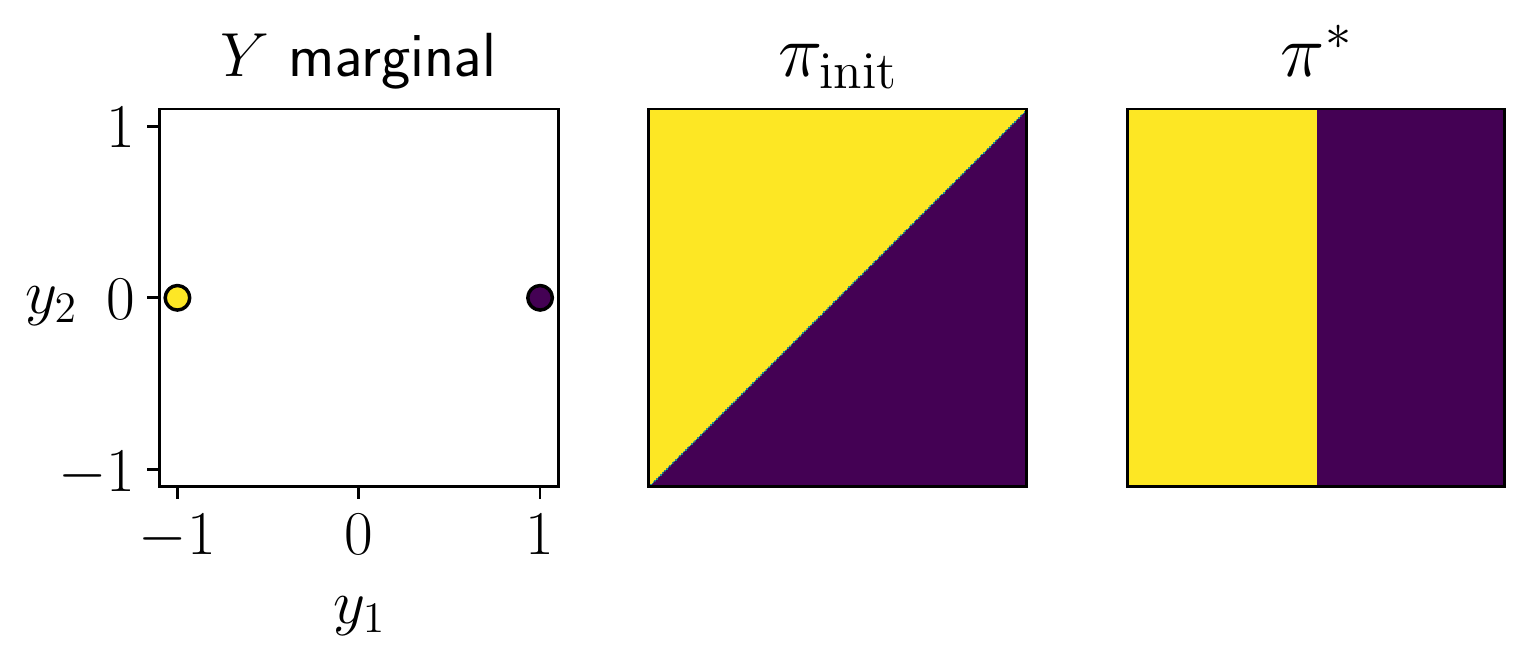}
	\caption{Configuration, initialization and optimal state for the semi-discrete example. Colors in the images for $\piInit$ and $\pi^\ast$ indicate the target point in $Y$.}
	\label{fig:semidiscrete}
\end{figure}
\begin{figure}[bt]
	\centering
	\includegraphics[width=0.91\linewidth]{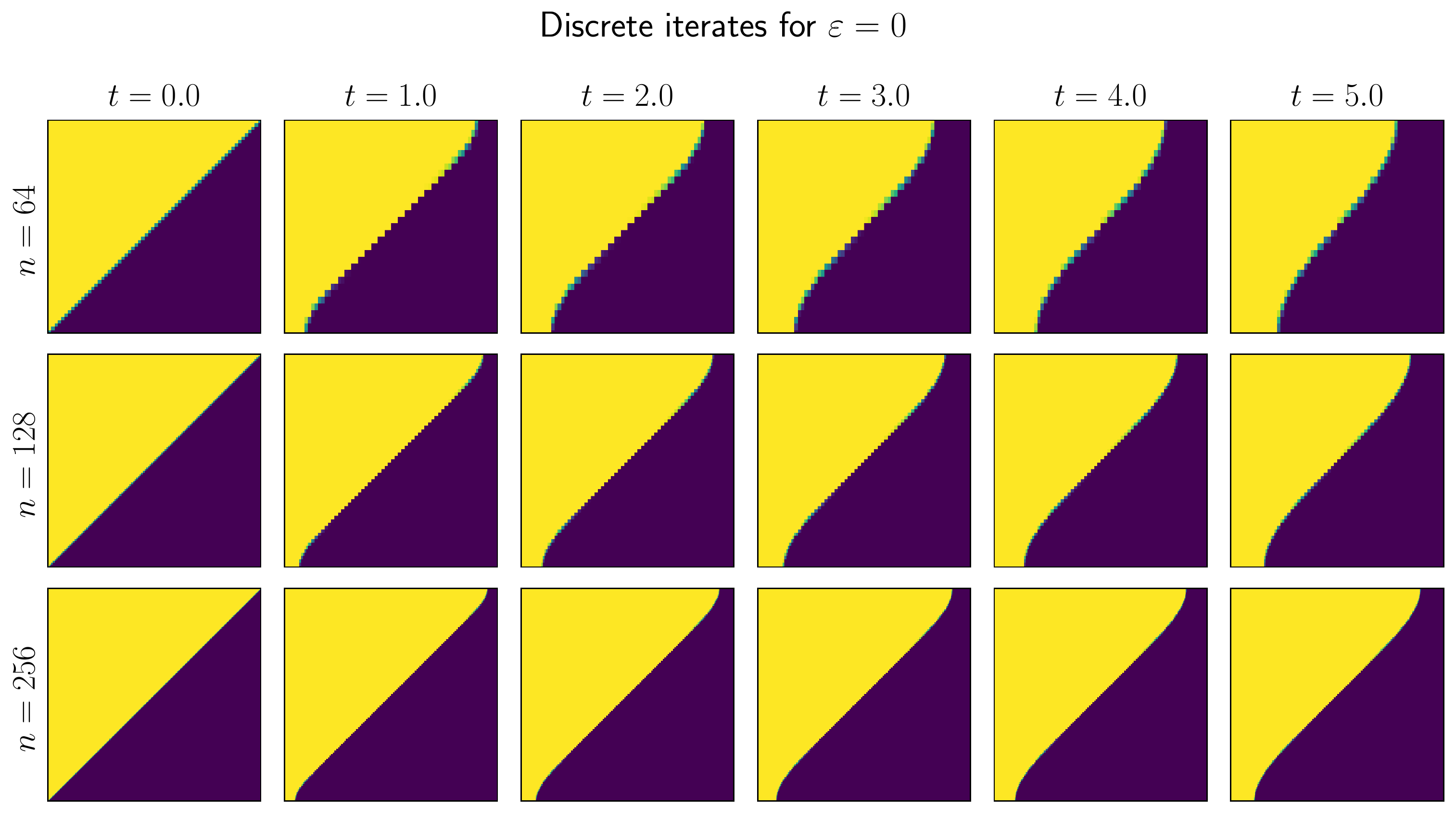}	
	\caption{Domain decomposition iterates $\iter{\pi}{n, \lfloor nt \rfloor}$ for the semi-discrete example. Colors in the images for couplings indicate the target point of mass in $Y$ for given cells. Along the interface mass is sent to both points, resulting in a mixed color. Note that changes in the iterations occur only along the interface, which remains approximately stable during iterations. As $n\to\infty$ the dynamics freeze.}
	\label{fig:semidiscreteB}
\end{figure}

Another asymptotic obstruction to global optimality occurs when a sub-optimal initial plan is chosen where sub-optimality is concentrated on an increasingly small subset of composite cells (as $n\to \infty$) and if this concentration is `stable' under iterations. Then most of the cells will not induce any change in the plan and asymptotically the trajectory freezes.
We illustrate this phenomenon with a semi-discrete example. Let $\mu = \mu^n = \Lebesgue \restr X$, $\nu = \nu^n$ is the sum of two Diracs at $(\pm 1,0)$ with equal mass, and the initialization $\piInit=\piInitN$ takes all the mass to the left of an approximately vertical interface line to $(-1,0)$ and everything to the right to $(+1,0)$ (see Figure \ref{fig:semidiscrete}). The optimal coupling would be given by a vertical interface.

All composite cells that do not touch the interface are locally optimal and will not change during an iteration, only cells that intersect the interface change. On a macroscopic level the effect is roughly as follows: mass for the left and right points in $Y$ will essentially `travel along the interface' in the appropriate direction, the rough structure of the interface itself remains stable. At the boundaries of $X$ the interface will curve towards the right orientation and this will gradually propagate into the interior of the domain (see Figure \ref{fig:semidiscreteB}, first row).
For each fixed $n$ convergence to the global minimizer follows from Benamou's analysis \cite{BenamouPolarDomainDecomposition1994} and the extension in Appendix \ref{sec:Benamou}.
However, the capacity of mass that can flow along the interface decreases with $n$ and thus convergence will become gradually slower, freezing in the limit $n \to \infty$ (see Figure \ref{fig:semidiscreteB}, rest of rows).

If the orientation of the initial interface is further than $\pi/2$ from the optimal orientation then it is not approximately stable under iterations and more dramatic changes to the plan happen at early times. However the eventual stationary point is in general also not globally optimal as $n \to \infty$.

\subsection{Example for potentially unbounded WTVB}\label{sec:WTVExample}
\begin{figure}[bt]
	\centering	
	\includegraphics[width=0.5\linewidth]{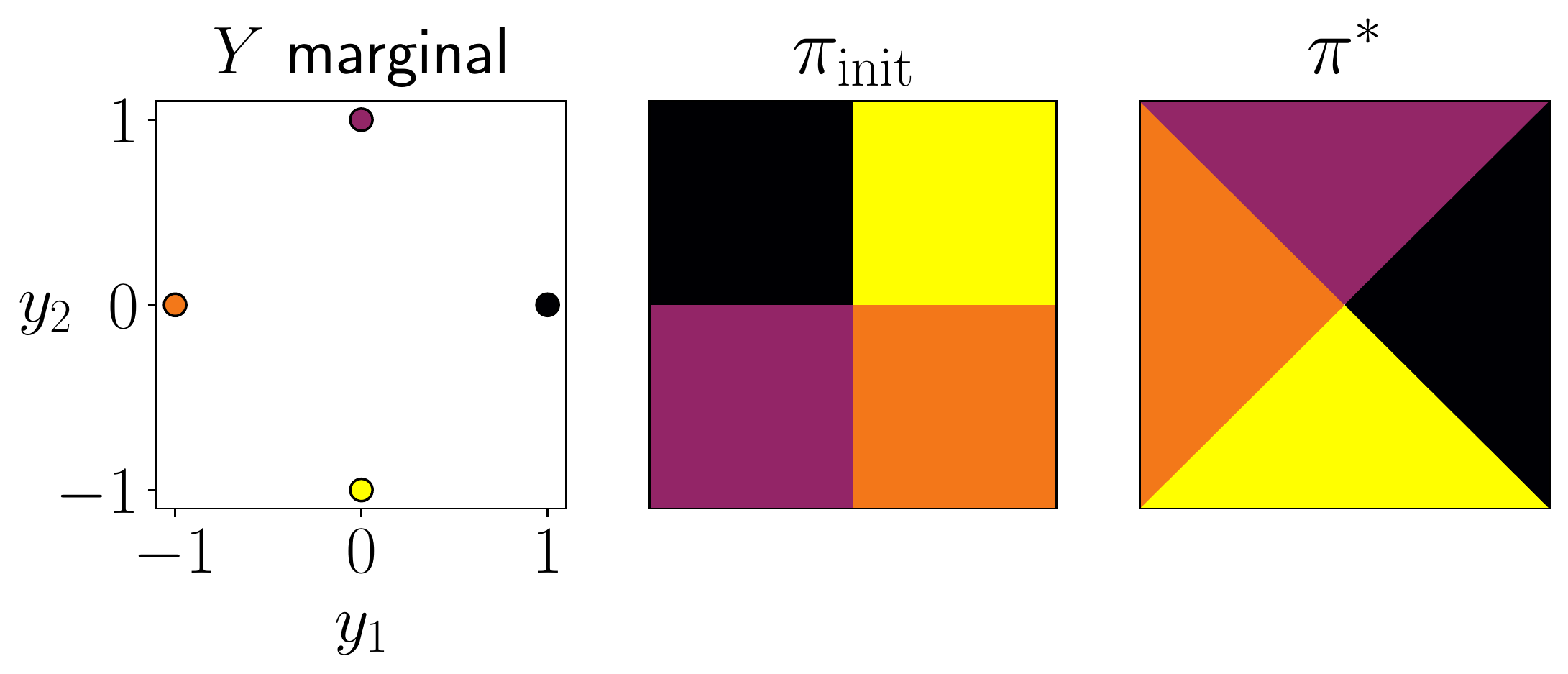}
	\caption{Configuration, initialization and optimal state for the unbounded $\WTVb$ example. 
	}
	\label{fig:semidiscrete-WTV}
\end{figure}

Finally, we give an example that seems to indicate that $\WTVb(\iter{\pi}{n,k})$ is not uniformly bounded in $n$ and $k$ in general. For this choose $\mu = \Lebesgue \restr X$, $\mu^n$ the one-point-per-basic-cell discretization of $\mu$, $\nu = \nu^n$ a measure composed of 4 Diracs with equal mass as shown in Figure \ref{fig:semidiscrete-WTV}, where we also show the initialization $\piInit$.

\begin{figure}[btp]
	\centering	
	\includegraphics[width=\linewidth]{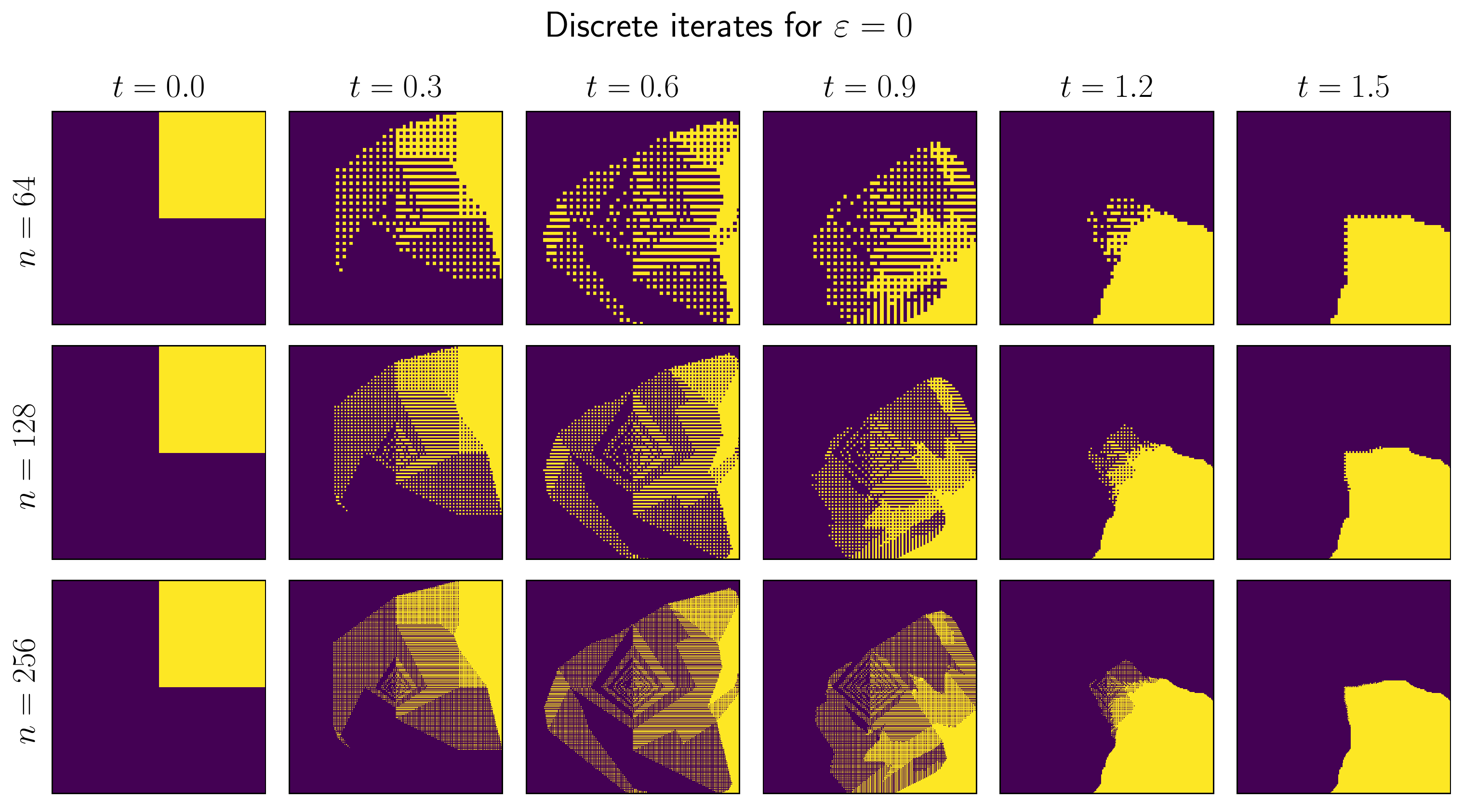}
	
	\vspace{2mm}
	\includegraphics[width=\linewidth]{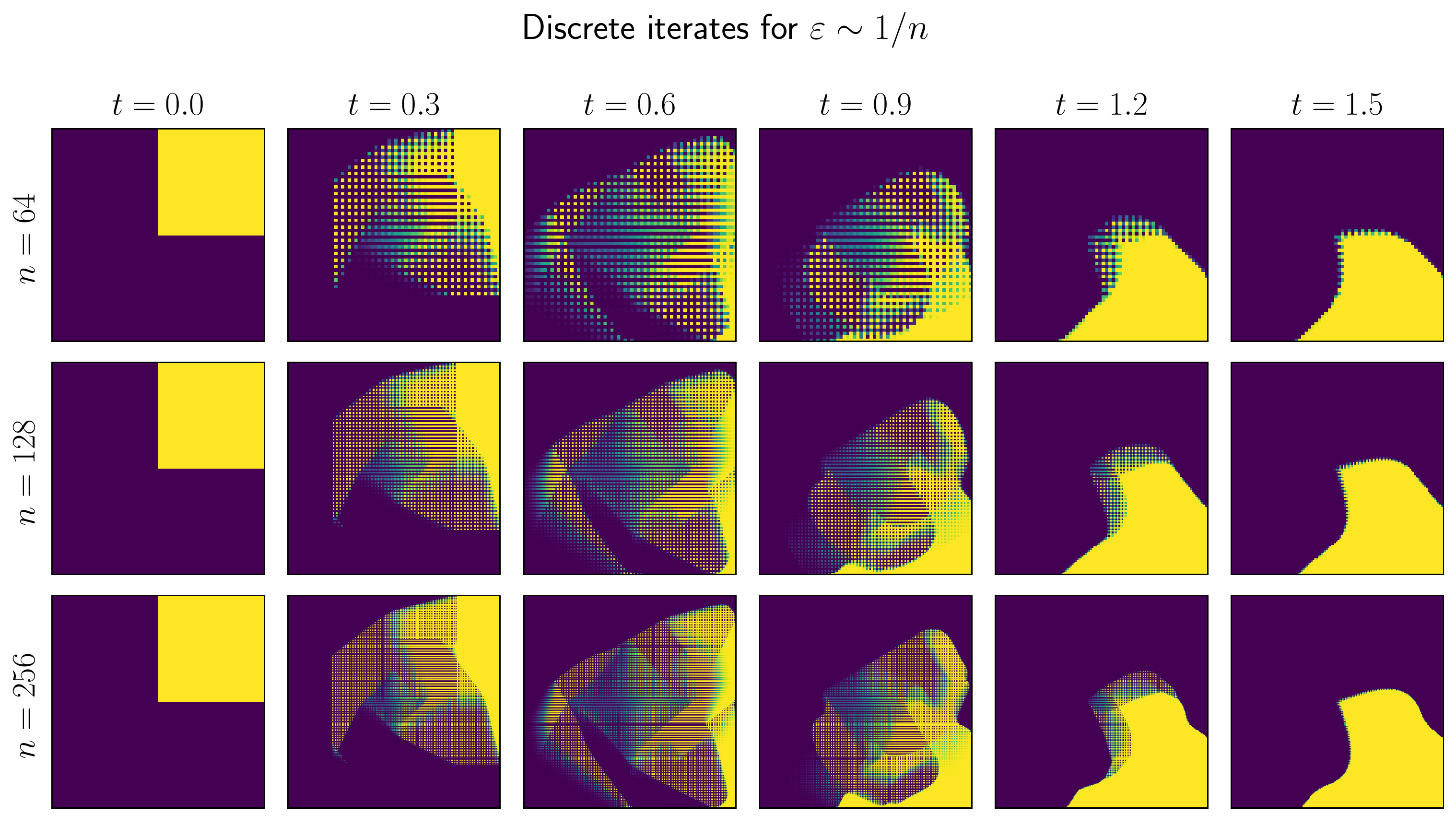}
	\caption{Top, iterates for the unbounded $\WTVb$ example for $\veps^n = 0$. Only the preimage of $(0,-1)$ is shown. Notice the intricate patterns near the center of $X$. Bottom, iterates for $\veps^n \sim 1/n$. The complex patterns are considerably smoothed out when regularization is introduced.} 
	\label{fig:semidiscrete-WTV-B}
\end{figure}

In Figure \ref{fig:semidiscrete-WTV-B} (top) we show a set of feasible discrete trajectories for $\veps^n= 0$ (solutions may not be unique) and discrete trajectories for $\veps^n \sim 1/n$ (bottom).
For simplicity, we show a color coding of the mass that each basic cell transports to $y_0=(0,-1)$ (i.e.~the disintegration of $\iter{\pi}{n,k}$ with respect to $Y$ at point $y_0$).
For $\veps^n=0$ this is binary, for $\veps^n>0$ it is generally not.
In Figure \ref{fig:semidiscrete-local-WTV} the local contributions of each basic cell to $\WTVb$ are shown for the  same unregularized and regularized couplings. Figure \ref{fig:semidiscrete-WTV-history} provides the total $\WTVb$ values of the trajectories over time.

\begin{figure}[btp]
	\centering
	\includegraphics[width=\linewidth]{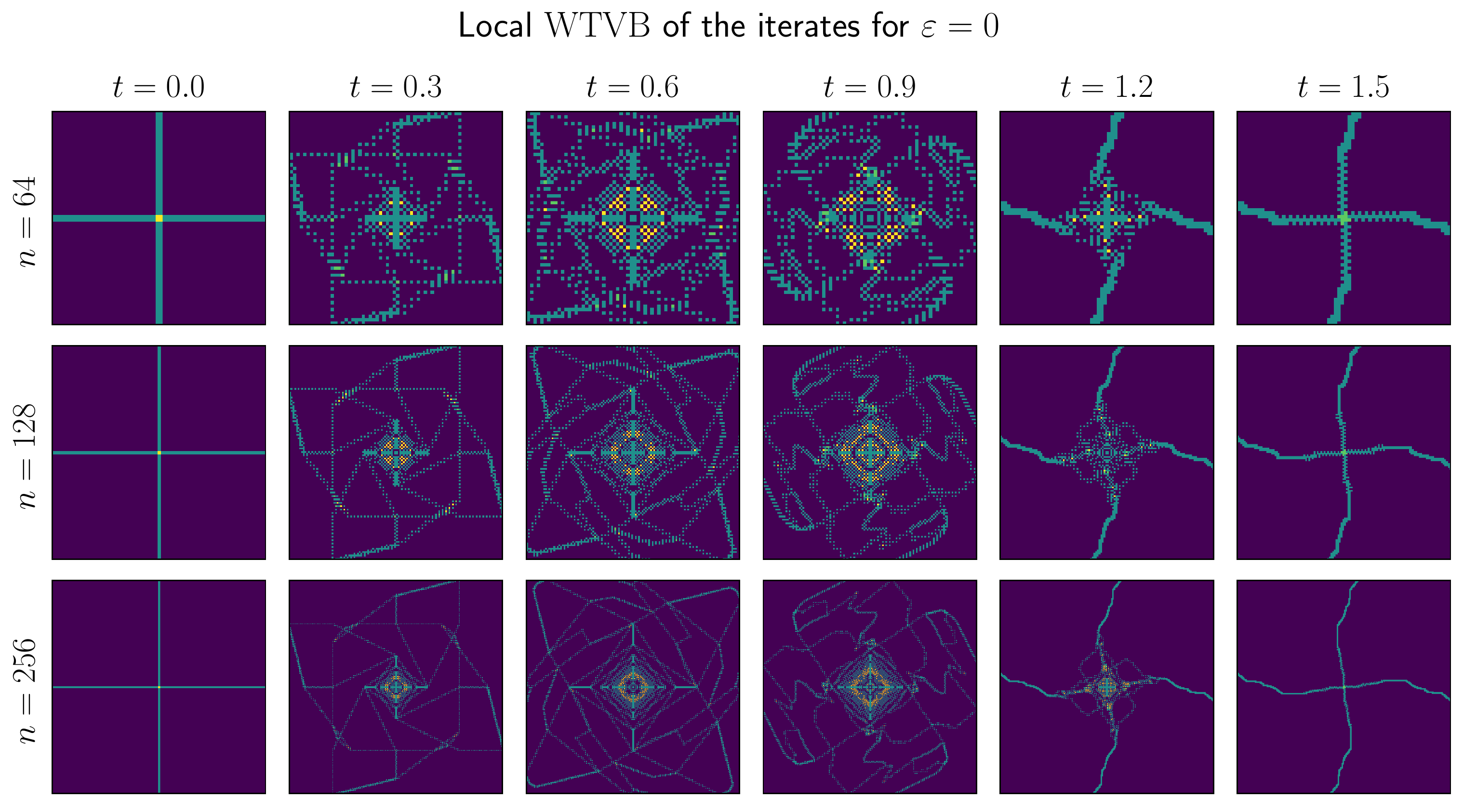}
	
	\vspace{2mm}	
	\includegraphics[width=\linewidth]{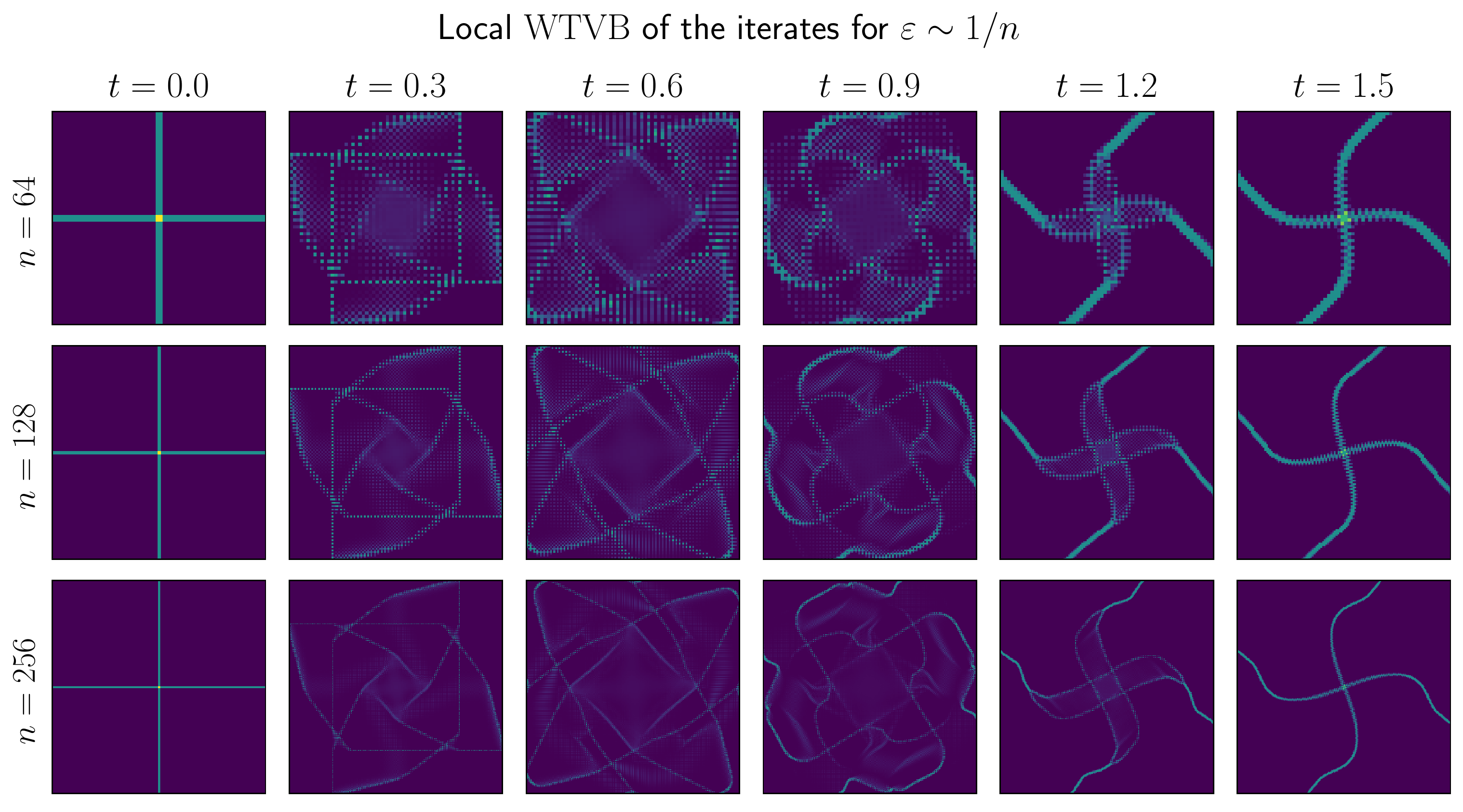}
	\caption{Top, local contributions to $\WTVb$ in the unregularized case. Bottom, local contributions in the regularized case. The (approximately) two-dimensional region that appears in the center for $\veps = 0$ disappears for $\veps \sim 1/n$. 
	}
	\label{fig:semidiscrete-local-WTV}
\end{figure}

In the unregularized case we find that the trajectories for $\veps^n=0$ exhibit increasingly intricate, irregular oscillation patterns near the center of $X$. The local contributions to $\WTVb$ (Figure \ref{fig:semidiscrete-local-WTV}, top) contain one-dimensional boundary contributions between different `regular' oscillation regions and approximately two-dimensional (possibly fractal) non-zero regions corresponding to the `irregular' oscillations. The total $\WTVb$-sums seem to increase logarithmically with $n$ (Figure \ref{fig:semidiscrete-WTV-history}), indicating that there may be no uniform bound.
Note that regions of `regular' alternating oscillations do not contribute to $\WTVb$.

\begin{figure}[bt]
	\centering
	\includegraphics[width=\linewidth]{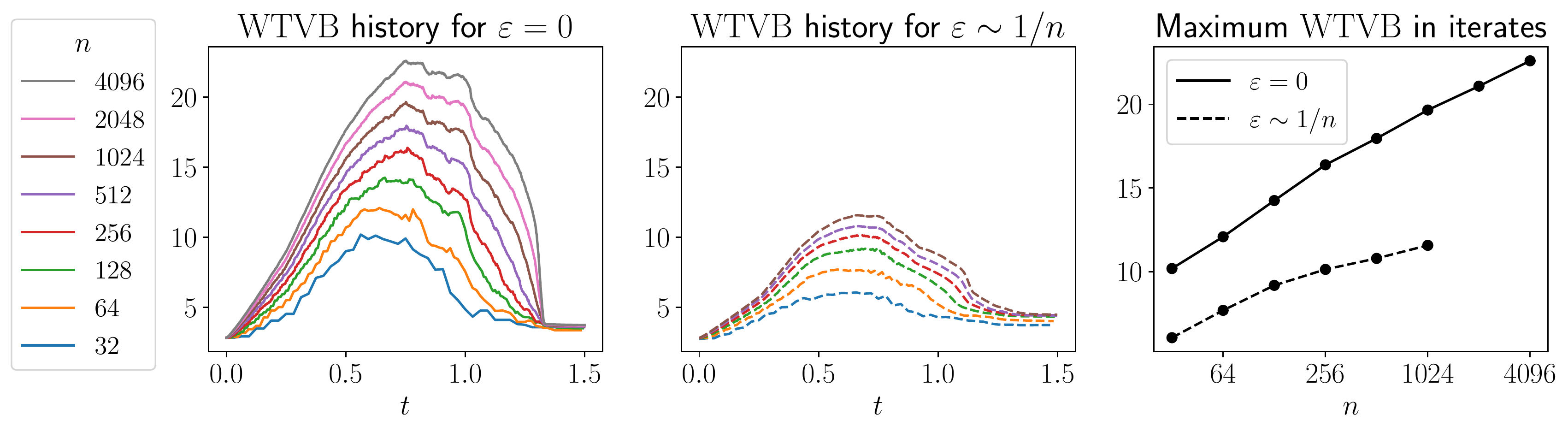}
	\caption{History of the $\WTVb$ for the iterates in Figure \ref{fig:semidiscrete-WTV}. For $\veps=0$ the maximum value over the iterations seems to increase approximately logarithmic in $n$, therefore being seemingly unbounded.
	In the regularized case $\WTVb$ is reduced substantially, but the eventual trend is hard to predict. 
	}
	\label{fig:semidiscrete-WTV-history}
\end{figure}

With $\veps^n \sim 1/n$ the patterns are somewhat subdued. The boundary contributions are smoothed and the irregular oscillations are damped. $\WTVb$ is still increasing in $n$ but the trend is weaker and the question whether it is unbounded is less clear.

\section{Conclusion and outlook}
\label{sec:Conclusion}
\paragraph{Summary.}
In this article we derived a description of the asymptotic limit dynamics of a family of domain decomposition algorithms for optimal transport (regularized and unregularized, discretized and continuous) as the size of the individual cells tends to zero.
To be able to analyze a `pointwise' limit of the cell problems, a sufficiently strong convergence of trajectories, at the level of the disintegrations along the first marginal had to be established first. We introduced a suitable metric for convergence and established said convergence under the assumption of an oscillation bound. Preliminary results on the validity of the oscillation bound assumption were given.
Subsequently, we proved $\Gamma$-convergence of the cell problems to a limit problem where the cells have shrunk to single points.
The trajectories generated by the domain decomposition algorithm could then be shown to converge to a limit trajectory which is driven by a `horizontal' momentum field, which is in turn extracted from solutions to the limit cell problems.
Numerical examples were given to illustrate several interesting aspects of the trajectories.

\paragraph{Open questions.}
The result immediately raises several open questions.
\begin{itemize}
\item We conjecture that under suitable conditions the limit trajectory $t \mapsto \bmpi_t$ converges to a stationary coupling $\pi_\infty$ as $t \to \infty$ and that this coupling is concentrated on the graph of a map (for suitable cost functions $c$, for instance of `McCann-type' \cite{McCannGangboOTGeometry1996}, $c(x,y)=h(x-y)$ for strictly convex $h : \R^d \to \R$).
\item Under what conditions is $\pi_\infty$ a minimizer of the transport problem? Several counter-examples with different underlying mechanisms that prevent optimality are presented or discussed in this article:
\begin{itemize}
\item In the discretized case, with a fixed number of points per basic cell and with $\veps^n=0$, local optimality on the composite cells does not necessarily induce global optimality of the problem (Section \ref{sec:HessianExample}).
\item In this case, for $\veps^n>0$, convergence to the global minimizer was shown in \cite{BoSch2020} for each fixed $n$.
But we conjecture that if $\veps^n \to 0$ too fast, the the asymptotic trajectory may still get stuck in a sub-optimal position.
\item Conversely, if $\veps^n$ does not tend to zero sufficiently fast ($\eta=\infty$), asymptotically the algorithm freezes in the initial configuration (Figure \ref{fig:example_iterations_flipped_bottleneck}, Theorem \ref{thm:main}).
\item Finally, even in the non-discretized, unregularized setting, where convergence of the algorithm for finite $n$ follows from Benamou's work \cite{BenamouPolarDomainDecomposition1994} and Appendix \ref{sec:Benamou},
the asymptotic trajectory may become stuck in a sub-optimal configuration if the sub-optimality is `concentrated on an interface' and almost all cells are locally optimal (Section \ref{sec:SemidiscreteExample}).
\end{itemize}
\item In cases where convergence to the minimizer can be established, how fast is this convergence?
This would be relevant to estimate the required number of domain decomposition iterations.
Intuitively, domain decomposition (and its asymptotic limit dynamics) resembles a minimizing movement scheme on the set $\Pi(\mu,\nu)$ with respect to a `horizontal $W_\infty$ metric', where particles are only allowed to move in $X$ direction by at most distance 1 along each spatial axis per unit time.
Can this relation be made rigorous?
\item Another open question is the bound on $\WTVb$ that is required for the convergence result of Section \ref{sec:Convergence}. Based on numerical evidence it seems to hold in the vast majority of cases, but possibly not always (Section \ref{sec:WTVExample}). What are sufficient conditions for this to hold and how can Proposition \ref{prop:OscillationsBoundPartialResult} be generalized?
\end{itemize}		

\paragraph{Implications of the main result.}
The main result of the article does not yet fully describe the asymptotic behaviour of the domain decomposition method. Various challenging open questions require further study. But already at this stage it does provide us with some valuable insights.

\cite{BenamouPolarDomainDecomposition1994} and \cite{BoSch2020} establish convergence of domain decomposition to the global minimizer in their respective settings (unregularized, regularized) at \emph{finite scale} $n$.
In this article we give examples for both where convergence fails asymptotically as $n \to \infty$.
Moreover, the former does not address the speed of convergence and the upper bound on the convergence speed given by the latter does not accurately describe the behaviour on `geometric' problems.
The main result of this article establishes that if the iterates are sufficiently regular (satisfying the $\WTVb$-bound), then from convergence to the limit dynamics we deduce that the algorithms at finite scale $n$ all exhibit similar behaviour at iterates $\lfloor t \cdot n \rfloor$ for $t \in \R_+$.
If the limit dynamics could be shown to converge to the optimal solution (or at least to some other stationary state), then convergence at finite scale $n$ will require a number of iterations that is proportional to $n$ to approximate this state.

In addition, the result sheds preliminary light on the efficiency of the coarse-to-fine approach of \cite{BoSch2020}. Assume that from an iterate $\pi^{n,k}$ at scale $n$ an approximation of the iterate $\pi^{2n,2k}$ at scale $2n$ can be obtained by refinement and that the approximation error can be remedied by running a finite, fixed number of additional iterations at scale $2n$. Then, by starting at low $n$, and then repeatedly running a fixed number of iterations, and refining to $2n$, one can obtain an approximation of the limit point $\bmpi_t$ in a number of iterations that is logarithmic in $t$. This is in agreement with the numerical results of \cite{BoSch2020}.

\paragraph{Conclusion.}
In summary we consider the results of this article as an important step in the geometric convergence analysis of the domain decomposition algorithm.
The obtained limit dynamics and its relation to global optimality should be studied further, both for its implications for numerical algorithms and in its own right as a new type of minimization-driven dynamics on measures.

\paragraph{Acknowledgement.}
This work was supported by the Emmy Noether programme of the DFG.

\bibliography{references}{}
\bibliographystyle{plain}

\appendix

\section{Domain decomposition for unregularized optimal transport}
\label{sec:Benamou}

In this section we give an extended version of the convergence proof by Benamou \cite{BenamouPolarDomainDecomposition1994} for domain decomposition in the continuous unregularized setting with quadratic cost.
First, we give an alternative (and shorter) argument for the convergence towards a fixed point. Then we give an extended argument for the continuity of the global Kantorovich potential on more than two partition cells, which was not discussed in detail in the original article. Finally, we weaken the convex overlap principle that was required in Benamou's proof. Originally, for any two cells $X_1$ and $X_2$ with overlapping interior it was required that any line segment starting in $\inter X_1$ and ending in $\inter X_2$ has a non-zero length intersection with the overlap $\inter X_1 \cap \inter X_2$. This is not satisfied by the tilings of staggered cubes considered in this article.
For future reference we consider a relatively general setting, going beyond the concrete partition structure considered in the rest of the article and give statement and proof in a self-contained way.

\paragraph{Setting.}
Let $X$ be a compact, convex polytope of $\R^d$. Let $Y$ be a bounded subset of $\R^d$. Let $\mu \in \prob(X)$ with $\mu \ll \Lebesgue$ and $\nu \in \prob(Y)$.
We assume that the density of $\mu$ has full support, i.e. $\RadNik{\mu}{\Lebesgue} > 0\ \Lebesgue$-a.e.
Set $c(x,y) \assign \|x-y\|^2$ for $(x,y) \in X \times Y$. We are concerned with solving the unregularized optimal transport problem
\begin{align*}
	\inf \left\{ \la c,\pi \ra \middle| \pi \in \Pi(\mu,\nu) \right\}
	\quad \tn{where} \quad \la c,\pi \ra \assign \int_{X \times Y} c\,\diff \pi
\end{align*}
with a domain decomposition algorithm.

Let $(X_i)_{i\in I}$ be a \emph{tiling} of $X$ composed by closed, convex polytopes with non-empty interior. This means $(X_i)_{i\in I}$ is a finite covering of $X$ and $\inter{X_i}\cap  \inter{X_j} = \emptyset$ if $i\neq j$. We call these the \emph{basic cells}.
We do not enforce any regularity condition at the intersection of basic cells, i.e., there is no problem when a corner of one cell meets the edge of another. 

Let $\partJ=\{J_1,\ldots,J_K\}$ be a finite collection of subsets of $I$, i.e.~$K \in \N$ and $J_k \subset I$ for $k=1,\ldots,K$.
The corresponding unions of basic cells,
\begin{equation}
	X_J = \bigcup_{i\in J} X_i
	,\qquad
	\tn{for all }J\in \partJ
\end{equation}
will be called the \emph{composite cells}.
We only admit $J$ such that $X_J$ is itself a convex polytope and where $(X_J)_{J\in \partJ}$ covers $X$.
The assumption that basic and composite cells are convex polytopes is not minimal but easy to satisfy in practice and allows for a simpler presentation of the proof arguments.

We say that composite cells $J$ and $\hat{J}$ are \emph{neighbors} if their intersection $X_J\cap X_{\hat{J}}$ has non-empty interior.
$J$ and $J'$ will be \emph{adjacent} when their intersection is non-empty, but has empty interior. 
Furthermore, we make the following assumption:
\begin{assumption}\label{assumption:BoundaryCovering}
	For $J$ and $J'$ adjacent, for any $x\in X_J\cap X_{J'}$ there exists some $\hat{J}$ neighboring both $J$ and $J'$ such that $x\in X_{\hat{J}}$. 
\end{assumption} 

Again, this assumption is easy to satisfy in practice. It will ensure that when the intersection of two cells is not big enough to guarantee consistency of the dual potential, there is always a third cell to arbitrate on the intersection.

\paragraph{Domain decomposition algorithm.}
Now we briefly describe the domain decomposition algorithm in this setting. 
We will iterate cyclically over the cells in $\partJ$, i.e.~we set $J_k \assign J_{(k-1)\% K +1}$ for integers $k>K$ where $\%$ denotes the modulo operator.
An initialization $\pi^0 \in \Pi(\mu,\nu)$ is given. Then, at iteration $k\ge1$ solve the problem 
\begin{equation}
	\pi_{J_k} \assign \argmin \{\la c, \hat{\pi}\ra\mid \hat{\pi}\in \Pi(\mu_{J_k} , \nu^k_{J_k}\}
\end{equation}
with $\mu_{J_k} \assign \mu \restr X_{J_k}$ and $\nu^k_{J_k} \assign \proj_Y \pi^{k-1} \restr( X_{J_k} \times Y))$. The solution is unique since $\mu_{J_k} \ll \Lebesgue$. The next iterate is then given by
\begin{equation}
	\pi^k \assign \pi_{J_k} + \pi^{k-1}\restr ((X \setminus X_{J_k})\times Y).
\end{equation}
Arguing as in \cite[Proposition 3.3]{BoSch2020}, we find that $\pi^k \in \Pi(\mu,\nu)$ for all $k \in \N$ and it is clear that the sequence of transport costs $(\la c,\pi^k \ra)_k$ is non-increasing.
The domain decomposition setting considered in the rest of this article, with two sets of staggered partitions is recovered by setting $\partJ = \partA \cup \partB$, and placing all $A$ cells first in the ordering of $\partJ$.
The staggered partition scheme ensures that Assumption \ref{assumption:BoundaryCovering} is satisfied.
Clearly, successive iterations where the interiors of the cells $X_J$ do not overlap can be carried out in parallel. 

\begin{theorem}
	\label{thm:DomDecAlgorithmConvergence}
	In the setting stated above the iterates of the domain decomposition algorithm converge to the unique global minimizer of the unregularized optimal transport problem.
\end{theorem}
\begin{proof}
	\textbf{Part 1: Existence of a fixed point.}
	This part is an adaptation of arguments from \cite[Proposition 3.6]{BoSch2020} to the unregularized setting.
	Denote by $S$ the solving map
	$$S: \{ \pi \in \measp(X \times Y) : \proj_X\pi \ll\Lebesgue\} \ni \pi \mapsto \argmin \{\la c, \hat{\pi}\ra\mid \hat{\pi}\in \Pi(\proj_X \pi, \proj_Y \pi)\}.$$
	Since $\proj_X\pi \ll\Lebesgue$ the minimizer is unique and therefore the mapping is well defined. To show that it is continuous, take a sequence $(\pi^n)_n$ weak* converging to $\pi$. Then the marginals of $\pi^n$ converge to those of $\pi$, and by stability of OT \cite[Theorem 5.20]{Villani-OptimalTransport-09}, any cluster point of $(S(\pi^n))_n$ must solve the OT problem for $\Pi(\proj_X \pi, \proj_Y \pi)$. But since the solution of this problem is unique (and given by $S(\pi)$), and any subsequence of $(S(\pi^n))_n$ must have a cluster point, it follows that $\lim_{n\rightarrow\infty} S(\pi^n) = S(\pi)$.
	
	Now call $F_J$ the function that performs an iteration on cell $J \in \partJ$. $F_J$ is built by restriction to the composite cell $X_J$, solving the partial problem and merging the solution with the rest of the coupling. Since the restriction in this setting is weak* continuous (because $\mu$ does not give mass to cell boundaries, cf.~proof of Proposition \ref{prop:MomentumConvergence} for a similar argument),
	$F_J$ is weak* continuous for all $J\in \partJ$.
	
	Consider now the subsequence $(\pi^{K \cdot \ell})_\ell$, i.e.,~the iterates obtained after every completion of a cycle over $\partJ$, and let $\pi^\ast$ be a cluster point of this subsequence. By continuity of $F_{J_1}$, $F_{J_1}(\pi^\ast)$ is therefore a cluster point of $(\pi^{K \cdot \ell+1})_\ell$ and more generally, all $(F_{J_k} \circ F_{J_{k-1}} \circ \ldots \circ F_{J_1})(\pi^\ast)$ must be cluster points of $(\pi^\ell)_\ell$ for $k=1,\ldots,K$.
	Since the sequence $(\la c,\pi^k \ra)_k$ is non-increasing, all cluster points of $(\pi^k)_k$ must yield the same transport cost. This means that applying $F_{J_1}$ to $\pi^\ast$ does not decrease its score. This can only happen if $F_{J_1}(\pi^\ast)=\pi^\ast$, since any change implies a decrease in score (as solutions are unique). Therefore, $(F_{J_{K}}\circ ... \circ F_{J_1})(\pi^\ast)=\pi^\ast$.
	This proves that $\pi^\ast$ is optimal on each composite cell.
	
	\textbf{2. Local convex potentials.}
	Now, since $\pi^*$ is optimal on each composite cell $X_J$, each composite cell is the closure of its interior, and the latter is connected, by Brenier's polar factorization \cite{MonotoneRerrangement-91} (here we use $\mu \ll \Lebesgue$) $\pi^\ast$ is concentrated on the graph of a map $T : X \to Y$ and on every composite cell $J$, $T|_{X_J}$ is ($\mu \restr X_J$-almost everywhere) the gradient of a convex function $\phi_J: \inter X_J \rightarrow \R$, which is unique up to constant shifts
	(here we use that the density of $\mu$ has full support).
	Since $Y$ is bounded and $T$ takes only values in $Y$ (almost everywhere), all $\phi_J$ are Lipschitz continuous and therefore there is a unique continuous extension of $\phi_J$ to the boundary of $X_J$.
	The rest of the proof shows that it is possible to choose the constant shifts in all $\phi_J$ such that one obtains a global convex function $\phi: X \rightarrow \R$ with $\phi|_{X_J} = \phi_J$ on all composite cells $J$ and consequently $T = \nabla \phi$ $\mu$-almost everywhere in $X$, which implies that $T$ is the unique optimal transport map and therefore that $\pi^\ast$ is optimal.
	Since all cluster points of $(\pi^k)_k$ have the same transport cost, the whole sequence must therefore converge to $\pi^\ast$.
	
	\textbf{3. Local consistency of potentials.}
	Now we show that
	\begin{align}
		\tn{for any $J, J' \in \partJ$, if $x, x' \in X_J\cap X_{J'}$, then $\phi_J(x')  -\phi_{J}(x) = \phi_{J'}(x')  -\phi_{J'}(x)$.}
		\label{eq:LocalConsistencyPotentials}
	\end{align}
	We first show this when $J, J'$ are neighboring cells. Then, the interior of $X_J\cap X_{J'}$ is non-empty and convex, and thus it is also connected. By restriction \cite[Theorem 4.6]{Villani-OptimalTransport-09}, $\pi^\ast$ is optimal on $X_J\cap X_{J'}$ and by uniqueness of the potentials $\phi_J$, $\phi_{J'}$ up to constant shifts, the partial potentials $\phi_J$ and $\phi_{J'}$ must agree up to a constant on $\inter (X_J\cap X_{J'})$.
	Arguing as before, since $\phi_J$ and $\phi_{J'}$ are Lipschitz continuous and $\ol{\inter(X_J \cap X_{J'})}=X_J \cap X_{J'}$ (implied by their convexity), $\phi_J$ and $\phi_{J'}$ agree up to a constant on the whole, closed set $X_J\cap X_{J'}$.
	
	Now assume $J, J'$ are just adjacent. 
	Denote by $[x,x']$ the line segment between $x$ and $x'$, which is a convex set contained in $X_J\cap X_{J'}$. So it can be covered by (possible multiple) simultaneous neighbors of $J$ and $J'$ (cf.~Assumption \ref{assumption:BoundaryCovering}). 
	Since the subdomains are closed and convex, there are some points $x=x_1, x_2,...,x_m = x'$ such that each segment $[x_i,x_{i+1}]$ is contained in some single neighbor of both $J$ and $J'$, that we denote $\hat{J}_i$.
	Then, 
	\begin{align*}
		\phi_J(x') - \phi_J(x) 
		&= 
		\sum_{i=1}^{m-1} \phi_J(x_{i+1}) - \phi_J(x_i)
		= 
		\sum_{i=1}^{m-1} \phi_{\hat{J}_i}(x_{i+1}) - \phi_{\hat{J}_i}(x_i)
		\\
		&=
		\sum_{i=1}^{m-1} \phi_{J'}(x_{i+1}) - \phi_{J'}(x_i)
		=
		\phi_{J'}(x_m) - \phi_{J'}(x_1) 
		=
		\phi_{J'}(x') - \phi_{J'}(x). 
	\end{align*} 
	\textbf{4. Path integrals on the potential.}
	Intuitively, our goal is to construct the global potential $\phi$ by integrating the vector field $T$. For this we need to show that integrals around closed loops vanish. To avoid regularity issues we express path integrals as suitable sums of finite differences of the local potentials $\phi_J$ instead. Moreover, it will be sufficient to integrate along paths that are composed of a finite number of affine segments.
	
	Let $\gamma : [0,1] \to X$ be a continuous curve, composed of a finite number of affine segments. We call this polygonal. We define the path integral of $\gamma$ as follows: Find a tiling of the interval $[0,1]$ into smaller intervals $[t_i,t_{i+1}]$ for $0=t_1 < t_2 < \ldots < t_m=1$ such that each partial image $\gamma([t_i, t_{i+1}])$ is contained in a single cell $J_i$ for all $i=1,\ldots,m-1$
	(this sequence of cells has nothing to do with the labeling of the subdomains in the partition $\partJ$; we denote them in the same manner to avoid overcrowding the notation).
	Since $\gamma$ is polygonal, this is always possible.
	Then define the path integral as 
	\begin{equation}
		\label{eq:PathIntegral}
		\IntWeight(\gamma)
		\assign 
		\sum_{i=1}^{m-1} \phi_{J_i}(\gamma(t_{i+1})) - \phi_{J_i}(\gamma(t_{i})).
	\end{equation} 
	The value of $\IntWeight(\gamma)$ does not depend on the choice of the tiling nor on the choice of the containing cells: Had we chosen a different tiling, specified by $(s_1,\ldots,s_n)$, with different associated cells, then we could first generate a refined combined tiling, specified by some $(r_1,\ldots,r_p)$, such that each interval $[r_i,r_{i+1}]$ is contained in an interval of both the $t$ and $s$-tiling and we can associate to each $r$-interval the two cells that were associated to that interval by the $t$ and $s$-tiling. The partial image $\gamma([r_i,r_{i+1}])$ is then contained in both of these cells.
	Switching from the $t$-tiling to the $r$-tiling (with the $t$-cells) does not change \eqref{eq:PathIntegral}. By \eqref{eq:LocalConsistencyPotentials} we can then on each $r$-interval swap the $t$-cells for the $s$-cells without changing the value, and finally switch from the $r$-tiling (with the $s$-cells) to the $s$-tiling.
	
	With similar arguments we can show that the path integral is parametrization invariant 
	and reverses its sign if we traverse the path in the opposite direction.

	\textbf{5. Vanishing integrals along cyclic polygonal paths.}
	Consider a polygonal cycle $\gamma : [0,1] \to X$ with $\gamma(0)=\gamma(1)$. We show that the path integral around $\gamma$ vanishes, by using that cycles on $X$ can be continuously contracted to single points, since $X$ has a trivial fundamental group. Assume that for some times $0=t_1<\ldots<t_m=1$, $\gamma$ is affine on each of the intervals $[t_i,t_{i+1}]$ and let $x_0$ be some point in $X$. Let now $\gamma_i : [0,1] \to X$ be a polygonal parametrization of the triangle boundary spanned by the points $(x_0,\gamma(t_i),\gamma(t_{i+1}))$ for $i=1,\ldots,m-1$ (by convexity of $X$ this lies entirely within $X$). The collection of paths $\gamma_i$ traverses each edge $[x_0,\gamma(t_i)]$ exactly once in each direction (using $\gamma(t_1)=\gamma(t_m)$) and each segment $[\gamma(t_i),\gamma(t_{i+1})]$ exactly once in the same direction as $\gamma$. Therefore
	\begin{align*}
		\IntWeight(\gamma) = \sum_{i=1}^m \IntWeight(\gamma_i).
	\end{align*}
	Now, for $i=1,\ldots,m$ consider the intersection of the filled triangle spanned by $(x_0,\gamma(t_i),\gamma(t_{i+1}))$ with all basic cells $X_k$, $k \in I$. For each $k$ this intersection will be a convex polytope of at most $2$ dimensions with a finite number of one-dimensional boundary segments.
	The sum of the integrals of the paths that constitute these boundaries ---which is a finite sum --- naturally yields the whole path integral of $\gamma_i$. But at the same time, each of the small paths on basic cells vanish, since we can take some composite cell $J$ including this basic cell (there is always one because basic cells tile $X$ and composite cells are unions of them that cover $X$) and evaluate \eqref{eq:PathIntegral} just using that cell's potential. This means that the terms in \eqref{eq:PathIntegral} cancel each other for each basic cell cycle, and therefore $\IntWeight(\gamma_i) = 0$. Finally, $\IntWeight(\gamma) = \sum_{i=1}^m \IntWeight(\gamma_i)=0$.
	
	\textbf{6. Global continuous potential.}
	Fix now some $x_0 \in X$ and set
	\begin{equation}
		\phi : X \to \R, \qquad x \mapsto \IntWeight([x_0, x]).
	\end{equation}
	Since $\IntWeight$ vanishes around closed polygonal paths we find for any $x,x' \in X_J$ for $J \in \partJ$ that
	$$ \phi(x)-\phi(x')=\IntWeight([x,x'])=\phi_J(x)-\phi_{J}(x')$$
	where we used that the segment $[x,x']$ lies entirely within $X_J$. Moreover, since all $\phi_J$ are equi-Lipschitz (by the boundedness of $Y$), we have $\IntWeight([x,x']) \leq L \cdot \|x-x'\|$ for some $L < \infty$ and thus $\phi$ is also Lipschitz continuous. Therefore, the local potentials $\phi_J$ can indeed be combined to a globally continuous function.
	
	\textbf{7. Convexity of the global potential.}
	At last, we check that $\phi$ is convex. We show first an auxiliary argument: the set
	\begin{equation}
		\BoundarySet 
		\assign 
		\{x \in \inter X ,\, x\notin \inter X_J  \tn{ for all $J\in \partJ$} \}
	\end{equation}
	has finite $(d-2)$-dimensional volume. We will prove this by showing that $\BoundarySet$ is contained in the set
	\begin{equation}
		\mathfrak{F}
		\assign
		\bigcup\{F_1 \cap F_2 \mid F_1, F_2 \tn{ faces of composite cells, $F_1\cap F_2$ being $(d-2)$-dimensional}\},
	\end{equation}
	where we allow $F_1$ and $F_2$ to be faces of the same composite cell. Note that by finiteness of $\partJ$, $\mathfrak{F}$ must have finite $(d-2)$-dimensional volume.
	
	Clearly, $\BoundarySet$ is contained in $\bigcup_{J \in \partJ} \partial X_J$. Let now $x \in \BoundarySet$ lie on $\partial X_J$ for some $J$. Either $x$ lies at the intersection of two faces of $X_J$ (so it is already in $\mathfrak{F}$), or in the relative interior of some face $F_1$. Let us assume the latter case, and let $H$ be the affine hyperplane containing $F_1$.
	By convexity, $X_J$ lives in one of the two closed halfspaces defined by $H$; call it $V_1$, and the other $V_2$. Since $x \notin \partial X$, we must have $x \in \partial X_{J'}$ for at least one other cell $X_{J'}$ that intersects the interior of $V_2$ (a way to see this is to take a sequence $(x_n)_n$ in $\inter V_2$, converging to $x$ and selecting a composite cell $X_{J'}$ that contains an infinite subsequence). 
	
	Since $x\in \partial X_{J'}$, it is again either at the intersection of two faces of $X_{J'}$ (and therefore in $\mathfrak{F}$) or in the relative interior of some face $F_2$.
	In the latter case, either $F_2$ has $(d-2)$-dimensional intersection with $F_1$ or they lie in the same hyperplane $H$. The second possibility would imply that $X_{J'}$ is contained in $V_2$, and therefore must be adjacent to $X_J$. Then, by Assumption \ref{assumption:BoundaryCovering} there exists a third cell $X_{\hat{J}}$ neighboring both $X_J$ and $X_{J'}$, with some face $F_3$ that contains $x$. This face cannot lie in the hyperplane $H$ because then $X_{\hat{J}}$ would live either in $V_1$ or in $V_2$, and could not intersect the interior of both $X_J$ and $X_{J'}$. We conclude then that $F_1\cap F_3$ has finite $(d-2)$-dimensional volume, and therefore $x\in \mathfrak{F}$. This proves that $\BoundarySet$ is contained in the set $\mathfrak{F}$, and consequently has finite $(d-2)$-dimensional volume.
	
	Now we will prove that $\phi$ is convex on every segment $[x,x']$ for $x,x'\subset X$; since $X$ is convex, the segment is always contained in $X$. If the segment does not cross the set $\BoundarySet$, we can cover $[x,x']$ with relatively open sub-segments, each of which lies in the interior of some $X_J$, $J \in \partJ$. $\phi$ is convex on each of the sub-segments and consequently on he whole segment.
	
	Finally, for a segment $[x,x']$ which does cross the set $\BoundarySet$, we can find a sequence of segments $([x_n, x_n'])_n$ converging to $[x,x']$ such that none of them crosses $\BoundarySet$; the reason is that $\BoundarySet$ has finite $(d-2)$-dimensional volume, and such a set cannot separate open sets of $\R^d$. Thus, $\phi$ is convex on each $[x_n, x'_n]$, and since the convexity condition passes to the limit, $\phi$ is convex also on the segment $[x,x']$.
\end{proof}

\section{Proof of Proposition \ref{prop:OscillationsBoundPartialResult}}
\label{app:ProoWTVbBound}
\textit{Step 1: optimal transport in one dimension.}
The proof relies heavily on monotonicity properties of optimal transport in one dimension, for details we refer to \cite[Chapter 2]{SantambrogioOT}. Let $\mu_1, \mu_2, \nu_1, \nu_2 \in \measp(\R)$ (with bounded first moments) and $\mu_i(\R)=\nu_i(\R)$ for $i=1,2$. If $x_1 \leq x_2$ for $\mu_1 \otimes \mu_2$ almost all $(x_1,x_2) \in \R^2$ then we write $\mu_1 \measleq \mu_2$. By monotonicity of optimal transport and convexity (or sub-additivity) of Wasserstein distances we obtain
\begin{align}
	\label{eq:OTMonotonous}
	\WoR(\mu_1+\mu_2,\nu_1+\nu_2) \leq \WoR(\mu_1,\nu_1) + \WoR(\mu_2,\nu_2)
	\quad \tn{with equality if } \mu_1 \measleq \mu_2 \,\wedge\, \nu_1 \measleq \nu_2.
\end{align}

\smallskip
\noindent
\textit{Step 2: $\WTVb$ bound.}
In the following recall that $\mu=\mu^n$ assigns equal mass to each basic cell, i.e.~$m_i^n=1/n$ for all basic cells $i \in I^n$ and $m_J^n=2/n$ for all composite cells $J \in \partAn$ and all interior composite cells in $\partBn$. This will simplify switching between normalized basic and composite cell marginals $\rho_i^{n,k}$ and $\rho_J^{n,k}$ as introduced in Section \ref{sec:DomDecNotation}, item \ref{item:NotationPartialMarginals}. Some attention must be paid on the boundaries of the composite $B$-partition.

We encourage the reader to compare the following arguments against the iterations shown in Figure \ref{fig:example_iterations_detail}. Let $k \in \N$, $k$ even: this means, iteration $k$ is on the $B$-partition (for $k>0$), $k+1$ is on the $A$-partition. Let $J=\{i,i+1\}$, $J'=\{i+2,i+3\}$ be two successive composite cells in $\partnkk=\partAn$ (i.e.~$i \in I^n$ is even). This implies $\rho_i^{n,k+1} \measleq \rho_{i+1}^{n,k+1}$ and $\rho_{i+2}^{n,k+1} \measleq \rho_{i+3}^{n,k+1}$ (because we just updated the composite cells $\{i,i+1\}$ and $\{i+2,i+3\}$ and $h$ is strictly convex) and therefore with \eqref{eq:OTMonotonous}
\begin{align}
	\WoY(\iter{\rho_J}{n,k+1}, \iter{\rho_{J'}}{n,k+1})
	&=
	\frac12(
	\WoY(\iter{\rho_{i}}{n,k+1}, \iter{\rho_{i+2}}{n,k+1})
	+
	\WoY(\iter{\rho_{i+1}}{n,k+1}, \iter{\rho_{i+3}}{n,k+1})
	),
	\label{eq:WTVb_equals_WTV_k_even}
	\intertext{and for the preceeding iteration (but on the same composite cells)}
	\WoY(\iter{\rho_J}{n,k}, \iter{\rho_{J'}}{n,k})
	&\le
	\frac12(
	\WoY(\iter{\rho_{i}}{n,k}, \iter{\rho_{i+2}}{n,k})
	+
	\WoY(\iter{\rho_{i+1}}{n,k}, \iter{\rho_{i+3}}{n,k})
	).
	\label{eq:WTVb_smaller_WTV_k_odd}
\end{align}
Now we enumerate the $A$ composite cells as $\{J_j\}_{j=1}^{n/2}$ and the $B$ composite cells as $\{\hat{J}_j\}_{j=0}^{n/2}$. Then, since $k+1$ is odd (which corresponds to an $A$-iteration) and recalling that $\iter{\rho_J}{n,k+1} = \iter{\rho_J}{n,k}$ for all $J\in \partnkk$, \eqref{eq:CompCellMarginalPreservation}, we have
\begin{align}
	\WTVb(\iter{\pi}{n,k+1})
	&=
	\sum_{i=0}^{n-3}
	\WoY(\rho_i^{n,k+1}, \rho_{i+2}^{n,k+1} )
	\overset{\eqref{eq:WTVb_equals_WTV_k_even}}{=}
	2\sum_{j=1}^{n/2-1}
	\WoY(\iter{\rho_{J_j}}{n,k+1}, \iter{\rho_{J_{j+1}}}{n,k+1})
	\nonumber
	\\
	&=
	2\sum_{j=1}^{n/2-1} \WoY(\iter{\rho_{J_j}}{n,k}, \iter{\rho_{J_{j+1}}}{n,k})
	\overset{\eqref{eq:WTVb_smaller_WTV_k_odd}}{\le}
	\sum_{i=0}^{n-3}
	\WoY(\iter{\rho_{i}}{n,k}, \iter{\rho_{i+2}}{n,k})
	=\WTVb(\iter{\pi}{n,k}).
	\label{eq:WTVbkOdd}
\intertext{Let now $k>0$ (still even). We now repeat the argument to go back another iteration but need to be careful about the boundary cells of the $B$-partition.}
	\WTVb(\iter{\pi}{n,k}) &=
	\WoY(\iter{\rho_{0}}{n,k}, \iter{\rho_{2}}{n,k})
	+
	\WoY(\iter{\rho_{n-3}}{n,k}, \iter{\rho_{n-1}}{n,k})
	+
	\sum_{i=1}^{n-4}
	\WoY(\iter{\rho_{i}}{n,k}, \iter{\rho_{i+2}}{n,k})
	\nonumber
	\\
	&\overset{\eqref{eq:WTVb_equals_WTV_k_even}}{=}
	\WoY(\iter{\rho_{0}}{n,k}, \iter{\rho_{2}}{n,k})
	+
	\WoY(\iter{\rho_{n-3}}{n,k}, \iter{\rho_{n-1}}{n,k})
	+
	2\sum_{j = 1}^{n/2 -2}
	\WoY(\iter{\rho_{\hat{J}_j}}{n,k}, \iter{\rho_{\hat{J}_{j+1}}}{n,k}).
	\label{eq:Bound1DStep1}
\end{align}
Using again preservation of the $Y$-marginals over composite cells at iteration $k$, the last term in \eqref{eq:Bound1DStep1} can be bounded as
\begin{align}
	2\sum_{j = 1}^{n/2 -2}&
	\WoY(\iter{\rho_{\hat{J}_j}}{n,k}, \iter{\rho_{\hat{J}_{j+1}}}{n,k})
	=
	2\sum_{j = 1}^{n/2-2}
	\WoY(\iter{\rho_{\hat{J}_j}}{n,k-1}, \iter{\rho_{\hat{J}_{j+1}}}{n,k-1})
	\overset{\eqref{eq:WTVb_smaller_WTV_k_odd}}{\le}
	\sum_{i = 1}^{n-4}
	\WoY(\iter{\rho_{i}}{n,k-1}, \iter{\rho_{i+2}}{n,k-1})
	\nonumber
	\\
	&=
	\WTVb(\iter{\pi}{n,k-1})
	-
	\WoY(\iter{\rho_{0}}{n,k-1}, \iter{\rho_{2}}{n,k-1})
	-
	\WoY(\iter{\rho_{n-3}}{n,k-1}, \iter{\rho_{n-1}}{n,k-1}).
	\label{eq:Bound1DStep2}
\end{align}
Now notice that $\iter{\rho_{0}}{n,k-1} = \iter{\rho_{0}}{n,k}$, since during iteration $k$ (which is a $B$ iteration) the cell $i=0$ is the only one in its corresponding composite cell; for the same reason, $\iter{\rho_{n-1}}{n,k-1} = \iter{\rho_{n-1}}{n,k}$. Using this fact and the triangle inequality, one obtains
\begin{align}
	\WoY(\iter{\rho_{0}}{n,k}, \iter{\rho_{2}}{n,k})
	-
	\WoY(\iter{\rho_{0}}{n,k-1}, \iter{\rho_{2}}{n,k-1})
	&=
	\WoY(\iter{\rho_{0}}{n,k}, \iter{\rho_{2}}{n,k})
	-
	\WoY(\iter{\rho_{0}}{n,k}, \iter{\rho_{2}}{n,k-1})
	\nonumber
	\\
	&\le
	\WoY(\iter{\rho_{2}}{n,k}, \iter{\rho_{2}}{n,k-1}).
	\label{eq:Bound1DStep3}
\end{align}
After an analogous consideration for the other boundary term, combining \eqref{eq:Bound1DStep1}, \eqref{eq:Bound1DStep2}, and \eqref{eq:Bound1DStep3} one obtains
\begin{align}
	\WTVb(\iter{\pi}{n,k}) \leq \WTVb(\iter{\pi}{n,k-1})
	+
	\WoY(\iter{\rho_2}{n,k}, \iter{\rho_2}{n,k-1})
	+
	\WoY(\iter{\rho_{n-3}}{n,k}, \iter{\rho_{n-3}}{n,k-1})
	\label{eq:WTVbkEven}
\end{align}

\smallskip
\noindent
\textit{Step 3: controlling the boundary contributions.}
Using once more the preservation of $Y$-marginals in composite $B$-cells during $B$-iterations, \eqref{eq:CompCellMarginalPreservation}, and recall that $k$ is still even, we observe
\begin{align*}
	\iter{\rho_1}{n,k} +\iter{\rho_2}{n,k} = \iter{\rho_1}{n,k-1} + \iter{\rho_2}{n,k-1}
	\qquad \Rightarrow \qquad
	\iter{\rho_1}{n,k} - \iter{\rho_1}{n,k-1} =  -\left(\iter{\rho_2}{n,k} - \iter{\rho_2}{n,k-1}\right)
\end{align*}
and using that the Wasserstein-1 distance between two measures only is a function of the difference of them (and is invariant under the order of the difference), see \eqref{eq:KantRubin}, this implies
\begin{align}
	\label{eq:WTVbBoundaryLeft}
	\WoY(\iter{\rho_2}{n,k}, \iter{\rho_2}{n,k-1}) = \WoY(\iter{\rho_1}{n,k}, \iter{\rho_1}{n,k-1})
	\leq 2\WoY(\rho_{J_1}^{n,k},\rho_{J_1}^{n,k-1})
\end{align}
where in the inequality we used $J_1=\{0,1\}$, $\rho_0^{n,k}=\rho_0^{n,k-1}$ ($B$-iteration on the boundary cell $\hat{J}_0=\{0\}$) and \eqref{eq:OTMonotonous}. The same argument on the right boundary yields
\begin{align}
	\label{eq:WTVbBoundaryRight}
	\WoY(\iter{\rho_{n-3}}{n,k}, \iter{\rho_{n-3}}{n,k-1}) = \WoY(\iter{\rho_{n-2}}{n,k}, \iter{\rho_{n-2}}{n,k-1})
	\leq 2\WoY(\rho_{J_{n/2}}^{n,k},\rho_{J_{n/2}}^{n,k-1})
\end{align}
Now, we first plug \eqref{eq:WTVbBoundaryLeft} and \eqref{eq:WTVbBoundaryRight} into \eqref{eq:WTVbkEven}, then add the artificial zero
\begin{align}
	\label{eq:WTVbArtificialZero}
0=2\WoY(\rho_{J_{1}}^{n,k},\rho_{J_{1}}^{n,k+1})+2\WoY(\rho_{J_{n/2}}^{n,k},\rho_{J_{n/2}}^{n,k+1})
\end{align}
(which follows again from the preservation of composite cell marginals during the respective iterations) to the right hand side of \eqref{eq:WTVbkOdd}, and finally repeatedly use \eqref{eq:WTVbkOdd} and \eqref{eq:WTVbkEven} to bound $\WTVb(\pi^{n,k+1})$ by $\WTVb(\pi^{n,0})$, we get for any $k \in \N$
\begin{align}
	\WTVb(\pi^{n,k}) & \leq \WTVb(\pi^{n,0}) + 2 \sum_{\ell=1}^{k} \left[
		\WoY(\rho_{J_1}^{n,\ell},\rho_{J_1}^{n,\ell-1})
		+\WoY(\rho_{J_{n/2}}^{n,\ell},\rho_{J_{n/2}}^{n,\ell-1}) \right].
		\label{eq:WTVbTelescopicPre} 
\end{align}
\smallskip
\noindent
\textit{Step 4: telescopic sum.}
We now show that the sum in \eqref{eq:WTVbTelescopicPre} has a telescopic structure (intuitively, mass within $J_1$ never moves `up', only `down', cf.~Figure \ref{fig:example_iterations_detail}). In the following, denote by
\begin{align*}
	F^{n,k}_i(y) \assign \rho^{n,k}_i((-\infty,y]) \qquad \tn{for } y \in \R
\end{align*}
the cumulative distribution functions of the normalized basic cell marginals. One finds \cite[Proposition 2.17]{SantambrogioOT} that
$$\WoY(\rho^{n,k}_i,\rho^{n,k-1}_i)=\|F^{n,k}_i-F^{n,k-1}_i\|_1 \assign \int_{\R} |F^{n,k}_i-F^{n,k-1}_i|\diff \Lebesgue.$$
Of course the formula holds also for the cumulative distribution functions of composite cells. Now we show that $F^{n,k}_1 \geq F^{n,k-1}_1$ for $k$ even (mass moves `down'). Indeed, after the $B$-iteration in iteration $k$ on $\hat{J}_1=\{1,2\}$, one has $\rho^{n,k}_1 \measleq \rho^{n,k}_2$ (by strict convexity of $h$ and monotonicity of optimal transport in one dimension). This implies that there is some $y \in \R$ such that
\begin{align*}
	F^{n,k}_1(y') & = 1 \geq F^{n,k-1}_1(y') & & \tn{ for $y'\geq y$, and} \\
	F^{n,k}_1(y') & = F^{n,k}_1(y')+\underbrace{F^{n,k}_2(y')}_{=0} = F^{n,k-1}_1(y')+F^{n,k-1}_2(y')
	\geq F^{n,k-1}_1(y') & & \tn{ for } y'<y,
\end{align*}
where we used $\rho^{n,k}_1+\rho^{n,k}_2=\rho^{n,k-1}_1+\rho^{n,k-2}_1$ (from which equality of the cumulative distributions follows). Combining this, for the first part of the sum in \eqref{eq:WTVbTelescopicPre} we obtain
\begin{align*}
	2 \sum_{\ell=1}^{k} \WoY(\rho_{J_1}^{n,\ell},\rho_{J_1}^{n,\ell-1}) &
	= \sum_{\ell=1}^{k} \left\|F_0^{n,\ell}+F_1^{n,\ell}-F_0^{n,\ell-1}-F_1^{n,\ell-1}\right\|_1. \\
\intertext{For $\ell$ odd, the function within $\|\cdot\|_1$ must be zero a.e.~due to \eqref{eq:WTVbArtificialZero}. For $\ell$ even, one has $F_0^{n,\ell}=F_0^{n,\ell-1}$ and we just showed $F_1^{n,\ell}\geq F_1^{n,\ell-1}$, hence the function within $\|\cdot\|_1$ is non-negative. Therefore, we may pull the sum into the norm and obtain}
	& = \left\|F_0^{n,k}+F_1^{n,k}-F_0^{n,0}-F_1^{n,0}\right\|_1=2\,\WoY(\rho_{J_1}^{n,k},\rho_{J_1}^{n,0}).
\end{align*}
The same argument applies to the right boundary term and thus \eqref{eq:WTVbTelescopicPre} becomes
\begin{align*}
	\WTVb(\pi^{n,k}) & \leq \WTVb(\pi^{n,0}) + 2 \left[
		\WoY(\rho_{J_1}^{n,k},\rho_{J_1}^{n,0})
		+\WoY(\rho_{J_{n/2}}^{n,k},\rho_{J_{n/2}}^{n,0}) \right]
		\nonumber \\
	& \leq \WTVb(\pi^{n,0}) + 4\, \diam Y
\end{align*}
where we used the boundedness of $Y$ in the second line.

\smallskip
\noindent
\textit{Step 5: bound on starting variation.}
We show that $\WTVb(\iter{\pi}{n,0}=\piInitN=\piInit)$ is bounded by $\WTV(\piInit)$:
\begin{align}
	\WTVb(\piInit)
	&=
	\sum_{i = 0}^{n-3}
	\WoY(\iter{\rho_{i}}{n,0}, \iter{\rho_{i+2}}{n,0})
	\le
	\sum_{i = 0}^{n-3}
	\int_{x^n_i-1/2n}^{x^n_i+1/2n}
	\WoY(\pi_{\text{init},x}, \pi_{\text{init},x+2/n})
	\cdot n\,\diff x
	\\
	\intertext{where we used that $\rho_i^{n,0}$ is obtained by averaging $\piInit$ over basic cells, sub-additivity of $\WoY$, and $m_i^n = 1/n$, }
	\label{eq:BVLemmaMetricSpaceUse}
	&=
	n \cdot
	\int_{0}^{1-2/n} 	\WoY(\pi_{\text{init},x}, \pi_{\text{init},x+2/n})
	\diff x
	\le
	n \cdot
	\frac{2}{n}
	\WTV(\piInit)
	=
	2\,\WTV(\piInit),
\end{align}
where in the inequality we used Lemma \ref{lemma:IntegralDifferenceTV} in its formulation for metric space-valued functions (cf.~\cite[Lemma 3.2(ii)]{AmbrosioMetricBVFunctions1990}). This provides the desired uniform bound
\[
\WTVb(\iter{\pi}{n,k}) \le 2\WTV(\piInit) + 4\diam Y.
\]

\section{Proof of Lemma \ref{lem:RegularityMuN}}
\label{sec:RegularityMuN}
The proof hinges on the Lebesgue differentiation theorem which we gather from \cite{RudinRealAndComplexAnalysis}.

\begin{theorem}[Lebesgue points \protect{\cite[Theorems 7.7, 7.10, Definition 7.9]{RudinRealAndComplexAnalysis}}]
For $x \in X$ we say a sequence $(E_n)_n$ of Borel sets in $X$ \emph{shrinks nicely} to $x$ if there exists $\alpha > 0$ and radii $(r_n)_n$ such that
\[
\lim_{n \to \infty} r_n = 0, \qquad E_n \subset B(x, r_n) \quad \text{ and } \quad \Lebesgue(E_n) \geq \alpha \cdot \Lebesgue(B(x,r_n)).
\]
Note that $x$ may not belong to $E_n$.

Assume $f \in L^1(X)$. Then $\Lebesgue$-a.e.~$x \in X$ is a Lebesgue point of $f$ and for every such $x$ it holds
\[
f(x) = \lim_{n \to \infty} \frac{1}{\Lebesgue(E_n)}\int_{E_n} f(x') \,\diff\Lebesgue(x').
\]
if the sets $(E_n)_n$ shrink nicely to $x$.
\end{theorem}

\begin{proof}[Proof of Lemma \ref{lem:RegularityMuN}]
By assumption $\mu \ll \Lebesgue$, so $\mu$ has a density with respect to $\Lebesgue$ and $\RadNik{\mu}{\Lebesgue} \in L^1(X)$. Therefore
\begin{equation}
\lim_{n \to \infty}
\frac{1}{\Lebesgue(E_n)}
\int_{E_n} \RadNikD{\mu}{\Lebesgue}(x') \diff \Lebesgue(x')
=
\RadNikD{\mu}{\Lebesgue}(x)
> 0
\end{equation}
for $\mu$-a.e.~$x \in X$ and $(E_n)_n$ a corresponding sequence shrinking nicely to $x$.
In particular, if $(A_n)_n$ and $(B_n)_n$ are two sequences shrinking nicely to $x$ such that $\Lebesgue(A_n) = a/n^d$ and $\Lebesgue(B_n) = b/n^d$ for some $a,b>0$, then
\begin{align}
\lim_{n \to \infty}
\frac{\mu(A_n)}{\mu(B_n)}
=
\lim_{n \to \infty}
\overbrace{
	\frac{\Lebesgue(A_n)}{\Lebesgue(B_n)}
}^{a/b}
\overbrace{
	\frac{
		\frac{1}{\Lebesgue(A_n)} \int_{A_n} \RadNik{\mu}{\Lebesgue}(x') \diff \Lebesgue(x')
	}{
		\frac{1}{\Lebesgue(B_n)} \int_{B_n} \RadNik{\mu}{\Lebesgue}(x') \diff \Lebesgue(x')
	}
}^{\tn{Tends to $1$}}
=
\frac{a}{b}.
\label{eq:LimitRatioMeasureMu}
\end{align}

Consider now the first scheme, \ref{item:MuNDirac}, were the mass of each basic cell $i \in I^n$ is collapsed to its center $x_i^n$, so that $\mu^n = \sum_{i \in I^n} m_i^n\delta_{x_i^n}$.
Let $x \in X\setminus \partial X$ be a Lebesgue point of $\RadNik{\mu}{\Lebesgue}$ with $\RadNik{\mu}{\Lebesgue}(x)>0$ (this holds for $\mu$-a.e.~$x \in X$).
Fix $t\in \R_+$.
For any $n\in 2\N$ denote $J^n \assign J_{t,x}^n$ (Definition \ref{def:reference_cell}).
For $n$ sufficiently large, $J^n$ will contain $2^d$ basic cells: this holds whenever $J^n$ is part of $\partAn$ and it will eventually hold for $J^n \in \partBn$ since $x \in X \setminus \partial X$.
Then
	\begin{equation}
	\sigma_{t,x}^n = \sigma_{J^n}^n
	=
	\sum_{b\in \{-1,1\}^d}
	\frac{m_{i(J^n,b)}^n}{m_{J^n}^n} \delta_{b/2}
	=
	\sum_{b\in \{-1,1\}^d}
	\frac{\mu(X_{i(J^n,b)}^n)}{\mu(X_{J^n}^n)} \delta_{b/2}
	\label{eq:sigmaJPointDiscretization}
	\end{equation}
	where $i(J^n, b)$ is the basic cell in composite cell $J^n$ whose center $x_i^n$ is at $x_J^n + b/2n$. Choose any $b\in \{-1,1\}^d$ and define $A^b_n = X^n_{i(J^n, b)}$ and
	$B_n = X^n_{J^n}$. The sequences $(A^b_n)_n$ and $(B_n)_n$ shrink nicely to $x$ and satisfy $\Lebesgue(A^b_n) = 1/n^d$ and $\Lebesgue(B_n) = 2^d/n^d$, so by \eqref{eq:LimitRatioMeasureMu}
	\begin{equation}
	\lim_{n \to \infty} \frac{\mu(X_{i(J^n,b)}^n)}{\mu(X_{J^n}^n)}
	=
	\frac{1}{2^d},
	\label{eq:ConvergenceMassSigma}
	\end{equation}
	and thus $(\sigma_{t,x}^n)_n$ converges as $n\rightarrow\infty$ to
	\begin{equation}
	\sigma = \sum_{b\in \{-1, +1\}^d} \frac{1}{2^d} \delta_{b/2}.
	\end{equation}
	for $\mu$-a.e.~$x \in X$ and all $t \in \R_+$.
	
	Now consider the variant \ref{item:MuNSelf} with $\mu^n = \mu$ for all $n$. The arguments are similar to the previous case. Choose a point $x \in X$ that is a Lebesgue point of $\RadNik{\mu}{\Lebesgue}$ with $\RadNik{\mu}{\Lebesgue}(x)>0$ ($\mu$-a.e.~$x\in X$ fulfills this) and choose a time $t \in \R_+$. Then for any measurable set $A\subset Z$, the sequence $(A_n)_n$ defined by $A_n = (S_{t,x}^n)^{-1}(A)$ shrinks nicely to $x$, and so does the sequence $(B_n)_n$ defined by $B_n = (S_{t,x}^n)^{-1}(Z) = X_{J_{t,x}^n}^n$. Thus, by \eqref{eq:LimitRatioMeasureMu}
	\begin{align*}
	\lim_{n \to \infty}
	\sigma_{t,x}^n(A)
	=
	\lim_{n \to \infty}
	\frac{\mu((S_{t,x}^n)^{-1}(A))}{m_J^n}
	=
	\lim_{n \to \infty}
	\frac{\mu(A_n)}{\mu(B_n)}
	=
	\frac{1}{2^d}\Lebesgue(A)
	\end{align*}
	from which we conclude that $\sigma=\frac{1}{2^d}\Lebesgue\restr Z$.
	
	Finally, in the scheme \ref{item:MuDiracFine}, at $\mu$-a.e.~$x \in X$ and every $t \in \R_+$, the sequence $\sigma^n_{t,x}$ has the same weak* cluster points as in the scheme \ref{item:MuNSelf}, since $s^n \to \infty$. Thus the result follows from the previous part.
\end{proof}

\begin{remark}
	\label{remark:ZQuadrantsSameMass}
	For $\mu^n$ a regular discretization sequence, the limit $\sigma$ assigns mass $1/2^d$ to each of the `quadrants'
	\begin{equation}
		Z_b 
		\assign
		 \{z\in Z \mid \sign(z_\ell) = b_\ell \tn{ for all } \ell = 1,...,d\}
	\end{equation}
	with $b\in \{-1,1\}^d$. This follows quickly analogous to \eqref{eq:ConvergenceMassSigma}: for $\mu$-a.e.~$x \in X$ and all $t \in \R_+$, denoting $J^n = J(t,x,n)$, we obtain
	\begin{equation}
		\sigma(Z_b)
		=
		\frac{\sigma(Z_b)}{\sigma(Z)}
		=
		\lim_{n \to \infty}
		\frac{\sigma_{J^n}^n(Z_b)}{\sigma_{J^n}^n(Z)}
		=
		\lim_{n \to \infty}
		\frac{m_{J^n}^n \sigma_{J^n}^n(Z_b)}{m_{J^n}^n\sigma_{J^n}^n(Z)}
		=
		\lim_{n \to \infty} \frac{\mu(X_{i(J^n,b)}^n)}{\mu(X_{J^n}^n)}
		=
		\frac{1}{2^d}.
	\end{equation}
\end{remark}

\section{Block approximation}
\label{app:BlockApproximation}
For the construction of recovery sequences in Section \ref{sec:Limsup} we use the \emph{block approximation} \cite{Carlier-EntropyJKO-2015} of a transport plan. The idea is to cover the product space with a grid of small cubes and then replace the mass of the original coupling within any cube with a suitable product measure (on that cube). This preserves the marginals and by controlling the size of the cubes one may balance between the entropy and the transport term. For simplicity, we compactly gather the required results and provide a sketch of proof.
\begin{lemma}[Block approximation]\label{lem:block_approximation}
	Let $\Omega \subset \R^d$ be compact, $\mu,\nu \in \prob(\Omega)$, $\gamma\in\Pi(\mu, \nu)$, $\len>0$.
	For each $k \in \Z^d$ let $Q_k \assign [k_1, k_1+1) \times ... \times [k_d,k_d+1)$ and $Q_k^\len \assign \len \cdot Q_k$. We define the \textit{block approximation of $\gamma$} at scale $\len$ as 
	
	\begin{align}
		\label{eq:BlockApproximation}
		\gamma_\len & \assign \sum_{\substack{j,k \in \Z^d:\\
			\mu(Q_j^\len) > 0,\ \nu(Q_k^\len)>0}}
			\frac{\gamma(Q_j^\len \otimes Q_k^\len)}{\mu(Q_j^\len) \cdot \nu(Q_k^\len)}
			\cdot (\mu \restr Q_j^\len) \otimes (\nu \restr Q_k^\len).
	\end{align}
	We find that
	\begin{align*}
		\gamma_\len & \in \Pi(\mu,\nu), &
		W_{\Omega \times \Omega}(\gamma,\gamma_\len) & \leq \len \cdot \sqrt{2d}, &
		\KL(\gamma_\len|\mu \otimes \nu) & \leq C - 2d\log \len.
	\end{align*}
	for a constant $C<\infty$ (only depending on $d$ and the diameter of $\Omega$).
\end{lemma}

\begin{proof}
In \cite{Carlier-EntropyJKO-2015}, $\Omega$ was not compact, but $\mu$ and $\nu$ were assumed to be Lebesgue-absolutely continuous. $\Omega$ compact allows for a much shorter, self-contained proof.

$\gamma_\len \in \Pi(\mu,\nu)$ follows from a direct computation as in \cite[Proposition 2.10]{Carlier-EntropyJKO-2015}. For all measurable $A \subset \Omega$ one has
\begin{align*}
	\gamma_\len(A \times \Omega) = \sum_{\substack{j,k \in \Z^d:\\
			\mu(Q_j^\len) > 0,\ \nu(Q_k^\len)>0}}
			\frac{\gamma(Q_j^\len \otimes Q_k^\len)}{\mu(Q_j^\len)}
			\cdot (\mu \restr Q_j^\len)(A)
			= \sum_{\substack{j \in \Z^d:\\
			\mu(Q_j^\len) > 0}} (\mu \restr Q_j^\len)(A) = \mu(A)
\end{align*}
where we used [$\gamma(Q_j^\len \otimes Q_k^\len)>0$] $\Rightarrow$ [$\nu(Q_k^\len)>0$] and that the first marginal of $\gamma$ is $\mu$ in the second equality. The same argument applies for the second marginal.

Transforming $\gamma$ into $\gamma_L$ only requires rearrangement of mass within each hypercube $Q_j^\len \otimes Q_k^\len$, which has diameter $\len \cdot \sqrt{2d}$ in $\R^{2d}$. This provides the bound $W_{\Omega \times \Omega}(\gamma,\gamma_\len) \leq \len \cdot \sqrt{2d}$, cf.~\cite[Corollary 2.12]{Carlier-EntropyJKO-2015}.

For the entropy bound we adopt \cite[Lemma 2.15]{Carlier-EntropyJKO-2015}, accounting for more general $\mu$, $\nu$ and considerably simplifying the last part due to compactness of $\Omega$:
\begin{align}
	\KL(\gamma_\len \mid \mu\otimes \nu)
	&=
	\int_{\Omega \times \Omega}
	\log\left(
	\RadNikD{\gamma_{\len}}{(\mu\otimes \nu)}
	\right)
	\diff \gamma_{\len} - \gamma_{\len}(\Omega \times \Omega) + (\mu \otimes \nu)(\Omega \times \Omega)
	\nonumber
	\\
	&=
	\sum_{\substack{
			j,k \in\Z^{d}: \\
			\mu(Q_{j}^{\len})>0,\, \nu(Q_{k}^{\len})>0
		}}
	\log\left(
	\frac{\gamma(Q_j^\len \otimes Q_k^\len)}{\mu(Q_j^\len)\nu(Q_k^\len)}
	\right)
	\gamma(Q_j^\len \otimes Q_k^\len)
	\nonumber
	\\
	\intertext{where we deduce from \eqref{eq:BlockApproximation} that $\RadNikD{\gamma_{\len}}{(\mu\otimes \nu)}=\frac{\gamma(Q_j^\len \otimes Q_k^\len)}{\mu(Q_j^\len)\nu(Q_k^\len)}$ on $Q_j^\len \otimes Q_k^\len$ and $\gamma_\len$ and $\gamma$ carry the same mass on $Q_j^\len \otimes Q_k^\len)$. We continue:}
	&=
	\sum_{\substack{
			(j, k) \in\left(\Z^{d}\right)^2: \\
			\mu(Q_{j}^{\len})>0,\, \nu(Q_{k}^{\len})>0
		}}
	\gamma(Q_j^\len \otimes Q_k^\len)\cdot \log(\gamma(Q_j^\len \otimes Q_k^\len))
	\nonumber
	\\
	&\qquad-
	\sum_{j\in \Z^d} \mu(Q_j^\len) \log(\mu(Q_j^\len))
	-
	\sum_{k\in \Z^d} \nu(Q_k^\len) \log(\nu(Q_k^\len))
	\label{eq:BlockApproxEntropyStep}
\end{align}
where we use the convention $0 \log 0 =0$.
The first term is less than or equal to zero. For the other two we observe that by compactness of $\Omega$ at most a finite number $N_\len \in \N$ of non-zero terms can appear in each sum with $N_\len \leq \tilde{C} \cdot \len^{-d}$ for some $\tilde{C}$ depending on $d$ and the diameter of $\Omega$.
Using Jensen's inequality and convexity of $\R_+ \ni s \mapsto \phi(s) \assign s \log(s)$ we find
\begin{align*}
	\sum_{j\in \Z^d} \phi(\mu(Q_j^\len)) \geq
	N_\len \cdot \phi\left(\tfrac{1}{N_\len} \sum_{j\in \Z^d} \mu(Q_j^\len)\right)
	= N_\len \phi(1/N_\len) = -\log(N_\len) \geq -\log(\tilde{C})+d \log(\len).
\end{align*}
Applying this bound on the second and third term in \eqref{eq:BlockApproxEntropyStep} we arrive at the desired entropy bound for $C \assign 2\log(\tilde{C})$.

\end{proof}

\end{document}